\documentclass[authoryear,preprint,12pt]{elsarticle}
\usepackage{lipsum}
\makeatletter
\def\ps@pprintTitle{%
 \let\@oddhead\@empty
 \let\@evenhead\@empty
 \def\@oddfoot{}%
 \let\@evenfoot\@oddfoot}
 \renewcommand{\MaketitleBox}{%
  \resetTitleCounters
  \def\baselinestretch{1}%
  \begin{center}
    \def\baselinestretch{1}%
    \Large \@title \par
    \vskip 18pt
    \normalsize\elsauthors \par
    \vskip 10pt
    \footnotesize \itshape \elsaddress \par
  \end{center}
  \vskip 12pt
}

\usepackage{amssymb}

\usepackage{amsthm}

\usepackage{amsfonts}
\usepackage{graphicx}
\usepackage{epstopdf}
\usepackage{algorithmic}
\ifpdf
  \DeclareGraphicsExtensions{.eps,.pdf,.png,.jpg}
\else
  \DeclareGraphicsExtensions{.eps}
\fi

\usepackage{amssymb}
\usepackage{mathtools}

\usepackage[mathscr]{eucal}
\usepackage{color}

\usepackage{tikz}
\usetikzlibrary{shapes,arrows}
\usetikzlibrary{decorations.pathreplacing,angles,quotes}

\usepackage{todonotes}
\usepackage{dsfont}
\usepackage{cleveref}
\usepackage{xcolor}
\usepackage{natbib}
\usepackage{verbatim}
\usepackage[normalem]{ulem}
\usepackage{geometry}
\usepackage{amsmath}
\usepackage{amssymb,mathtools}
\usepackage{dsfont}
\DeclarePairedDelimiter{\norm}{\lVert}{\rVert}
\usepackage{array}
\newcolumntype{C}{>{\centering\arraybackslash}p{0.3\linewidth}}
\usepackage{enumitem}
\newlist{steps}{enumerate}{1}
\setlist[steps, 1]{label = \underline{Step  \arabic*}}
\usepackage{CJKutf8}
\usepackage{enumerate}
\usepackage{adjustbox}
\usepackage{nccmath}

\newcommand{\mb}{\mathbb}

\newcommand{\mc}{\mathcal}
\newcommand{\br}{\breve}
\newcommand{\eps}{\epsilon}
\newcommand{\mbR}{{\mathds{R}}}
\newcommand{\Id}{{\mathbb{I}}}
\newcommand{\tra }{\intercal}

\newcommand{\mfN}{{\mathfrak{N}}}
\newcommand{\mfK}{{\mathfrak{K} }}
\newcommand{\mfT}{{\mathfrak{T} }}
\newcommand{\hQ}{{\hat{Q}}}

\newcommand{\nnb}{{\nonumber}}

\newtheorem{theorem}{Theorem}

\newtheorem{proposition}[theorem]{Proposition}

\newtheorem{definition}{Definition}
\newtheorem{assumption}{Assumption}

\tikzstyle{block1} = [rectangle, draw, thick,fill=blue!10, text width=4.5cm, text centered, rounded corners, minimum height=2cm]

\tikzstyle{block2} = [rectangle, draw, thick,fill=blue!10, text width=2cm, text centered, rounded corners, minimum height=2cm]

\tikzstyle{line} = [draw, -latex']

\begin{document}

\begin{frontmatter}



\title{
Risk-Sensitive Mean Field Games with Common Noise:  A Theoretical Study with Applications to Interbank Markets}
\author[1]{Xin Yue Ren}
\author[1]{Dena Firoozi}


\address[1]{Department of Decision Sciences, HEC Montréal, Montréal, QC \\
email: (xin-yue.ren@hec.ca, dena.firoozi@hec.ca)}

\begin{abstract}
In this paper, we address linear-quadratic-Gaussian (LQG) risk-sensitive mean field games (MFGs) with common noise. In this framework agents are exposed to a common noise and aim to minimize an exponential cost functional that reflects their risk sensitivity. We leverage the convex analysis method to derive the optimal strategies of agents in the limit as the number of agents goes to infinity. These strategies yield a Nash equilibrium for the limiting model. 

The model is then applied to interbank markets, focusing on optimizing lending and borrowing activities to assess systemic and individual bank risks when reserves drop below a critical threshold. We employ Fokker-Planck equations and the first hitting time method to formulate the overall probability of a bank or market default. We observe that the risk-averse behavior of agents reduces the probability of individual defaults and systemic risk, enhancing the resilience of the financial system. Adopting a similar approach based on stochastic Fokker-Planck equations, we further expand our analysis to investigate the conditional probabilities of individual default under specific trajectories of the common market shock. 
\end{abstract}



\begin{keyword}
Risk-sensitivity, mean-field games, common noise, exponential utility, LQG systems, interbank market, default probability, systemic risk, Fokker-Planck equation.



\end{keyword}

\end{frontmatter}

\section{Introduction}
\subsection{Literature Review} 
In stochastic games, multiple agents search for maximizing the profit or minimizing the cost while competing continuously with others. However, the complexity of the problem increases when the number of agents is large. In fact, each agent's stochastic optimal control problem becomes mathematically intractable in this case. As a solution for this issue, for such large-population games, mean field game (MFG) theory is used to approximate the solution to the high dimensional optimization problem each agent faces. MFG theory was introduced in a series of works \citep{ Lasry2006a, Lasry2006b, Huang2006, Lions2007, Huang2007} in the early 21st century. In MFG games, when there is an infinite number of agents, a Nash equilibrium is reached when one agent takes the best-response action to the environment where the mass behavior of others is modelled by the mean field distribution \citep{Huang2006}. \citet{Lions2007} and \citet{Huang2006} reduce the optimal control problem of a representative agent to a set of coupled forward-backward partial differential equations, where the forward component is the Fokker-Planck-Kolmogorov equation generating the mean field distribution and the backward component is the Hamilton-Jacobi-Bellman equation generating the value function of the agent. The authors subsequently discuss the existence and uniqueness of the solutions within this context. When it comes to a finite number of agents, an approximate Nash equilibrium, called $\epsilon$-Nash equilibrium, can be reached by employing the limiting strategies obtained. In other words, one agent can profit at most by $\epsilon$ by unilaterally deviating its strategy from others' \citep{Huang2006, Huang2007}.

Linear quadratic Gaussian (LQG) MFGs involve agents with linear dynamics in its own state and control action and the mean field as well as an additive noisemodelled by a Brownian motion. In addition, the cost functional to optimize is a quadratic function of these processes \citep{Huang2007}. For this type of MFGs, explicit solutions can be obtained which is very convenient in the context of applications.

In the classical setup of MFGs, there are a large number of agents where each has an asymptotically negligible influence on the system as the number of agents grows to infinity. However, in applications, there are usually a few agents which are not asymptotically negligible. \citet{Huang2010} considers LQG games with a major agent and a large number of minor agents to address such situations in practice. The behavior (dynamics and cost functionals) of individually negligible minor agents and the influential major agent contribute collectively to the mean field formation. \citet{Nourian2013} presents nonlinear MFGs with a major agent and a large number of minor agents. In this case, as the major agent's state or control action induces random fluctuation in the mean field distribution, the authors decompose the MFG problem into a non-standard stochastic optimal control problem with random coefficient for a representative minor agent and a stochastic coefficient McKean-Vlasov optimal control problem for the major agent. Other works in this area include \citep{Nguyen2012LinearQuadraticGaussianMixedGames,Carmona2016probabilisticapproachmean,Carmona2017AlternativeApproachMean,Carmona2016FiniteStateMean,sen_mean_2016,firoozi_epsilon-nash_2021,firoozi_class_2022,LASRY2018886,Bensoussan2017,MOON2018200, DenaConvex, HuangCIS2020, firoozi_lqg_2022}.


    In the context of applications, it is natural to consider a common random shock to all agents, especially when the game happens within the same environment for all agents.  A common Brownian motion may be added to the agent's dynamics to model such shocks. \citet{Carmona2014} presents the MFGs where the agents' dynamics include a common Brownian motion. The authors prove the existence of a weak MFG solution under general assumptions. With additional convexity assumptions, the existence of solutions is established without relaxed or externally randomized controls. Moreover, under the monotonicity condition of Lasry and Lions \citep{Lions2007}, the authors prove the existence and uniqueness of the solutions in a strong sense as the consequences of their pathwise uniqueness. \citet{Carmona2018II} develops the solutions for such games and extend the subject to the games with major and minor players as well as the games of timing. \citet{Caines2017} elaborates on two approaches to MFGs with common noise. The first one is an extension of the master equation formulation to the MFG problems. The second one is to treat the common noise as the dynamics of an uncontrolled major agent that embeds in each agent's dynamics.

When solving mean field games in the risk-neutral case, only the first moment of the integral cost is included in the cost functional of the agent. On the contrary, in the risk-sensitive case, an exponential function of integral cost is considered. In other words,  all moments of the integral cost, including the second moment which is a risk measure for the agent, are considered. Thus, the risk-sensitive behavior of the agent is captured \citep{Moon2017, Moon2019}.
\citet{Moon2017} solves a multi-agent linear-quadratic game with the exponential cost functional. The authors first solve a generic risk-sensitive optimal problem and then characterize an approximated mass behavior effect on one agent via the fixed-point analysis of the mean field system. They show that the approximated mass behavior is the best estimate of the actual one as the population size goes to infinity.
\citet{Caines2017} mentions the use of dynamic programming for such problems with exponential integral cost. In \citet{Moon2019} stochastic maximum principle is used to address nonlinear risk-sensitive mean field games. The authors analyze the optimal control problem under a fixed probability measure. Then, via Schauder's theorem they obtain the conditions under which a fixed-point solution exists to the consistency equation which equates the probability law of the optimally controlled state of the representative agent to the fixed measure. \citet{Tembine2012} elaborates on the fact that the mean field value derived using this method coincides with the value function from Hamilton-Jacobi-Bellman equation with an additional quadratic term under appropriate regularity conditions. 


MFGs have found diverse applications in mathematical finance. \citet{FirooziISDG2017, casgrain2018mean,Cardaliaguet2018,Huang2019} use MFGs in a dynamic trading context where the goal of each agent is to maximize the expected wealth and close the position at the end. The authors express the solution of the game explicitly in terms of a deterministic fixed point problem and discuss $\epsilon$-Nash equilibrium when considering a finite population. \citet{Carmona2017} incorporates the mean field game of timing into a bank runs' context. The authors model the value of the deposits with dividends of agents in relation to the moment at which the agents exit the game satisfactorily in continuous time. Through a probabilistic approach, the fixed point for best responses is found using the Nash equilibrium. \citet{CarmonaSysRisk2015,CarmonaSysRisk2018} uses LQG MFGs to model an inter-bank borrowing and lending system in which each agent's dynamic represents the log-monetary reserves of one bank. In addition, the authors investigate the individual and systemic risk by defining a default threshold for each bank. Then, the Nash equilibrium is established following the banks' optimization of monetary reserve adjustments. Other applications include equilibrium pricing in financial markets \citep{gomes2020mean, shrivats2020mean, feron2021price, fujii2020mean}, portfolio trading \citep{Horst2018}, financial market design \citep{shrivats2022mean}, and cryptocurrencies \citep{li2019mean}. Specifically, we 

\subsection{Problem Description}
This paper studies LQG MFGs with a common noise, where agents have an exponential cost to capture risk sensitivity. The methodology developed is mainly inspired by \citet{DenaConvex} and \citet{liu2023}, where the authors develop a convex analysis method to address, respectively, risk-neutral and risk-sensitive LQG optimal control problems and then extend the analysis to MFGs with major and minor agents. We leverage the convex analysis method to derive the best-response strategies of risk-sensitive agents when exposed to a common noise. These strategies yield a Nash equilibrium for the liming model when the number of agents goes to infinity.


Then, the model is used to investigate the impact of risk sensitivity on individual default and systemic default probabilities in the context of an interbank market where individual banks seek to pursue optimal strategies to reduce costs.  For this purpose, first, risk-sensitivity is incorporated in the market model introduced by \citet{Carmona2013} through an exponential cost functional. Then, to calculate the default probability of an individual agent and of the system, Fokker-Planck equations are formulated based on \citet{wille2004new} via the hitting time approach for diffusion processes. Then, the equations are numerically solved and the impact of various parameters, in particular risk-sensitivity, is examined. Finally, the default probabilities subject to specific trajectories of the common shock in the market are examined through a stochastic Fokker-Planck equation, drawing inspiration from \citet{Carmona2013}. The conditional default probability is numerically computed and the impact of distinct trajectories of the common noises is investigated.

The contributions of the paper are summarized as follows:
\begin{itemize}

    \item The paper uses convex analysis to address linear-quadratic-Gaussian (LQG) risk-sensitive mean field games (MFGs) with common noise. More specifically, this model introduces exponential cost and a common Brownian motion, shedding light on risk-sensitive behavior in the context of MFGs, where all agents are influenced by a shared noise. Optimal feedback control actions for agents leading to a Nash equilibrium are derived.
    
    \item Within the context of interbank transactions, the paper makes contributions by (i) introducing risk sensitivity, (ii) utilizing the Fokker-Planck equation to derive the total probabilities of individual default as well as systemic default, and (iii) investigating the conditional probability of individual default under specific trajectories of the common shocks using stochastic Fokker-Planck equation. 
    
\end{itemize}

\subsection{Paper Organization}
The paper is organized as follows. Firstly, a model of Linear-Quadratic-Gaussian (LQG) risk-sensitive mean field games is presented \ref{sec: Finite Population Model}, which incorporates a common Brownian motion with exponential cost. Next, optimal feedback control actions for agents leading to a Nash equilibrium are derived for the infinite-population scenario, where the number of agents goes to infinity in Section in  \ref{sec: Infinite Population Model}. Then, the paper demonstrates the application of this model in an interbank market in Section \ref{sec: Application}. 
In Section \ref{sec: Individual Default and Systemic Risk} and Section \ref{numerical methods}, the probability of default for the agent and the system is investigated by deriving corresponding Fokker-Planck equations and numerically solving them. The impact of various parameters, in particular risk-sensitivity on these probabilities is examined in Section \ref{SIDP}. A more thorough investigation is conducted to study the conditional probability of individual default using numerical methods over the stochastic Fokker-Planck equation, focusing on specific trajectories of common shock in Section \ref{CDF}. 


\section{Finite-Population Model}\label{sec: Finite Population Model}
In this section we present a general model for linear-quadratic-Gaussian (LQG) risk-sensitive mean field game (MFG) systems with a finite number of agents impacted by a common noise.

Consider a system consisting of $N$ competitive dynamic agents. We assume that agents are heterogeneous and belong to $K$ distinct types, where each type is characterized by a specific set of model parameters. The index set of agents is defined by $\mfN = \{1,2,\dots,N\}$. Moreover, 
the index set $\mc{I}_k$ of type $k, k \in \mfK=\{1,2,\dots,K\},$ is defined as 
\[\mc{I}_k = \lbrace i : \theta_i = \theta^{(k)},~ i \in \mfN \rbrace,\] 
where $\theta_i$ and $\theta^{(k)}\in\Theta$ denote, respectively, the model parameters of agent $i$ and type $k$, with $\Theta$ being the system parameter set. The cardinality $|\mc{I}_k|$ of the index set ${I}_k$ determines the number of agents in type $k$, denoted by $N_k$ henceforth. The proportion of the agents that belong to type $ k \in \mfK$, is defined by $\pi_k^{[N]}= \frac{N_k}{N}$. Thus, for the entire population, we obtain the vector of proportions $ \pi^{[N]} = [\pi_1^{[N]} \hspace{0.2cm}  \pi_2^{[N]} \hspace{0.2cm}...\hspace{0.2cm} \pi_K^{[N]}]$, which represents the empirical distribution of system parameters. 

\subsection{Dynamics}
Agent $i, i \in \mfN$, is governed by linear dynamics given by 
\begin{equation}\label{finite_dx_t}
    dx_t^i=\left(A_kx_t^i+F_k x_t^{[N]}+B_ku_t^i+b_k(t)\right)dt+\sigma_kdw_t^i+\sigma_0dw_t^0
\end{equation}
where $t \in \mfT \coloneqq [0, T]$ and $k\in \mfK$. We denote respectively $x_t^i \in \mathbb{R}^n$ and $u_t^i \in \mathbb{R}^m$ as the state and the control action of agent $i$ at time $t$. We define $ w \coloneqq \{(w^0_t)_{t\in \mfT}, (w^i_t)_{t\in \mfT}, i\in \mfN\}$ as a set of $(N+1)$ independent Brownian motions, where $ w^i_t \in \mathbb{R}^r$ denotes the idiosyncratic noise of agent $i$ at time $t$ and $w^0_t \in \mathbb{R}^r$ denotes a common noise that impacts the dynamics of all $N$ agents at time $t$. The latter models a random shock in the system. The coefficients $A_k \in \mathbb{R}^{n \times n}$, $F_k \in \mathbb{R}^{n \times n}$, $B_k \in \mathbb{R}^{n \times m}$, $\sigma_k, \sigma_0 \in \mathbb{R}^{n \times r}$, and the function $b_k(t) \in \mathbb{R}^n$ are deterministic and known. 

Moreover, $x_t^{[N]}\in\mathbb{R}^n$ defines the average state of the entire population of agents at time $t$ and is given by
\begin{equation}
    x_t^{[N]}=\frac{1}{N}\sum_{i=1}^N x_t^i.
\end{equation}
From \eqref{finite_dx_t}, each agent in the system is impacted by the average state of the entire population. 

\subsection{Filtration and Control $\sigma$-Fields}\label{Filtration and Control}
Let $(\Omega, \pmb{F}, (\mathcal{F}^{[N]}_t)_{t\in \mfT}, \mathbb{P})$ be a filtered probability space, where $\Omega$ is the sample space, $\pmb{F}$ is a $\sigma$-algebra, $(\mathcal{F}^{[N]}_t)_{t\in \mfT}$ is a filtration, and $\mathbb{P}$ is a probability measure. We define the $\sigma$-algebra $\mathcal{F}_t^{[N]}  \coloneqq \sigma(w^0_s, w^i_s, 0 \leq s \leq t, i \in \mfN)$. The admissible set of controls $\mc{U}^i$ of an agent $i$ is the set of continuous linear state feedback $\mbR^m$-valued control laws $u_t^i = u(t,x_t^i), {t\in\mfT},$ that are $\mc{F}_t^{[N]}$-adapted such that $\mb{E}[\int_0^T (u_t^i)^\tra u_t^i dt] < \infty$, for $T < \infty$.

\begin{assumption}\label{$x_0$ finite population}
    The initial states $\{x_0^i, {i \in \mfN}\}$, defined on $(\Omega, \pmb{F}, (\mathcal{F}_t^{[N]})_{t\in \mfT}, \mathbb{P})$, are identically distributed,  mutually independent and also independent of $w$.
\end{assumption}

\subsection{Cost Functional}
Let $\mathbb{S}^{n \times m}$ denote the set of symmetric matrices of dimension $n \times m$, and let $\Vert a\Vert_B^2=a^\top Ba$ denote the seminorm of vector $a$ with respect to $B \geq 0$. Additionally, we define $u^{-i}\coloneqq (u^0,\dots,u^{i-1},u^{i+1},\dots,u^N)$ to represent the control actions performed by agents other than agent $i$.

The cost functional of agent $i$ to be minimized is given by 
\begin{equation}\label{cost integral_f}
  J^{i,[N]}(u^i, u^{-i}) = \gamma_k\log\mb{E} \left[\exp\left(\frac{1}{\gamma_k}\left(g^k(x_T^i,x_T^{[N]})+\int_0^Tf^k(x^i,x_t^{[N]},u_t^i)dt\right)\right)\right]
\end{equation} where 
\begin{equation}
    g^k(x_T^i,x_T^{[N]})=\frac{1}{2}\Vert x_T^i-H_kx_T^{[N]}-\eta_k\Vert_{\hQ_k}^2
\end{equation}
\begin{equation}
    f^k(x^i_t,x_t^{[N]},u_t^i)=\frac{1}{2} \left\{ \Vert
    x_t^i-H_kx_t^{[N]}-\eta_k\Vert_{Q_k}^2+2(x_t^i-H_kx_t^{[N]}-\eta_k)^\tra  S_k u_t^i+\Vert u_t^i\Vert_{R_k}^2 \right\}
\end{equation}
with $\tfrac{1}{\gamma_k} \in (0, \infty)$ indicating the degree of risk aversion of the agent. In particular, as $\tfrac{1}{\gamma}$ increases, the agent's risk aversion intensifies. In the limit when $\tfrac{1}{\gamma} \rightarrow 0$, the cost functional reduces into a risk-neutral form. The other parameters are $\hQ_k, Q_k \in \mathbb{S}^{n \times n}$, $R_k \in \mathbb{S}^{m \times m}$,  $H_k \in \mathbb{R}^{n \times n}$, $\eta_k \in \mathbb{R}^{n}$, $S_k \in \mathbb{R}^{n \times m}$ for all $ k \in \mfK$. 

The cost functional is defined as the expected value of an exponential function of the integral cost, enabling it to capture all moments of the integral cost, including those that indicate risk. As a result, the cost functional incorporates risk, making it a risk-sensitive cost. 

For a representative agent, the optimization problem involves finding the optimal control $u_t^i$ that minimizes the cost functional while taking into account the agent's dynamics and its interactions with all other agents modeled by the average state. However, as the number of agents $N$ increases, the complexity of this problem escalates, rendering it intractable. Mean Field Game (MFG) theory provides a mathematically tractable approach to analyze such interactions among a large number of agents. The MFG methodology involves finding the solution to the asymptotic game as the number of agents approaches infinity. In this limiting case, the average state of the population, known as the mean field, can be mathematically characterized. As each agent can compute the mean field, the problem becomes significantly simplified and can be represented as a set of individual optimal control problems linked together through the mean field. In the next section, we present the optimization problem in the limiting case referred to as the infinite-population model.

\section{Infinite-Population Model}\label{sec: Infinite Population Model}
In this section we present the infinite-population model, as $N \rightarrow \infty$, for the linear-quadratic-Gaussian (LQG) risk-sensitive mean field games described in the preceding section. The model consists of an infinite number of competitive dynamic agents that belong to $K < \infty$ distinct types, each with a unique set of model parameters. Stated differently, we are considering the limiting case where each type is comprised of an infinite population. The index set of agents is denoted by $\mfN ^ \infty =\{1,2,\dots\}$.
\begin{assumption}\label{Porportion of the MF}
    The empirical distribution of model parameters converges to a theoretical distribution. In other words, there exists $\pi_k $ such that $\lim_{N \rightarrow \infty} \pi_k^{[N]}\coloneqq\lim_{N \rightarrow \infty} \frac{N_k}{N}=\pi_k$ for all $k  \in \mfK$. Thus, $\lim_{N \rightarrow \infty}\pi^{[N]}  =\pi$, where $\pi = [\pi_1, \dots,\pi_K ]$.
\end{assumption}
\subsection{Dynamics}\label{MF dyn}
From the dynamics \eqref{finite_dx_t}, we consider the limit case of the empirical average for an infinite population case and acknowledge the convergence criterion imposed in Assumption \ref{Porportion of the MF}. Then, agent $i, i \in \mfN^\infty$, in the infinite-population limit is governed by linear dynamics given by 
\begin{equation}\label{dx_t}
    dx_t^i=\left(A_kx_t^i+F_k^\pi\bar{x}_t+B_ku_t^i+b_k(t)\right)dt+\sigma_kdw_t^i+\sigma_0dw_t^0
\end{equation}
where $F_k^\pi  \in \mathbb{R}^{n \times Kn}$ and $\bar{x}_t \in \mathbb{R}^{Kn}$. We define $F_k^\pi=\pi \otimes F_k \coloneqq [F_k\pi_1 \hspace{0.2cm}  F_k\pi_2 \hspace{0.2cm}...\hspace{0.2cm} F_k\pi_K]$. In LQG case, the mean field can be written as $\bar{x}_t^\tra =\begin{bmatrix} (\bar{x}_t^1)^\tra & \dots & (\bar{x}_t^K)^\tra \end{bmatrix}$ which denotes the population mean field at time $t$, where $\bar{x}_t^k \in \mathbb{R}^n$ is defined as 
  \begin{align}\label{MFk def}
 \bar{x}_t^k =\lim_{N_k\rightarrow \infty} \frac{1}{N_k}\sum_{i \in \mc{I}_k} x_t^i 
 \end{align}
representing the mean field of type $k$ at time $t$. The mean field dynamics is derived in Section \ref{sub: Mean Field Dynamics}. All other continuous states and coefficients maintain their definitions from the finite population model. The assumption on the starting states remains also the same but in the filtered probability space defined in Section \ref{Probability Space and Control sigma-Fields}.

\subsection{Filtration} \label{Probability Space and Control sigma-Fields}

 We define the filtration for agent $i$ as $(\mathcal{F}_t^{i})_{t\in \mfT}  \coloneqq \sigma({w^0_s, w^i_s, 0 \leq s \leq t})$ for all $i \in \mfN^\infty$ and the filtration for the mean field as $(\mathcal{F}_t^{0})_{t\in \mfT}  \coloneqq \sigma({w^0_s, 0 \leq s \leq t})$. The admissible set of controls $\mc{U}^i$ for an agent $i$ is the set of continuous linear state feedback control laws $u_t^i = u(t,x_t^i), {t\in\mfT},$ that are $\mc{F}_t^{i}$-adapted $\mbR^m$-valued processes such that $\mb{E}[\int_0^T (u_t^i)^\tra u_t^i dt] < \infty$, for $T < \infty$.

\subsection{Cost Functional}
 The cost functional to be minimized is given by
\begin{equation}\label{cost integral_inf}
  J^{i,\infty}(u^i) = \gamma_k\log\mb{E} \left[\exp\left(\frac{1}{\gamma_k}\left(g^k(x_T^i,\bar{x}_T)+\int_0^Tf^k(x^i,\bar{x}_t,u_t^i)dt\right)\right)\right]
\end{equation} where 
\begin{equation}
    g^k(x_T^i,\bar{x}_T)=\frac{1}{2}\Vert x_T^i-H_k^{\pi}\bar{x}_T-\eta_k\Vert_{\hQ_k}^2
\end{equation}
\begin{equation}\label{cost_comp_inf}
    f^k(x^i_t,\bar{x}_t,u_t^i)=\frac{1}{2} \left\{ \Vert
    x_t^i-H_k^{\pi}\bar{x}_t-\eta_k\Vert_{Q_k}^2+2(x_t^i-H_k^{\pi}\bar{x}_t-\eta_k)^\tra  S_k u_t^i+\Vert u_t^i\Vert_{R_k}^2 \right\}
\end{equation}
with $ H_k^\pi \in \mathbb{R}^{n \times Kn}$ defined as $ H_k^\pi =\pi \otimes H_k = [H_k\pi_1 \hspace{0.2cm}  H_k\pi_2 \hspace{0.2cm}...\hspace{0.2cm} H_k\pi_K]$. The other parameters are the same as the ones in the finite-population model.  
\begin{assumption}\label{ass:ConvCond}
$ \hat{Q}_k\geq0,\, R_k>0,\, Q_k-S_kR_k^{-1}S_k^\tra \geq 0 $.
\end{assumption}
Assumption \ref{ass:ConvCond} ensures the convexity of the cost functional \eqref{cost integral_inf} with respect to $x_t^i$ and $u_t^i$. By completing the square, we obtain the following equality
\begin{align}
    f^k(x^i_t,\bar{x}_t,u_t^i) &=\frac{1}{2} \left\{ \Vert
    x_t^i-H_k^{\pi}\bar{x}_t-\eta_k\Vert_{Q_k}^2+2(x_t^i-H_k^{\pi}\bar{x}_t-\eta_k)^\tra  S_k u_t^i+\Vert u_t^i\Vert_{R_k}^2 \right\}\nnb\\&=\frac{1}{2} \bigl\{ \Vert
    x_t^i-H_k^{\pi}\bar{x}_t-\eta_k\Vert_{Q_k}^2+2(x_t^i-H_k^{\pi}\bar{x}_t-\eta_k)^\tra  S_k u_t^i-\Vert S_k^\tra (x_t^i-H_k^{\pi}\bar{x}_t-\eta_k)\Vert_{R_k^{-1}} \nnb\\&\text{\hspace{4mm}} +\Vert S_k^\tra (x_t^i-H_k^{\pi}\bar{x}_t-\eta_k)\Vert_{R_k^{-1}}+\Vert u_t^i\Vert_{R_k}^2 \bigr\}\nnb
    \\&=\frac{1}{2} \bigl\{ \Vert
    x_t^i-H_k^{\pi}\bar{x}_t-\eta_k\Vert_{Q_k-S_k R_k^{-1}S_k^\tra}^2 \nnb\\&\text{\hspace{4mm}} +\Vert u_t^i+R_k^{-1} S_k(x_t^i-H_k^{\pi}\bar{x}_t-\eta_k)\Vert_{R_k}^2 \bigr\}.
\end{align}
Consider $g^k(x_T^i,\bar{x}_T)$ and $f^k(x^i_t,\bar{x}_t,u_t^i)$, we refer to \citet{Jacobson1973} for the conditions in Assumption \ref{ass:ConvCond} that guarantee the convexity of the cost functional \eqref{cost integral_inf} with respect of $x_t^i$ and $u_t^i$.
\\ \\
Coefficients $A_k, F_k^\pi, B_k$ and $\sigma_k$ in the agent's dynamics can be viewed as type-specific factors with respect to the associated variable. The function $b_k(t)$ is an additional deterministic function with the dynamics' drift. The factor $\sigma_0$ is a multiplier to the common noise presented in the environment in which all agents inhabit. There are numerous potential financial applications linked to these variables. For instance, the state $x_t^i$ can be interpreted as the portfolio value, market price of inventory, or monetary reserves of a fund. The corresponding feedback control $u_t^i$ can be regarded as the trade or transaction rate.

The cost functional that the agent wants to minimize can be viewed as a regulator's imposition or the agent's preference or cost. In this model, parts of the cost functional include the distance of the agent's state to a factor of the mean field up to a constant $\eta_k$. From equation \eqref{cost integral_inf}, the impact of the agent's control action on the cost functional is also present.

A thorough interpretation of the parameters will be presented in application sections within the interbank context.

\subsection{Mean Field Dynamics}\label{sub: Mean Field Dynamics}
The mean field dynamics for the agents of the type $k$ is derived from the definition provided in equation \eqref{MFk def}. An equivalent representation in the infinite-population limit can be written in the conditional expectation form $\bar{x}_t^k = \mb{E}[x_t^{i,k} | \mathcal{F}_t^{0}]$, where $x^{i,k}_t$ represents the state of a representative agent of the type $k$ \citep{Carmona2018II}. The mean field dynamics is then derived as
\begin{align}\label{dx_t^k}
    d\bar{x}_t^k =\left(A_k\bar{x}_t^k +F_k^\pi\bar{x}_t+B_k\bar{u}_t^k+b_k(t)\right)dt+\sigma_0dw_t^0
\end{align}
where $\bar{u}_t^k \in \mathbb{R}^m$ is defined by  \[\bar{u}_t^k = \lim_{N_k\rightarrow \infty} \frac{1}{N_k}\sum_{i \in I_k} u_t^i.\] 
If the limit exists, $\bar{u}_t^k$ represents the control mean field of agents of type $k \in \mc{K}$.\\
Note that as $N_k$ increases to infinity for all types of agent, by the strong Law of Large Numbers,
\begin{equation}
   \lim_{N_k\rightarrow \infty}\frac{1}{N_k}\sum_{i \in \mc{I}_k} \int dw_t^i=0
\end{equation}
or equivalently $\mb{E}[\int dw_t^i | \mathcal{F}_t^{0}] = 0$.

Subsequently, the state mean field of the population $\bar{x}_t \in \mathbb{R}^{Kn}$ can be represented as the vector $\bar{x}_t^\tra =\begin{bmatrix}(\bar{x}_t^1)^\tra & \dots & (\bar{x}_t^K)^\tra   
\end{bmatrix}$ satisfying
\begin{equation}\label{MF_dyn_int}
d\bar{x}_t = (\br{A}\bar{x}_t  + \br{B} \bar{u}_t  + \br{m}_t)dt + \pmb{1}_{Kn \times n} \sigma_0 dw^0_t
\end{equation}
where $\bar{u}_t \in \mathbb{R}^{Km}$ represents the population control mean field  $ \bar{u}_t^\tra = \begin{bmatrix}
(\bar{u}_t^1)^\tra &
\dots &
(\bar{u}_t^K)^\tra
\end{bmatrix}$.
The associated coefficients $\br{A}_t\in \mathbb{R}^{Kn \times Kn}, \br{B} \in \mathbb{R}^{Kn \times Km}$, $\br{m}_t \in \mathbb{R}^{Kn \times 1}$, and $\pmb{1}_{Kn \times n} \in \mathbb{R}^{Kn \times n}$ are defined as in 
\begin{equation}\label{MFmatrices}
\br{A} = \begin{bmatrix}
A_1\pmb{e}_1+ F^{\pi}_1\\
\vdots \\
A_K\pmb{e}_K+F^{\pi}_K
\end{bmatrix}, \quad
\br{B} = \begin{bmatrix}
B_1 & &0\\
        & \ddots &\\
       0 & & B_K
\end{bmatrix}, \quad
\br{m}_t = \begin{bmatrix}
b_1(t)\\
\vdots\\
b_K(t)
\end{bmatrix}, \quad
\pmb{1}_{Kn \times n}=\begin{bmatrix}
\Id_n\\
\vdots \\
 \Id_n
 \end{bmatrix}.
\end{equation}
 Moreover, the matrix $\pmb{e}_k \in \mathbb{R}^{n \times Kn}$ is defined as $\pmb{e}_k = [0_{n \times n}, ..., 0_{n \times n}, \Id_n, 0_{n \times n}, ..., 0_{n \times n}]$, which has the $n \times n$ identity matrix $\Id_n$ at the $k$th block. 
 
\section{Solutions to the Infinite-Population Model}\label{sec: sols to inf}
\subsection{Optimal Control Action}\label{sec: Optimal Control Action}
Consider the infinite-population LQG risk-sensitive MFG model with common noise presented in Section \ref{sec: Infinite Population Model}, our objective is to determine the optimal control actions that achieve the best response using convex analysis. To implement this approach, we adapt the definition of the Gâteaux derivative described in \citet{ConvecAnalysisBook1999} and \citet{Allaire2007} to our specific problem. By using this modified definition, we can identify the control action that leads to the vanishing of the Gâteaux derivative of the cost function. Then, given the exponential nature of the cost integral, we use completion of squares and Girsanov's theorem to change the measure and determine the optimal control action.

\begin{definition}[G\^ateaux Derivative] \label{ddef:GateauxDerivative}
The cost functional $J^{i, \infty}$ defined on a neighbourhood of $u^i \in \mc{U}^i$ with values in $\mathbb{R}$ is Gâteaux differentiable at $u^i$ in the direction of $\omega^i \in \mc{U}^i$ if there exists a Gâteaux differential $ DJ(u^i)$ such that
\begin{equation}\label{def-Gateaux}
    \langle DJ(u^i), \omega^i \rangle = \lim_{\epsilon \to 0} \frac{J(u^i+\epsilon\omega^i)-J(u^i)}{\epsilon}.
\end{equation}
\end{definition}

\begin{theorem}[Gâteaux Derivative Expanded] \label{Theorem of Gâteaux Derivative Expanded}
The Gâteaux derivative of the cost functional \eqref{cost integral_inf} in the infinite population case is given by
\begin{equation}\label{Gâteaux derivative}
    \langle DJ^{i, \infty}(u^i), \omega^i \rangle =
    \frac{\mb{E}\left[\int_0^T (\omega_t^i)^\tra h^k(\epsilon,x^i_t,\bar{x}_t,u_t^i)dt\right]}{\mb{E}\left[\exp(G_T^i(u))\right]}
\end{equation}
where
\begin{align}
    &G_T^i(u)=\frac{1}{\gamma_k}\left[g^k(x_T^i,\bar{x}_T)+\int_0^T f^k(x^i_t,\bar{x}_t,u_t^i)dt\right]\\
    &h^k(\epsilon,x^i_t,\bar{x}_t,u_t^i)=M^i_{1,t}\Big(S_k^\tra (x_t^i-H_k^{\pi}\bar{x}_t-\eta_k)+R_k u_t^i- B_k^\tra \int_0^t \exp{\left(A_k^\tra(s-t)\right)} \nnb\\&\text{\hspace{3cm}}\times\bigl( Q_k (x_s^i-H_k^{\pi}\bar{x}_s-\eta_k)+S_k u_s^i \bigr) ds\Bigr)+B_k^\tra \exp{\left(-A_k^\tra t\right)} M_{2,t}^i \\
    &M_{1,t}^i = \mb{E}\left[ \exp{(G^i_T(u))}|\mathcal{F}_t^i\right],\allowdisplaybreaks\\
    &M_{2,t}^i = \mb{E} \Biggl[\exp{(G^i_T(u))}\biggl(\exp{\left(A_k^\tra T\right)} \hQ_k (x_T^i-H_k^{\pi}\bar{x}_T-\eta_k)+\int_0^T  \exp{\left(A_k^\tra s\right)}  \nnb\\&\text{\hspace{3cm}} \times \Bigr(Q_k (x_s^i-H_k^{\pi}\bar{x}_s-\eta_k)+S_k u_s^i \Bigl)ds  \biggr)\bigg| \mathcal{F}_t^i\Biggr].
\end{align}
\end{theorem}
\begin{proof}
    To compute the Gâteaux derivative, we start by deriving the agent's state as the solution to the  stochastic differential equation (SDE) given by \eqref{dx_t}. We perturb the control action of the representative agent $i$ and analyze the impact of this perturbation on the agent's state, the mean field, and the cost functional. Finally, we use Definition \ref{ddef:GateauxDerivative} to derive the Gâteaux derivative of the agent's cost functional. This approach allows us to effectively capture the impact of a small perturbation on the agent's overall performance and on the entire system.
    
    Consider the transformation $y_t=\exp{(-A_k t)}x_t^i $. Using Itô's lemma we can show that $y_t$ satisfies  
    \begin{align}\label{dy_t perturbed}
        dy_t&=-A_k\exp{\left(-A_kt\right)} x_t^i dt+\exp{(-A_k t)}\bigl([A_kx_t^i+F_k^\pi\bar{x}_t+B_k u_t^i+b_k(t)]dt\nnb\\&\text{\hspace{3mm}}+\sigma_kdw_t^i+\sigma_0dw_t^0\bigr).
    \end{align}
    Integrating both sides of \eqref{dy_t perturbed} from $0$ to $t$ and then multiplying by $\exp{(A_kt)}$, we get 
    \begin{align}\label{x-state}
        x_t^i&=\exp{(A_kt)}x_0+\int_0^t \exp{\left(A_k(t-s)\right)}(F_k^\pi\bar{x}_s +B_ku_s^i+b_k(s))ds +\int_0^t\exp{\left(A_k(t-s)\right)}\sigma_kdw_s^i\nnb\\&\text{\hspace{3mm}}+\int_0^t \exp{\left(A_k(t-s)\right)}\sigma_0dw_s^0.
    \end{align}
    
    Let $x_t^{i,\epsilon}$ denote the solution to \eqref{dx_t} subject to a perturbed control action $u_t^i+\epsilon\omega_t^i$ in the direction of $\omega_t^i \in \mathcal{U}^i$. From \eqref{x-state}, we can write
    \begin{equation}\label{x_per}
        x_t^{i,\epsilon}=x_t^i+\epsilon\int_0^t \exp{\left(A_k(t-s)\right)}B_k\omega_s^ids.
    \end{equation}
    Subsequently, the infinitesimal variation of $x_t^{i,\epsilon}$ is given by 
    \begin{equation}\label{dx_per}
        dx_t^{i,\epsilon}=dx_t^i+\epsilon B_k \omega_t^i dt+\epsilon A_k\int_0^t \exp{\left(A_k(t-s)\right)}B_k\omega_s^i ds. 
    \end{equation}
    
    On the one hand, we observe that the perturbed mean field $\bar{x}_t^{k,\epsilon}$ because of the perturbed control action of agent $i$ in type $k$, if the limit exists, is defined by 
    \begin{equation}
        \bar{x}_t^{k,\epsilon}=\lim_{N_k \rightarrow \infty}\frac{1}{N_k}\left(\sum_{j \in \mc{I}_k ,j\neq i} x_t^j+x_t^{i,\epsilon}\right).\label{MF_type_pert}
    \end{equation}
    On the other hand, the mean field of other agents belonging to the other types is not perturbed.
    Thus, we note that the population mean field $\bar{x}_t^\epsilon$, if the limits exist, is defined by 
    \begin{equation}
        \bar{x}_t^\epsilon =\begin{bmatrix}(\bar{x}_t^1)^\tra & \dots & (\bar{x}_t^{k,\epsilon})^\tra & \dots & (\bar{x}_t^K)^\tra   
    \end{bmatrix}.
    \end{equation}
    From \eqref{MF_type_pert}, for the infinite-population model, the impact of the perturbed control action of agent $i$ on the population mean field is negligible. Hence, we conclude that $\bar{x}_t^\epsilon = \bar{x}_t$.
    
    The cost of the perturbed control action $u_t^i+\epsilon\omega_t^i$ and the corresponding perturbed state $x^{i,\epsilon}_t$ is given by
    \begin{equation}
       J^{i,\infty}(u^i+\epsilon\omega^i) = \gamma_k \log\mb{E}\left[\exp\left(\frac{1}{\gamma_k}\left(g^k(x_T^{i,\epsilon},\bar{x}_T)+\int_0^T f^k(x^{i,\epsilon},\bar{x}_t,u_t^i+\epsilon \omega_t^i)dt\right)\right)\right].
    \end{equation} 
    To simplify the notation, we can define 
    \begin{equation}
        G_T^i(u)\coloneqq \frac{1}{\gamma_k}\bigl(g^k(x_T^i,\bar{x}_T)+\int_0^Tf^k(x^i_t,\bar{x}_t,u_t^i)dt\bigr).
    \end{equation}
  Let $\Phi_t=H_k^{\pi}\bar{x}_t+\eta_k$. From \eqref{x_per}, we can write the perturbed integral cost as 
    \begin{equation}
         J^{i,\infty}(u^i+\epsilon\omega^i) = \gamma_k \log\mb{E}\left[ \exp (G^{i,\eps}_T) \right]
    \end{equation}
    where
    \begin{align}
         G^{i,\eps}_T & = \frac{1}{\gamma_k}\bigl(g^k(x_T^{i,\epsilon},\bar{x}_T)+\int_0^T f^k(x^{i,\epsilon}_t,\bar{x}_t,u_t^i+\epsilon\omega^i_t)dt\bigr)  \nnb\\& = G^{i}_T+\frac{1}{2 \gamma_k} \Vert \epsilon\int_0^T \exp{\left(A_k(T-s)\right)}B_k\omega^i_s 
        ds\Vert_{\hQ_k}^2+\frac{1}{ \gamma_k}(x_T^i-\Phi_T)^\tra \hQ_k\epsilon\nnb\\&\text{\hspace{3mm}}\times\int_0^T \exp{\left(A_k(T-s)\right)} B_k\omega^i_sds+\frac{1}{ \gamma_k}
        \int_0^T  \Bigl\{ \frac{1}{2}\Vert \epsilon\int_0^t \exp{\left(A_k(t-s)\right)}B_k\omega^i_sds\Vert_{Q_k}^2\nnb\\&\text{\hspace{3mm}}+(x_t^i-\Phi_t)^\tra  \bigl(Q_k\epsilon\int_0^t \exp{\left(A_k(t-s)\right)}B_k\omega^i_sds+ S_k (\epsilon\omega^i_t)\bigr)+(\epsilon\int_0^t \exp{\left(A_k(t-s)\right)}\nnb\\&\text{\hspace{3mm}}\times B_k\omega^i_sds)^\tra S_k (u_t^i+\epsilon\omega^i_t)+\frac{1}{2}\Vert\epsilon\omega^i_t\Vert_{R_k}^2+(u_t^i)^\tra R_k\epsilon\omega^i_t\Bigr\}dt.
    \end{align}
    By reordering the variables, we obtain
    \begin{align}
        G^{i,\eps}_T &= G^{i}_T+\frac{\epsilon}{\gamma_k}(x_T^i-\Phi_T)^\tra \hQ_k\int_0^T \exp{\left(A_k(T-s)\right)} B_k\omega^i_sds+\frac{\epsilon}{ \gamma_k}
        \int_0^T  \Bigl\{(x_t^i-\Phi_t)^\tra  \nnb\\&\text{\hspace{3mm}}\times\bigl(Q_k\int_0^t \exp{\left(A_k(t-s)\right)}B_k\omega^i_sds+ S_k \omega^i_t\bigr)+(\int_0^t \exp{\left(A_k(t-s)\right)}\nnb\\&\text{\hspace{3mm}}\times B_k\omega^i_sds)^\tra S_k u_t^i+(u_t^i)^\tra R_k\omega^i_t\Bigr\}dt+\frac{\epsilon^2}{2 \gamma_k} \Vert \int_0^T \exp{\left(A_k(T-s)\right)}B_k\omega^i_s      ds\Vert_{\hQ_k}^2\nnb\\&\text{\hspace{3mm}}+\frac{\epsilon^2}{ \gamma_k}
        \int_0^T  \Bigl\{ \frac{1}{2}\Vert \int_0^t \exp{\left(A_k(t-s)\right)}B_k\omega^i_sds\Vert_{Q_k}^2\nnb\\&\text{\hspace{3mm}}+(\int_0^t \exp{\left(A_k(t-s)\right)} B_k\omega^i_sds)^\tra S_k \omega^i_t+\frac{1}{2}\Vert\omega^i_t\Vert_{R_k}^2\Bigr\}dt.
    \end{align}

    According to Definition \ref{ddef:GateauxDerivative}, for the representative agent-$i$ the Gâteaux derivative is given as 
    \begin{align}
        &\langle DJ^{i, \infty}(u^i), \omega^i \rangle = \lim_{\epsilon \to 0} \frac{\gamma_k}{\epsilon} \log \frac{\mb{E}\left[\exp(G^{i,\eps}_T)\right]}{\mb{E}\left[\exp(G^{i}_T)\right]}.
    \end{align}
     As the limit involves an indeterminate quotient, we can employ L'Hôpital's rule while applying Talor expansion on $\exp(G^{i,\eps}_T)$ to continue the analysis as in
    \begin{align}
        &\langle DJ^{i, \infty}(u^i), \omega^i \rangle =\lim_{\epsilon \to 0} \gamma_k\frac{1}{\mb{E}[\exp(G^{i,\eps}_T)]}\frac{\partial }{\partial \epsilon}\mb{E}\Biggl[ \exp(G^{i}_T)\biggl(1+\frac{\epsilon}{\gamma_k}(x_T^i-\Phi_T)^\tra \hQ_k\nnb\\&\text{\hspace{2.9cm}}\times\int_0^T \exp{\left(A_k(T-s)\right)} B_k\omega^i_sds+\frac{\epsilon}{ \gamma_k}
        \int_0^T  \Bigl\{(x_t^i-\Phi_t)^\tra  \bigl(Q_k \nnb\\&\text{\hspace{2.9cm}}\times \int_0^t \exp{\left(A_k(t-s)\right)} B_k\omega^i_sds+ S_k \omega^i_t\bigr)+(\int_0^t \exp{\left(A_k(t-s)\right)} B_k\omega^i_sds)^\tra \nnb\\&\text{\hspace{2.9cm}} \times S_k u_t^i+(u_t^i)^\tra R_k\omega^i_t\Bigr\}dt + \mathcal{O}(\epsilon^2)\biggr)\Biggr].
    \end{align}
    By linearity of the expectation, we have
    \begin{align}
        &\langle DJ^{i, \infty}(u^i), \omega^i \rangle =\lim_{\epsilon \to 0} \gamma_k\frac{1}{\mb{E}[\exp(G^{i,\eps}_T)]}\frac{\partial }{\partial \epsilon}\Biggl[\mb{E}(\exp(G^{i}_T)+\epsilon\mb{E}\biggl(\exp(G^{i}_T)\Bigl(\frac{1}{\gamma_k}(x_T^i-\Phi_T)^\tra \hQ_k\nnb\\&\text{\hspace{2.9cm}}\times\int_0^T \exp{\left(A_k(T-s)\right)} B_k\omega^i_sds+\frac{1}{ \gamma_k}
        \int_0^T  \Bigl\{(x_t^i-\Phi_t)^\tra  \bigl(Q_k\int_0^t \exp{\left(A_k(t-s)\right)}\nnb\\&\text{\hspace{2.9cm}} \times B_k\omega^i_sds+ S_k \omega^i_t\bigr)+(\int_0^t \exp{\left(A_k(t-s)\right)} B_k\omega^i_sds)^\tra S_k u_t^i+(u_t^i)^\tra R_k\omega^i_t\Bigr\}dt\Bigr)\biggr) \nnb\\&\text{\hspace{2.9cm}}+ \epsilon^2 \mb{E}\left(\frac{\exp(G^{i}_T)}{\epsilon^2}\mathcal{O}(\epsilon^2)\right)\Biggr].
    \end{align}
    Then, we can perform the derivative and obtain
    \begin{align}
        &\langle DJ^{i, \infty}(u^i), \omega^i \rangle =\lim_{\epsilon \to 0} \gamma_k\frac{1}{\mb{E}[\exp(G^{i,\eps}_T)]}\Biggl[\mb{E}\biggl(\exp(G^{i}_T)\Bigl(\frac{1}{\gamma_k}(x_T^i-\Phi_T)^\tra \hQ_k\nnb\\&\text{\hspace{2.9cm}}\times\int_0^T \exp{\left(A_k(T-s)\right)} B_k\omega^i_sds+\frac{1}{ \gamma_k}
        \int_0^T  \Bigl\{(x_t^i-\Phi_t)^\tra  \bigl(Q_k \nnb\\&\text{\hspace{2.9cm}}\times \int_0^t \exp{\left(A_k(t-s)\right)}B_k\omega^i_sds+ S_k \omega^i_t\bigr)+(\int_0^t \exp{\left(A_k(t-s)\right)} B_k\omega^i_sds)^\tra \nnb\\&\text{\hspace{2.9cm}}\times S_k u_t^i+(u_t^i)^\tra R_k\omega^i_t\Bigr\}dt\Bigr)\biggr) + 2\epsilon \mb{E}\left(\frac{\exp(G^{i}_T)}{\epsilon^2}\mathcal{O}(\epsilon^2)\right)\nnb\\&\text{\hspace{2.9cm}}+{\epsilon}^2\frac{\partial }{\partial \epsilon}\mb{E}\left(\frac{\exp(G^{i}_T)}{\epsilon^2}\mathcal{O}(\epsilon^2)\right)\Biggr].
    \end{align}
    By performing the limit and simplifying the equation, we obtain
    \begin{align}
        &\langle DJ^{i, \infty}(u^i), \omega^i \rangle =\frac{1}{\mb{E}[\exp(G^{i}_T)]}\Biggl[\mb{E}\biggl(\exp(G^{i}_T)\Bigl((x_T^i-\Phi_T)^\tra \hQ_k\nnb\\&\text{\hspace{2.9cm}}\times\int_0^T \exp{\left(A_k(T-s)\right)} B_k\omega^i_sds+
        \int_0^T  \Bigl\{(x_t^i-\Phi_t)^\tra  \bigl(Q_k\int_0^t \exp{\left(A_k(t-s)\right)}\nnb\\&\text{\hspace{2.9cm}}\times B_k\omega^i_sds+ S_k \omega^i_t\bigr)+(\int_0^t \exp{\left(A_k(t-s)\right)} B_k\omega^i_sds)^\tra\nnb\\&\text{\hspace{2.9cm}}\times S_k u_t^i+(u_t^i)^\tra R_k\omega^i_t\Bigr\}dt\Bigr)\biggr)\Biggr].
    \end{align}
    For clarity, we can transpose and manipulate the order of some matrix multiplications to get 
    \begin{align}\label{GD before Fubini}
        \langle DJ^{i, \infty}(u^i), \omega^i \rangle &= \frac{1}{\mb{E}[\exp(G_T^i(u))]}\mb{E}\Biggl[\exp(G_T^i(u))\biggl(\int_0^T (\omega^i_s)^\tra B_k^\tra  \exp{\left(A_k^\tra(T-s)\right)} ds&\nnb\\&\text{\hspace{3mm}}\times \hQ_k (x_T^i-\Phi_T)+\int_0^T\Bigl( (\omega^i_t)^\tra S_k^\tra (x_t^i-\Phi_t)+ (\omega^i_t)^\tra R_k u_t^i\Bigr)dt&\nnb\\&\text{\hspace{3mm}}+\int_0^T\Bigl(\int_0^t (\omega^i_s)^\tra B_k^\tra  \exp{\left(A_k^\tra(T-s)\right)}  ds Q_k (x_t^i-\Phi_t)&\nnb\allowdisplaybreaks\\&\text{\hspace{3mm}}+\int_0^t  (\omega^i_s)^\tra B_k^\tra  \exp{\left(A_k^\tra(t-s)\right)} ds S_k u_t^i\Bigr)dt
        \biggr)  \Biggr].
    \end{align}
    As the function within the integral is continuous, by Fubini's theorem and the change of order of integrals \citep{CalculusStrang}, the last term in the above equation can be written as
    \begin{multline}\label{GD after Fubini}
        \int_0^T\left(\int_0^t (\omega^i_s)^\tra  B_k^\tra  \exp{\left(A_k^\tra(t-s)\right)}  ds Q_k (x_t^i-\Phi_t)+\int_0^t (\omega^i_s)^\tra B_k^\tra  \exp{\left(A_k^\tra(t-s)\right)} ds S_k u_t^i\right)dt\\=\int_0^T (\omega^i_s)^\tra \left(\int_s^T B_k^\tra  \exp{\left(A_k^\tra(t-s)\right)}  Q_k (x_t^i-\Phi_t) dt+\int_s^T B_k^\tra  \exp{\left(A_k^\tra(t-s)\right)} S_k u_t^idt \right)ds.
    \end{multline}
    From \eqref{GD before Fubini} and \eqref{GD after Fubini}, we can then change the integration variable for the second integral and factor out $(\omega_s^i)^\tra$  and substitute \eqref{GD after Fubini} to get
    \begin{align}\label{GD w factor out}
        \langle DJ^{i, \infty}(u^i), \omega^i \rangle &= \frac{1}{\mb{E}[\exp(G_T^i(u))]}\mb{E}\Biggl[\exp(G_T^i(u))\biggl(\int_0^T (\omega^i_s)^\tra \Bigl\{ S_k^\tra (x_s^i-\Phi_s)+R_k u_s^i&\nnb\\&\text{\hspace{3mm}} + B_k^\tra \Bigl[ \exp{\left(A_k^\tra(T-s)\right)} \hQ_k (x_T^i-\Phi_T)+\int_s^T \Bigl(  \exp{\left(A_k^\tra(T-s)\right)}  Q_k (x_t^i-\Phi_t)&\nnb\\&\text{\hspace{3mm}}+\exp{\left(A_k^\tra(T-s)\right)} S_k u_t^i\Bigr)dt \Bigr] \Bigr\} ds \biggr) \Biggr].
    \end{align}
   The inner integral within \eqref{GD w factor out} can be split in two terms as in  
    \begin{align}
        &\int_s^T \left( \exp{\left(A_k^\tra(t-s)\right)}  Q_k (x_t^i-\Phi_t)+  \exp{\left(A_k^\tra(t-s)\right)} S_k u_t^i\right)dt  \nnb \\&= \int_0^T \exp{\left(A_k^\tra(t-s)\right)}  \left( Q_k (x_t^i-\Phi_t)+ S_k u_t^i \right) dt- \int_0^s \exp{\left(A_k^\tra(t-s)\right)}  \left( Q_k (x_t^i-\Phi_t)+S_k u_t^i \right) dt.
    \end{align}
    Thus, an equivalent expression for the Gâteaux derivative is given by
    \begin{flalign}\label{GD before tower}
        \langle DJ^{i, \infty}(u^i), \omega^i \rangle &= \frac{1}{\mb{E}[\exp(G_T^i(u))]}\mb{E}\Biggl[\exp(G_T^i(u))\biggl(\int_0^T (\omega^i_s)^\tra \biggl\{S_k^\tra (x_s^i-\Phi_s)+R_k u_s^i\nnb\\&\text{\hspace{3mm}}+B_k^\tra \Bigl[ - \int_0^s \exp{\left(A_k^\tra(t-s)\right)}  \left( Q_k (x_t^i-\Phi_t)+S_k u_t^i \right) dt \nnb\\&\text{\hspace{3mm}}+\exp{\left(-A_k^\tra s\right)} \Bigl( \exp{\left(A_k^\tra T\right)} \hQ_k (x_T^i-\Phi_T)\nnb\\&\text{\hspace{3mm}}+\int_0^T \exp{\left(A_k^\tra t\right)}  \left( Q_k (x_t^i-\Phi_t)+ S_k u_t^i \right) dt\Bigr) \Bigr] \biggr\} ds \biggr)\Biggr].
    \end{flalign}
    By taking $\exp \left(G^i_T(u)\right)$ inside the integral in \eqref{GD before tower} and applying the tower rule based on the filtration $\mathcal{F}_s^i$, the Gâteaux derivative then can be written  as 
    \begin{align}
        &\langle DJ^{i, \infty}(u^i), \omega^i \rangle =
        \frac{1}{\mb{E}\left[\exp\left(G_T^i(u)\right)\right]}\mb{E}\Biggl[\int_0^T (\omega^i_s)^\tra \Biggl\{M^i_{1,s}\Big(S_k^\tra (x_s^i-\Phi_s)+R_k u_s^i\nnb\\&\text{\hspace{2.9cm}}+ B_k^\tra \Bigl[ -\int_0^s \exp{\left(A_k^\tra(t-s)\right)} \left( Q_k (x_t^i-\Phi_t)+S_k u_t^i \right) dt\Big)\nnb\\&\text{\hspace{2.9cm}}+\exp{\left(-A_k^\tra s\right)} M_{2,s}^i \Bigr] \Biggr\} ds\Biggr]
    \end{align}
    where
    \begin{align}
        M_{1,s}^i &= \mb{E}\left[ \exp \left(G^i_T(u)\right)|\mathcal{F}_s^i\right],\\
        M_{2,s}^i &= \mb{E} \Biggl[\exp\left(G^i_T(u)\right)\biggl(\exp{\left(A_k^\tra T\right)} \hQ_k (x_T^i-\Phi_T)+\int_0^T  \exp{\left(A_k^\tra t\right)} \nnb \\&\text{\hspace{3mm}} \times \bigl( Q_k (x_t^i-\Phi_t) +S_k u_t^i \bigr) dt  \biggr)\bigg| \mathcal{F}_s^i\Biggr].
    \end{align}
\end{proof}
Using the Gâteaux derivative given in Theorem \ref{Theorem of Gâteaux Derivative Expanded}, we can determine the optimal action $u_t^{i,*}$ that minimizes the cost functional \eqref{cost integral_inf} of the representative agent. A necessary condition for $u_t^{i,*} \in \mc{U}^i$ to be the optimal control action under $\mathbb{P}$ is 
\begin{equation}
    \langle DJ^{i, \infty}(u^{i,*}), \omega^i \rangle = 0 \hspace{20 mm} \forall w \in\mc{U}^i.
\end{equation}
If Assumption \ref{ass:ConvCond} holds, this condition is also a sufficient optimality condition for $u_t^{i,*}$. Hence the optimal control action under $\mathbb{P}$ is given by

\begin{align}\label{optimal-u_risk}
    \medmath{u_t^{i,*} = -R_k^{-1}\Biggl(B_k^\tra  \exp{\left(-A_k^\tra t\right)}  \left[\frac{ M_{2,t}^i}{M_{1,t}^i}-\int_0^t \left(\exp{\left(A_k^\tra s\right)}  \left( Q_k (x_s^i-\Phi_s)+ S_k u_s^{i,*} \right) \right)ds\right]+S_k^\tra (x_t^i-\Phi_t)\Biggr)}
\end{align}
where 
\begin{equation}\label{M2/M1}
    \frac{M_{2,t}^i}{ M_{1,t}^i} = \tfrac{\mb{E}^\mathbb{P} \left[\exp{(G^i_T(u))}\left(\exp{\left(A_k^\tra T\right)} \hQ_k (x_T^i-\Phi_T)+\int_0^T  \exp{\left(A_k^\tra s\right)}  \left(Q_k (x_s^i-\Phi_s)+S_k u_s^{i,*} \right)ds  \right)\bigg| \mathcal{F}_t^i\right]}{\mb{E}^\mathbb{P}\left[ \exp{(G^i_T(u))}|\mathcal{F}_t^i\right]}.
\end{equation}

We observe that in its current form, the optimal control action is not practicable in the context of applications. In particular, we are interested in a linear state feedback form for the optimal control action as it is very convenient when it comes to implementing the optimal strategy. However, due to the nonlinearity introduced by the term \eqref{M2/M1}  in the optimal control action, it is not obvious how such a linear form can be achieved at first glance. By inspecting \eqref{M2/M1} we observe that both the numerator and the denominator are involved with the exponential term $\exp{(G^i_T(u))}$. This fact suggests that a linear form for the optimal control action may be achievable through a change of measure. To investigate this matter, the initial step involves determining whether or not $\exp{(G^i_T(u))}$ may represent a Radon Nikodym derivative. If such a representation is possible, we can transform \eqref{M2/M1} from a quotient of martingales under the measure $\mathbb{P}$ to a martingale under a new measure denoted by $\mathbb{\hat{P}}$. The subsequent step involves identifying the optimal control under $\hat{\mathbb{P}}$, followed by applying the equivalent measure theorem to recover the optimal control under $\mathbb{P}$.

\subsection{Change of Measure}\label{sub: Change of Measure}
This section focuses on the derivation of the Radon-Nikodym exponent, which is needed to transform \eqref{M2/M1} into a martingale under a new probability measure, denoted by $\mathbb{\hat{P}}$. To achieve this, we adopt a strategy of selecting a set of control coefficients. With the help of a judiciously chosen variable and its cumulative change with respect to its infinitesimal difference, $G_T^i(u)$ can be reduced to the desired form. The inspiration behind the introduced change of measure stems from the financial realm, where we evaluate derivatives under the risk-neutral probability to simplify complex pricing calculations. Specifically, we transition from the physical world to the risk-neutral setting by quantifying the risk premium through the Radon-Nikodym exponent.

\begin{theorem}\label{Girsanov exponent}
Consider the LQG risk-sensitive system described by \eqref{dx_t}, \eqref{cost integral_inf}-\eqref{cost_comp_inf},  \eqref{MF_dyn_int}-\eqref{MFmatrices} and suppose that Assumption \ref{ass:ConvCond} holds. The variable $G_T^i(u)- \Theta_0$ admits the representation
\begin{equation}
    G_T^i(u)- \Theta_0=- \frac{1}{2\gamma_k^2}
    \int_0^T \norm{\pmb{\mu(t,\pmb{W}_t)}}^2dt +\frac{1}{\gamma_k}\int_0^T \pmb{\mu(t,\pmb{W}_t)}d\pmb{W}_t
\end{equation}
with $\Theta_0 \in \mathbb{R},\pmb{\mu(t,\pmb{W}_t)}=((\pmb{X}_t)^\tra\pmb{H}_t^k+(\pmb{C}_t^k)^\tra)\pmb{\Sigma}^k$ such that 
\begin{equation}
\pmb{W}_t 
=\begin{bmatrix} w_t^i \\ w_t^0\end{bmatrix},\\
\quad
\pmb{X}_t = \begin{bmatrix} x_t^i\\\bar{x}_t\end{bmatrix},
\quad
\pmb{H}_t^k = \begin{bmatrix}
\Pi_t^k & &\Lambda_t^k\\
       (\Lambda_t^k)^\tra & & \Delta_t^k
\end{bmatrix},\\
\end{equation}
\begin{equation}
\pmb{C}_t^k = \begin{bmatrix}
\Upsilon_t^k \\ \Gamma_t^k
\end{bmatrix},\\
\quad
\pmb{\Sigma}^k=
\begin{bmatrix}\sigma_k && \sigma_0\\
               \pmb{0}_{Kn \times r} && \pmb{1}_{Kn \times n}\sigma_0 
\end{bmatrix}
\end{equation}
with $\Pi_t^k \in \mathbb{S}^{n \times n}$, $ \Delta_t^k \in \mathbb{S}^{Kn \times Kn}$, $\Lambda_t^k \in \mathbb{R}^{n \times Kn} , \Upsilon_t^k \in \mathbb{R}^n$, $\Gamma_t^k \in \mathbb{R}^{Kn}$, if the following condition is satisfied 

\begin{align}\label{Condition change of measure}
    \zeta(u)=&\int_0^T \Bigl(\frac{1}{2\gamma_k}\pmb{X}_t^\tra \pmb{Q_k}\pmb{X}_t+\frac{1}{\gamma_k}\pmb{\eta_k}\pmb{X}_t+\frac{1}{2\gamma_k}\eta^\tra_k Q_k \eta_k +\frac{1}{\gamma_k}\pmb{X}_t^\tra\pmb{S}_k u_t^{i,*} -\frac{1}{\gamma_k}\eta^\tra_k S_k u_t^{i,*} \nnb\\&\text{\hspace{3mm}} + \frac{1}{2\gamma_k}(u_t^{i,*})^\tra R_k u_t^{i,*} +\frac{1}{\gamma_k}((\pmb{X}_t)^\tra\pmb{H}_t^k+(\pmb{C}_t^k)^\tra)\{\tilde{\pmb{A}}_k\pmb{X}_t+\tilde{\pmb{B}}_k u_t^{i,*}+\tilde{\pmb{\beta_k}} \bar{u}_t+\tilde{\pmb{M}}_t\}\Bigr) dt \nnb\\&\text{\hspace{3mm}} +\frac{1}{2\gamma_k}\int_0^T tr\bigl(\sigma_k^\tra \Pi_t^k \sigma_k)+\sigma_0^\tra\left(\Pi_t^k+2\pmb{1}_{n \times Kn}(\Lambda_t^k)^\tra+\pmb{1}_{n \times Kn}\Delta_t^k\pmb{1}_{Kn \times n}\right)\sigma_0\bigr)dt\nnb\\&\text{\hspace{3mm}}+\frac{1}{2\gamma_k}\int_0^T\pmb{X}_t^\tra d\pmb{H}_t^k \pmb{X}_t+\frac{1}{\gamma_k}\int_0^Td(\pmb{C}_t^k)^\tra\pmb{X}_t+\frac{1}{2 \gamma_k}\int_0^Td\Psi_t^k\nnb\\&\text{\hspace{3mm}} + \frac{1}{2\gamma_k^2}\int_0^T\norm{\pmb{\mu(t,\pmb{W}_t)}}^2 dt = 0 
 \end{align}
with $\Theta_0,\Psi_t^k \in \mathbb{R}$,
\begin{equation}
\quad
\pmb{Q_k} = \begin{bmatrix}
Q_k & & -Q_kH_k^\pi\\
       -(H_k^\pi)^\tra Q_k & & (H_k^\pi)^\tra Q_kH_k^\pi
\end{bmatrix},
\quad
\pmb{\eta_k} = \begin{bmatrix}
-\eta^\tra_k Q_k && \eta^\tra_k Q_k H_k^\pi
\end{bmatrix},
\end{equation}
\begin{equation}
\pmb{S}_k = \begin{bmatrix}
S_k \\ -(H_k^\pi)^\tra S_k
\end{bmatrix},
\quad
\tilde{\pmb{A}}_k
=\begin{bmatrix}A_k && F_k^\pi\\
                \pmb{0}_{Kn \times Kn} && \br{A}
\end{bmatrix},
\quad
\tilde{\pmb{B}}_k
=\begin{bmatrix} B_k \\ \pmb{0}_{Kn \times m} \end{bmatrix},
\end{equation}
\begin{equation}
\tilde{\pmb{\beta_k}}
=\begin{bmatrix} \pmb{0}_{n \times Km} \\ \br{B} \end{bmatrix},
\quad
\tilde{\pmb{M}}_t =\begin{bmatrix} b_k(t)\\ \br{m}_t \end{bmatrix},
\quad
\pmb{\Sigma}^k=
\begin{bmatrix}\sigma_k && \sigma_0\\
               \pmb{0}_{Kn \times r} && \pmb{1}_{Kn \times n}\sigma_0 
\end{bmatrix}.
\end{equation}
Moreover, there exists a probability measure $\mathbb{\hat{P}}$ characterized by the Radon Nikodym variable $\frac{d\mathbb{\hat{P}}}{d\mathbb{P}}=\exp{\left(G_T^i(u) - \Theta_0\right)}$.
\end{theorem}
\begin{proof}
For the sake of clarity and organization, we will employ matrix notation instead of more cumbersome scalar notation. For this purpose, we consider
\begin{equation}
\pmb{X}_t = \begin{bmatrix} x_t^i\\\bar{x}_t\end{bmatrix},
\quad
\pmb{H}_t^k = \begin{bmatrix}
\Pi_t^k & &\Lambda_t^k\\
       (\Lambda_t^k)^\tra & & \Delta_t^k
\end{bmatrix},
\quad
\pmb{C}_t^k = \begin{bmatrix}
\Upsilon_t^k \\ \Gamma_t^k
\end{bmatrix}
\end{equation}
where $\Pi_t^k \in \mathbb{S}^{n \times n}$, $ \Delta_t^k \in \mathbb{S}^{Kn \times Kn}$, $\Lambda_t^k \in \mathbb{R}^{n \times Kn} , \Upsilon_t^k \in \mathbb{R}^n$, $\Gamma_t^k \in \mathbb{R}^{Kn}$. For our purpose, motivated by \citet{Duncan2013}, we define the expression
\begin{equation}\label{auxiliary}
    \Theta_t^k=\frac{1}{2\gamma_k}\pmb{X}_t^\tra \pmb{H}_t^k \pmb{X}_t+\frac{1}{\gamma_k}(\pmb{C}_t^k)^\tra\pmb{X}_t+\frac{1}{2 \gamma_k}\Psi_t^k
\end{equation}
where $\Psi_t^k \in \mathbb{R}$, $\pmb{H}_t^k,\pmb{C}_t^k,\Psi_t^k$ are deterministic. We have
\begin{equation}\label{integral-eq-theta}
    \int_0^T d\Theta_t=\Theta_T - \Theta_0.
\end{equation}
Then we apply Itô's lemma to obtain the infinitesimal variations of $\Theta^k_t$ as in 
\begin{align}\label{dtheta_eq}
    \int_0^T d\Theta_t=&\int_0^T \Biggl\{\frac{1}{2\gamma_k}\pmb{X}_t^\tra d\pmb{H}_t^k \pmb{X}_t+\frac{1}{\gamma_k}\pmb{X}_t^\tra \pmb{H}_t^k d\pmb{X}_t+\frac{1}{2\gamma_k}d\left<\pmb{X}^\tra \pmb{H}_t^k \pmb{X}\right>_t+\frac{1}{\gamma_k}d(\pmb{C}_t^k)^\tra\pmb{X}_t\nnb\\&+\frac{1}{\gamma_k}(\pmb{C}_t^k)^\tra d \pmb{X}_t+\frac{1}{2 \gamma_k}d\Psi_t^k \Biggr\}.
\end{align}
By substituting \eqref{dtheta_eq} in \eqref{integral-eq-theta} and taking all the terms to one side we have 
\begin{align}
    0=&-(\Theta_T - \Theta_0)\nnb
    \\&+\int_0^T \Bigg\{\frac{1}{2\gamma_k}\pmb{X}_t^\tra d\pmb{H}_t^k \pmb{X}_t+\frac{1}{\gamma_k}\pmb{X}_t^\tra \pmb{H}_t^k d\pmb{X}_t+\frac{1}{2\gamma_k}d\left<\pmb{X}^\tra \pmb{H}_t^k \pmb{X}\right>_t+\frac{1}{\gamma_k}d(\pmb{C}_t^k)^\tra\pmb{X}_t\nnb\\&+\frac{1}{\gamma_k}(\pmb{C}_t^k)^\tra d \pmb{X}_t+\frac{1}{2 \gamma_k}d\Psi_t^k \Bigg\}\label{Itô's lemma for change of measure}
\end{align}
where
\begin{equation}
    d\pmb{X}_t = \{\tilde{\pmb{A}}_k\pmb{X}_t+\tilde{\pmb{B}}_k u_t^{i,*}+\tilde{\pmb{\beta_k}} \bar{u}_t^*+\tilde{\pmb{M}}_t\}dt+ \pmb{\Sigma}^k d\pmb{W}_t
\end{equation}
with 
\begin{equation}
\tilde{\pmb{A}}_k
=\begin{bmatrix}A_k && F_k^\pi\\
                \pmb{0}_{Kn \times Kn} && \br{A}
\end{bmatrix},
\quad
\tilde{\pmb{B}}_k
=\begin{bmatrix} B_k \\ \pmb{0}_{Kn \times m} \end{bmatrix},
\quad
\tilde{\pmb{\beta_k}}
=\begin{bmatrix} \pmb{0}_{n \times Km} \\ \br{B} \end{bmatrix},
\quad
\tilde{\pmb{M}}_t =\begin{bmatrix} b_k(t)\\ \br{m}_t \end{bmatrix}
\end{equation}
\begin{equation}
\pmb{\Sigma}^k=
\begin{bmatrix}\sigma_k && \sigma_0\\
               \pmb{0}_{Kn \times r} && \pmb{1}_{Kn \times n}\sigma_0 
\end{bmatrix},
\quad
\pmb{W}_t 
=\begin{bmatrix} w_t^i \\ w_t^0\end{bmatrix}.
\end{equation}
For the sake of clarity, we also write $G^i_T(u)$ in terms of $\pmb{X}$,
\begin{align}\label{G in matrices}
    G_T^i(u)&=\frac{1}{\gamma_k}\biggl[\frac{1}{2}\Vert x_T^i-H_k^{\pi}\bar{x}_T-\eta_k\Vert_{\hQ_k}^2\nnb\\&\text{\hspace{3mm}}+\int_0^T \frac{1}{2} \biggl\{ \Vert
    x_t^i-H_k^{\pi}\bar{x}_t-\eta_k\Vert_{Q_k}^2+2(x_t^i-H_k^{\pi}\bar{x}_t-\eta_k)^\tra  S_k u_t^{i,*}+\Vert u_t^{i,*}\Vert_{R_k}^2 \biggr\}dt\biggr]\nnb\\
    &=\frac{1}{2\gamma_k}\pmb{X}_T^\tra \pmb{\hat{Q}_k}\pmb{X}_T+\frac{1}{\gamma_k}\pmb{\hat{\eta}_k}\pmb{X}_T+\frac{1}{2\gamma_k}\eta^\tra _k\hat{Q}_k \eta_k+\int_0^T \biggl\{\frac{1}{2\gamma_k}\pmb{X}_t^\tra \pmb{Q_k}\pmb{X}_t+\frac{1}{\gamma_k}\pmb{\eta_k}\pmb{X}_t \nnb\\&\text{\hspace{3mm}}+\frac{1}{2\gamma_k}\eta^\tra_k Q_k \eta_k+\frac{1}{\gamma_k}\pmb{X}_t^\tra\pmb{S}_k u_t^{i,*} -\frac{1}{\gamma_k}\eta^\tra_k S_k u_t^{i,*} + \frac{1}{2\gamma_k}(u_t^{i,*})^\tra R_k u_t^{i,*} \biggr\}dt
\end{align}
where
\begin{equation}
\quad
\pmb{\hat{Q}_k} = \begin{bmatrix}
\hat{Q}_k & & -\hat{Q}_k H_k^\pi\\
       -(H_k^\pi)^\tra \hat{Q}_k & & (H_k^\pi)^\tra \hat{Q}_kH_k^\pi
\end{bmatrix},
\quad
\pmb{\hat{\eta}_k} = \begin{bmatrix}
-\eta^\tra_k\hat{Q}_k && \eta^\tra_k\hat{Q}_kH_k^\pi
\end{bmatrix}
\quad
\end{equation}
\begin{equation}
\quad
\pmb{Q_k} = \begin{bmatrix}
Q_k & & -Q_kH_k^\pi\\
       -(H_k^\pi)^\tra Q_k & & (H_k^\pi)^\tra Q_k H_k^\pi
\end{bmatrix},
\pmb{\eta_k} = \begin{bmatrix}
-\eta^\tra_k Q_k && \eta^\tra_k Q_k H_k^\pi
\end{bmatrix},
\pmb{S}_k = \begin{bmatrix}
S_k \\ -(H_k^\pi)^\tra S_k
\end{bmatrix}.
\end{equation}
Then, we add together both sides of \eqref{Itô's lemma for change of measure} and \eqref{G in matrices} to get 
\begin{align}
    G_T^i(u)&=\frac{1}{2\gamma_k}\pmb{X}_T^\tra \pmb{\hat{Q}_k}\pmb{X}_T+\frac{1}{\gamma_k}\pmb{\hat{\eta}_k}\pmb{X}_T+\frac{1}{2\gamma_k}\eta^\tra_k \hat{Q}_k \eta_k -\Theta_T + \Theta_0+\int_0^T \biggl\{\frac{1}{2\gamma_k}\pmb{X}_t^\tra \pmb{Q_k}\pmb{X}_t\nnb\\&\text{\hspace{3mm}}+\frac{1}{\gamma_k}\pmb{\eta_k}\pmb{X}_t+\frac{1}{2\gamma_k}\eta^\tra_k Q_k \eta_k +\frac{1}{\gamma_k}\pmb{X}_t^\tra\pmb{S}_k u_t^{i,*} -\frac{1}{\gamma_k}\eta^\tra_k S_k u_t^{i,*} + \frac{1}{2\gamma_k}(u_t^{i,*})^\tra R_k u_t^{i,*} \biggr\}dt\nnb\\&\text{\hspace{3mm}}+\int_0^T \frac{1}{\gamma_k}\pmb{X}_t^\tra \pmb{H}_t^k d\pmb{X}_t+\int_0^T \frac{1}{\gamma_k}(\pmb{C}_t^k)^\tra d \pmb{X}_t+\int_0^T \frac{1}{2\gamma_k}d\left<\pmb{X}^\tra \pmb{H}_t^k \pmb{X}\right>_t\nnb\\&\text{\hspace{3mm}}+\int_0^T \frac{1}{2\gamma_k}\pmb{X}_t^\tra d\pmb{H}_t^k \pmb{X}_t+\int_0^T \frac{1}{\gamma_k}d(\pmb{C}_t^k)^\tra\pmb{X}_t+\int_0^T \frac{1}{2 \gamma_k}d\Psi_t^k.\label{exponent Girsanov eq}
\end{align}

The idea is to reduce \eqref{exponent Girsanov eq} to the form 
\begin{equation}\label{Girsanov eq}
    G_T^i(u)- \Theta_0=- \frac{1}{2\gamma_k^2}
    \int_0^T \norm{\pmb{\mu(t,\pmb{W}_t)}}^2dt +\frac{1}{\gamma_k}\int_0^T \pmb{\mu(t,\pmb{W}_t)}d\pmb{W}_t.
\end{equation}
From \eqref{exponent Girsanov eq}, set the terminal conditions $\frac{1}{2\gamma_k}\pmb{X}_T^\tra \pmb{\hat{Q}_k}\pmb{X}_T+\frac{1}{\gamma_k}\pmb{\hat{\eta}_k}\pmb{X}_T+\frac{1}{2\gamma_k}\eta^\tra_k \hat{Q}_k \eta_k=\Theta_T$, then we can consider $\pmb{\mu(t,\pmb{W}_t)}=(\pmb{X}_t^\tra\pmb{H}_t^k+(\pmb{C}_t^k)^\tra)\pmb{\Sigma}^k$ which belongs to the space of adapted stochastic processes $(\Omega, \pmb{F}, (\mc{F}^i_t)_{t\in \mfT}, \mathbb{P})$, especially to the space of square-integrable functions defined on the interval $\mfT$. Next, we add and subtract the following formula to \eqref{exponent Girsanov eq}
\begin{equation}
    \frac{1}{2\gamma_k^2}\int_0^T\norm{\pmb{\mu(t,\pmb{W}_t)}}^2dt =\frac{1}{2\gamma_k^2}\int_0^T tr\left((\pmb{\Sigma}^k)^\tra\left(\pmb{H}_t^k\pmb{X}_t+\pmb{C}_t^k\right)\left((\pmb{X}_t)^\tra\pmb{H}_t^k+(\pmb{C}_t^k)^\tra\right)\pmb{\Sigma}^k\right)dt.
\end{equation}
Additionally, from \eqref{dx_t} and \eqref{MF_dyn_int}, we can further expand the quadratic variation term
\begin{align}
    d\left<\pmb{X}^\tra \pmb{H}_t^k \pmb{X}\right>_t= tr\left(\sigma_k^\tra \Pi_t^k \sigma_k+\sigma_0^\tra\left(\Pi_t^k+2\pmb{1}_{n \times Kn}(\Lambda_t^k)^\tra+\pmb{1}_{n \times Kn}\Delta_t^k\pmb{1}_{Kn \times n}\right)\sigma_0\right)dt. 
\end{align}
Then, \eqref{exponent Girsanov eq} may be represented as

\begin{align}
    G_T^i(u)=&\Theta_0+\zeta(u) - \frac{1}{2\gamma_k^2}\int_0^T\norm{\pmb{\mu(t,\pmb{W}_t)}}^2 dt+\frac{1}{\gamma_k}\int_0^T\pmb{\mu(t,\pmb{W}_t)} d\pmb{W}_t.
\end{align}
where 
\begin{align}
    \zeta(u)=&\int_0^T \Bigl(\frac{1}{2\gamma_k}\pmb{X}_t^\tra \pmb{Q_k}\pmb{X}_t+\frac{1}{\gamma_k}\pmb{\eta_k}\pmb{X}_t+\frac{1}{2\gamma_k}\eta^\tra_k Q_k \eta_k +\frac{1}{\gamma_k}\pmb{X}_t^\tra\pmb{S}_k u_t^{i,*} -\frac{1}{\gamma_k}\eta^\tra_k S_k u_t^{i,*}  \nnb\\&+ \frac{1}{2\gamma_k}(u_t^{i,*})^\tra R_k u_t^{i,*}+\frac{1}{\gamma_k}((\pmb{X}_t)^\tra\pmb{H}_t^k+(\pmb{C}_t^k)^\tra)\{\tilde{\pmb{A}}_k\pmb{X}_t+\tilde{\pmb{B}}_k u_t^{i,*}+\tilde{\pmb{\beta_k}} \bar{u}_t^*+\tilde{\pmb{M}}_t\}\Bigr) dt\nnb\\&+ \frac{1}{2\gamma_k^2}\int_0^T\norm{\pmb{\mu(t,\pmb{W}_t)}}^2 dt+\frac{1}{2\gamma_k}\int_0^T\pmb{X}_t^\tra d\pmb{H}_t^k \pmb{X}_t+\frac{1}{\gamma_k}\int_0^Td(\pmb{C}_t^k)^\tra\pmb{X}_t+\frac{1}{2 \gamma_k}\int_0^Td\Psi_t^k \nnb\\&+ \frac{1}{2\gamma_k}\int_0^Ttr\bigl(\sigma_k^\tra \Pi_t^k \sigma_k+\sigma_0^\tra \bigl(\Pi_t^k+2\pmb{1}_{n \times Kn}(\Lambda_t^k)^\tra +\pmb{1}_{n \times Kn}\Delta_t^k\pmb{1}_{Kn \times n}\bigr)\sigma_0\bigr)dt.
\end{align}
Finally, we obtain the following desired form \eqref{Girsanov eq} for the change of measure
\begin{align}
    G_T^i(u)- \Theta_0=\zeta(u)- \frac{1}{2\gamma_k^2}\int_0^T\norm{\pmb{\mu(t,\pmb{W}_t)}}^2 dt +\frac{1}{\gamma_k}\int_0^T\pmb{\mu(t,\pmb{W}_t)} d\pmb{W}_t
\end{align}
given that $\zeta(u)=0$ is satisfied. In other words, subject to this condition, we have 
\begin{align}\label{FINAL_G}
    \exp(G_T^i(u)- \Theta_0)=\exp\left(- \frac{1}{2\gamma_k^2}\int_0^T\norm{\pmb{\mu(t,\pmb{W}_t)}}^2 dt +\frac{1}{\gamma_k}\int_0^T\pmb{\mu(t,\pmb{W}_t)} d\pmb{W}_t\right).
\end{align}
We refer to \citet{Duncan2013} and \citet{karatzas1991} for the fact that \eqref{FINAL_G} can define an equivalent probability measure $\mathbb{\hat{P}}$, such that  $\frac{d\mathbb{\hat{P}}}{d\mathbb{P}}=\exp{\left(G_T^i(u) - \Theta_0\right)}$ under the condition $\zeta(u)=0$. The proof is complete.
Thus, $\exp(G_T^i(u))$ is a martingale. 

\end{proof}
Consider the quotient of martingales $\frac{M_{2,t}^i}{ M_{1,t}^i}$ from equation $\eqref{M2/M1}$ and the constant $\Theta_0$ from Theorem \ref{Girsanov exponent}. The quotient of two expectations will remain unchanged by being multiplied by a constant value $\exp{(-\Theta_0)}$ in the numerator and denominator leading to 
\begin{align}
    \frac{M_{2,t}^i}{ M_{1,t}^i}& = 
    \tfrac{\mb{E}^\mathbb{P} \Biggl[\exp{(G^i_T(u)-\Theta_0)}\biggl(\exp{\left(A_k^\tra T\right)} \hQ_k (x_T^i-\Phi_T)+\int_0^T  \exp{\left(A_k^\tra s\right)}  \left(Q_k (x_s^i-\Phi_s)+S_k u_s^{i,*} \right)ds  \biggr)\bigg| \mathcal{F}_t^i\Biggr]}{\mb{E}^\mathbb{P}\left[ \exp{(G^i_T(u)-\Theta_0)}|\mathcal{F}_t^i\right]}.
\end{align}
Recall that from Theorem \ref{Girsanov exponent}, we obtain a new measure $\mathbb{\hat{P}}$ defined by 
\begin{equation}
    \frac{d\mathbb{\hat{P}}}{d\mathbb{P}}=\exp(G_T^i(u)- \Theta_0) = \exp \left(- \frac{1}{2\gamma_k^2} \int_0^T \norm{\pmb{\mu(t,\pmb{W}_t)}}^2dt +\frac{1}{\gamma_k}\int_0^T\pmb{\mu(t,\pmb{W}_t)} d\pmb{W}_t\right).
\end{equation}
Moreover, based on \citet[Lemma 8.9.2]{StochasticIntegration},  we obtain the following equality 
\begin{align}
    \medmath{\frac{M_{2,t}^i}{ M_{1,t}^i}=\mb{E} \left[\exp{\left(A_k^\tra T\right)} \hQ_k (x_T^i-\Phi_T)+\int_0^T  \exp{\left(A_k^\tra s\right)}  \left(Q_k (x_s^i-\Phi_s)+S_k u_s^{i,*} \right)ds \bigg| \mathcal{F}_t^i\right] \text{\hspace{3 mm} $\mathbb{\hat{P}}$--a.s.}}\label{M under new measure}
\end{align}
We remark that \eqref{M under new measure} is a martingale under the measure $\mathbb{\hat{P}}$. For the sake of clarity and organization, we define
\begin{equation}
    \hat{M_t^i}=\mb{E}^\mathbb{\hat{P}} \left[\exp{\left(A_k^\tra T\right)} \hQ_k (x_T^i-\Phi_T)+\int_0^T  \exp{\left(A_k^\tra s\right)}  \left(Q_k (x_s^i-\Phi_s)+S_k u_s^{i,*} \right)ds \bigg| \mathcal{F}_t^i\right].
\end{equation}
Therefore, under $\mathbb{\hat{P}}$, \eqref{optimal-u_risk} transforms to
\begin{equation}
     \medmath{u_t^{i,*} = -R_k^{-1}\left[B_k^\tra  \exp{\left(-A_k^\tra  t\right)}  \left[\hat{M_t^i}-\int_0^t \biggl(  \exp{\left(A_k^\tra  s\right)}  \bigl( Q_k (x_s^i-\Phi_s)+ S_k u_s^{i,*} \bigr) \biggr)ds\right]+S_k^\tra (x_t^i-\Phi_t)\right]}.\label{u - risk_neutral}
\end{equation}
Under $\mathbb{\hat{P}}$, the computed $u_t^{i,*}$ is an implicit function. Subsequently, in order to obtain an explicit $u_t^{i,*}$, we investigate the existence of a linear feedback control representation under the new measure.

\subsection{Linear Feedback Representation of Optimal Control}\label{sec: Linear feedback control}
Using the Theorem \ref{Girsanov exponent}, we can obtain an implicit control law as shown in equation \eqref{u - risk_neutral}. To investigate the existence of linear feedback control under the new measure $\mathbb{\hat{P}}$, we introduce an adjoint process, which allows us to transform the control function into a linear process. Specifically, we can identify the control coefficients for the linear feedback control by equating the drift and diffusion terms of the agent dynamics in equation \eqref{dx_t} under the control functions obtained using the martingale representation theorem method, with the ones derived by applying Itô's lemma directly to the dynamics under the new measure $\mathbb{\hat{P}}$. Interestingly, we observe that the control coefficients coincide with the ones used in order to find the Radon-Nikodym exponent. This result underscores the intimate link between these vital methods for analyzing and optimizing stochastic processes.

\begin{theorem}\label{Explicit solution for control}
For the LQG risk-sensitive system described by \eqref{dx_t} and \eqref{cost integral_inf}, under the risk-neutral measure $\mathbb{\hat{P}}$, the optimal control action satisfying \eqref{u - risk_neutral}
admits the linear state feedback representation
\begin{align}
    u_t^{i,*}=-R_k^{-1}\left[(B_k^\tra\Pi^k_t+S_k^\tra) x_t^i+(B_k^\tra\Lambda^k_t-S_k^\tra H_k^{\pi})\bar{x}_t+B_k^\tra\Upsilon^k_t-S_k^\tra \eta_k \right]
\end{align}
where 
\begin{align}
&\begin{cases}\label{Pi}
    &d\Pi_t^k=\Biggl\{-\Pi_t^kA_k-A_k^\tra\Pi_t^k- Q_k+(\Pi_t^kB_k+S_k)R_k^{-1}(B_k^\tra\Pi_t^k+S_k^\tra)\\&\hspace{1.2cm}-\frac{1}{\gamma_k}\biggl[\Pi_t^k\sigma_k\sigma_k^\tra\Pi_t^k+(\Pi_t^k+\Lambda_t^k\pmb{1}_{Kn \times n})\sigma_0\sigma_0^\tra(\Pi_t^k+\pmb{1}_{n \times Kn}(\Lambda_t^k)^\tra)\biggr]\Biggr\}dt\\ 
    &\Pi_T^k=\hat{Q}_k\allowdisplaybreaks\\
\end{cases}\\
&\begin{cases}\label{Lambda}
    &d \Lambda_t^k=\Biggl\{-\Pi_t^kF_k^\pi-\Lambda_t^k\bar{A}_t-A_k^\tra\Lambda_t^k+Q_kH_k^{\pi} +(\Pi_t^kB_k+S_k)R_k^{-1}(B_k^\tra \Lambda_t^k-S_k^\tra H_k^{\pi})\\&\hspace{1.2cm}-\frac{1}{\gamma_k}\biggl[\Pi_t^k\sigma_k\sigma_k^\tra\Lambda_t^k+(\Pi_t^k+\Lambda_t^k\pmb{1}_{Kn \times n})\sigma_0\sigma_0^\tra(\pmb{1}_{n \times Kn}\Delta_t^k+\Lambda_t^k)\biggr]\Biggr\}dt\\
    &\Lambda_T^k=-\hat{Q}_kH_k^\pi
\end{cases}\\
&\begin{cases}\label{Upsilon}
    &d\Upsilon_t^k= \Biggl\{-\Pi_t^kb_k(t)-\Lambda_t^k \bar{m}_t-A_k^\tra \Upsilon_t^k+Q_k \eta_k+(\Pi_t^kB_k+S_k)R_k^{-1}(B_k^\tra \Upsilon_t^k-S_k^\tra \eta_k)\\&\hspace{1.2cm}-\frac{1}{\gamma_k}\left[\Pi_t^k\sigma_k\sigma_k^\tra\Upsilon_t^k+(\Pi_t^k+\Lambda_t^k\pmb{1}_{Kn \times n})\sigma_0\sigma_0^\tra(\Upsilon_t^k+\pmb{1}_{n \times Kn}\Gamma_t^k)\right]\Biggr\}dt\\ 
    &\Upsilon_T^k=-\hat{Q}_k\eta_k.
\end{cases}\\
&\begin{cases}\label{Delta}
    &d\Delta_t^k=\Biggl\{-(H_k^\pi)^\tra Q_k H_k^\pi +\Delta_t^k\bar{A}_t+\bar{A}_t^\tra\Delta_t^k-2(F_k^\pi)^\tra\Lambda_t^k\\&\hspace{1.2cm}+\left((\Lambda_t^k)^\tra B_k - (H_k^\pi)^\tra S_k\right)R_k^{-1}\bigl(B_k^\tra \Lambda_t^k - S_k^\tra H_k^\pi\bigr)\\&\hspace{1.2cm}-\frac{1}{\gamma_k}\biggl[(\Lambda_t^k)^\tra\sigma_k\sigma_k^\tra\Lambda_t^k+((\Lambda_t^k)^\tra+\Delta_t^k\pmb{1}_{Kn \times n})\sigma_0\sigma_0^\tra(\Lambda_t^k+\pmb{1}_{n \times Kn}\Delta_t^k)\biggr]\Biggr\}dt\\
    &\Delta_T^k=-(H_k^\pi)^\tra\Lambda_T^k
\end{cases}\allowdisplaybreaks\\
&\begin{cases}\label{Gamma}
    &d\Gamma_t^k=\Biggl\{-(H_k^\pi)^\tra Q_k \eta_k- (F_k^\pi)^\tra \Upsilon_t^k-(\Lambda_t^k)^\tra b_k(t) -\Delta_t^k\bar{m}_t-(\bar{A}_t)^\tra\Gamma_t^k \\&\hspace{1.2cm}+\left((\Lambda_t^k)^\tra B_k - (H_k^\pi)^\tra S_k\right)R_k^{-1}\left(B_k^\tra \Upsilon_t^k - S_k^\tra \eta_k \right)\\&\hspace{1.2cm}-\frac{1}{\gamma_k}\left[(\Lambda_t^k)^\tra\sigma_k\sigma_k^\tra\Upsilon_t^k+((\Lambda_t^k)^\tra+\Delta_t^k\pmb{1}_{Kn \times n})\sigma_0\sigma_0^\tra(\Upsilon_t^k+\pmb{1}_{n \times Kn}\Gamma_t^k)\right]\Biggr\}dt\\
    &\Gamma_T^k=-(H_k^\pi)^\tra\Upsilon_T^k
\end{cases}\allowdisplaybreaks\\
&\begin{cases}\label{Psi}
    &d\Psi_t^k=\Biggl\{-\eta^\tra_k
    Q_k \eta_k - 2(\Upsilon_t^k)^\tra b_k(t) - 2 (\Gamma_t^k)^\tra\bar{m}_t- tr(\sigma_0^\tra(\Pi_t^k+2\pmb{1}_{n \times Kn}(\Lambda_t^k)^\tra\\&\hspace{1.2cm}+\pmb{1}_{n \times Kn}\Delta_t^k\pmb{1}_{Kn \times n})\sigma_0)-tr(\sigma_k^\tra \Pi_t^k \sigma_k) + ((\Upsilon_t^k)^\tra B_k - \eta^\tra_k S_k)R^{-1}_k(B_k^\tra\Upsilon_t^k-S_k^\tra \eta_k)\\&\hspace{1.2cm}-\frac{1}{\gamma_k}\left[(\Upsilon_t^k)^\tra\sigma_k\sigma_k^\tra\Upsilon_t^k+((\Upsilon_t^k)^\tra+(\Gamma_t^k)^\tra\pmb{1}_{Kn \times n})\sigma_0\sigma_0^\tra(\Upsilon_t^k+\pmb{1}_{n \times Kn}\Gamma_t^k)\right]\Biggr\}dt\\
    &\Psi_t^k=-\eta_k^\tra\Upsilon_T^k.
\end{cases}
\end{align}
with 
\begin{equation}
    \bar{A}_t = \begin{bmatrix} \bar{A}_1\\ \vdots\\\bar{A}_K\end{bmatrix}\in \mathbb{R}^{Kn \times Kn}, \quad 
    \bar{m}_t = \begin{bmatrix} \bar{m}_1\\ \vdots \\ \bar{m}_K \end{bmatrix} \in \mathbb{R}^{Kn},
\end{equation}
and for $k \in \{1,2,...,K\}$
\begin{align}
    \bar{A}_k  =& \left[A_k - B_k R_k^{-1}(B_k^\tra\Pi^k_t+S_k^\tra)\right] \pmb{e}_k+ F^{\pi}_k - B_k R_k^{-1} (B_k^\tra\Lambda^k_t-S_k^\tra H_k^{\pi}),\\
    \bar{m}_k =& b_k+B_kR_k^{-1}S_k^\tra \eta_k-B_k R_k^{-1} B_k^\tra\Upsilon^k_t.
\end{align}
Furthermore, the diffusion terms satisfy the following equations
\begin{align}
    &\Pi_t^k\sigma_k=\exp{\left(-A_k^\tra t\right)} Z_t^i\\
    &(\Pi_t^k+\Lambda_t^k\pmb{1}_{Kn \times n})\sigma_0=\exp{\left(-A_k^\tra t\right)} Z_t^0.
\end{align}
In addition, $u_t^{i,*}$ satisfies Condition \eqref{Condition change of measure} under $\mathbb{\hat{P}}$.
\end{theorem}
\begin{proof}
Under $\hat{\mathbb{P}}$, we define the adjoint process $(p_t^i)_{t \in \mfT}$ where as 
\begin{equation}\label{p-def-neutral}
    p_t^i =\exp{\left(-A_k^\tra t\right)}  \left[\hat{M_t^i}-\int_0^t \left(  \exp{\left(A_k^\tra s\right)}  \left( Q_k (x_s^i-\Phi_s)+S_k u_s^{i,*} \right) \right)ds\right] \text{\hspace{5 mm}$\hat{\mathbb{P}}$--a.s.}
\end{equation}
By the martingale representation theorem, there exists a
$\mc{F}_t^i$-adapted process $(Z_s)_{s \in \mfT}$ such that
\begin{equation}\label{martingale-neutral}
    \hat{M_t^i}= \hat{M_0^i} + \int_0^t Z_s^i d\hat{w}_s^i+\int_0^t Z_s^0 d\hat{w}_s^0.
\end{equation}
Under $\hat{\mathbb{P}}$, we adopt the following ansatz for the adjoint process
\begin{equation}\label{p-ansatz-neutral}
    p_t^i=\Pi_t^kx_t^i+\Lambda_t^k \bar{x}_t+\Upsilon_t^k \text{\hspace{5 mm}$\hat{\mathbb{P}}$--a.s.}, 
\end{equation}
where  $\Pi_t^k \in \mathbb{S}^{n \times n}$, $\Lambda_t^k \in \mathbb{R}^{n \times Kn}$  and $\Upsilon_t^k \in \mathbb{R}^n$ are deterministic functions of time.\\
We substitute \eqref{p-ansatz-neutral} in \eqref{u - risk_neutral} to get 
\begin{align}
    u_t^{i,*}&=-R_k^{-1} \left[B_k^\tra (\Pi^k_tx_t^i+\Lambda_t^k \bar{x}_t+\Upsilon_t^k)+S_k^\tra (x_t^i - \Phi_t)\right] \nnb
    \\&=-R_k^{-1}\left[(B_k^\tra\Pi^k_t+S_k^\tra) x_t^i+(B_k^\tra\Lambda^k_t-S_k^\tra H_k^{\pi})\bar{x}_t+B_k^\tra\Upsilon^k_t-S_k^\tra \eta_k \right]\label{optimal-u2-neutral} \text{\hspace{5 mm}$\hat{\mathbb{P}}$--a.s.}
\end{align}
Subsequently, the mean field of control actions is given by   $ \bar{u}_t^\tra = \begin{bmatrix}
(\bar{u}_t^1)^\tra &
\dots &
(\bar{u}_t^K)^\tra
\end{bmatrix}$ where
\begin{align}
    \bar{u}_t^k =-R_k^{-1}\left[(B_k^\tra\Pi_t^k+S_k^\tra) \bar{x}_t^k+(B_k^\tra\Lambda^k_t-S_k^\tra H_k^\pi)\bar{x}_t+B_k^\tra\Upsilon^k_t-S_k^\tra \eta_k \right]\label{optimal-bar-u2-neutral} \text{\hspace{5 mm}$\hat{\mathbb{P}}$--a.s.}
\end{align}
We then substitute \eqref{martingale-neutral} in \eqref{p-def-neutral}, and apply Itô's lemma to get
\begin{align}\label{dp_t}
    dp_t^i = & -\left\{A_k^\tra p_t^i+ Q_k (x_t^i-H_k^{\pi}\bar{x}_t-\eta_k)+ S_ku_t^{i,*}\right\}dt \nnb \\ & \hspace{3 mm}+\exp{\left(-A_k^\tra t\right)} Z_t^id\hat{w}_t^i+\exp{\left(-A_k^\tra t\right)} Z_t^0d\hat{w}_t^0 \text{\hspace{3 mm}$\hat{\mathbb{P}}$--a.s.}
\end{align}
Next, we substitute \eqref{optimal-u2-neutral} in \eqref{dp_t}, which results in
\begin{align}
    dp_t^i&=-\Bigl\{A_k^\tra \left\{ \Pi_t^kx_t^i+\Lambda_t^k \bar{x}_t+\Upsilon_t^k\right\}+ Q_k (x_t^i-H_k^{\pi}\bar{x}_t-\eta_k)\nnb\\&\text{\hspace{3mm}}- S_kR_k^{-1}\left[(B_k^\tra\Pi_t^k+S_k^\tra) x_t^i+(B_k^\tra\Lambda_t^k-S_k^\tra H_k^{\pi})\bar{x}_t+B_k^\tra\Upsilon_t^k-S_k^\tra \eta_k \right]\Bigr\}dt\nnb\\&\text{\hspace{3mm}}+\exp{\left(-A_k^\tra t\right)} Z_t^id\hat{w}_t^i+\exp{\left(-A_k^\tra t\right)} Z_t^0d\hat{w}_t^0, \text{\hspace{5 mm}$\hat{\mathbb{P}}$--a.s.}
\end{align}
By reordering the terms, the above equation is expressed as 
\begin{align}
    dp_t^i&=\left\{-A_k^\tra\Pi_t^k- Q_k+S_kR_k^{-1}(B_k^\tra\Pi_t^k+S_k^\tra) \right\}x_t^idt\nnb \\&\text{\hspace{3mm}}
    +\left\{-A_k^\tra\Lambda_t^k+Q_kH_k^{\pi}+ S_kR_k^{-1}(B_k^\tra\Lambda_t^k-S_k^\tra H_k^{\pi}) \right\}\bar{x}_tdt\nnb\\&\text{\hspace{3mm}}
    +\left\{-A_k^\tra \Upsilon_t^k+Q_k \eta_k+S_kR_k^{-1}B_k^\tra\Upsilon_t^k-S_kR_k^{-1}S_k^\tra \eta_k  \right\}dt\nnb\\&\text{\hspace{3mm}}
    +\exp{\left(-A_k^\tra t\right)} Z_t^id\hat{w}_t^i+\exp{\left(-A_k^\tra t\right)} Z_t^0d\hat{w}_t^0, \text{\hspace{5 mm}$\hat{\mathbb{P}}$--a.s.} \label{dp-sol2-neutral}
\end{align}
Next, we apply Itô's lemma to \eqref{p-ansatz-neutral} to get 
\begin{align}
    dp_t^i &=d\Pi_t^kx_t^i+\Pi_t^kdx_t^i+d\Lambda_t^k \bar{x}_t+ \Lambda_t^k         
    d\bar{x}_t+d\Upsilon_t^k, \text{\hspace{5 mm}$\hat{\mathbb{P}}$--a.s.} \label{dp_t2}
\end{align}
In order to obtain the dynamics that $p^i_t$ satisfies under $\hat{\mathbb{P}}$, it is essential to derive both the agent's and the mean field dynamics under the new measure $\hat{\mathbb{P}}$. From Theorem \ref{Girsanov exponent}, and by expanding the term $\pmb{\mu(t,\pmb{W}_t)}$, the Wiener processes under $\mathbb{P}$ are given by 
\begin{align}
    d\hat{w}_t^i=&dw_t^i-\frac{1}{\gamma_k}\sigma_k^\tra (\Pi_t^k x_t^i +\Lambda_t^k \bar{x}_t +\Upsilon_t^k)dt\\
    d\hat{w}_t^0=&dw_t^0-\frac{1}{\gamma_k}\sigma_0^\tra(\Pi_t^k x_t^i +\Lambda_t^k \bar{x}_t+\pmb{1}_{n \times Kn}(\Lambda_t^k)^\tra x_t^i+\pmb{1}_{n \times Kn}\Delta_t^k \bar{x}_t +\Upsilon_t^k+\pmb{1}_{n \times Kn}\Gamma_t^k)dt.
\end{align}
Thus, under $\hat{\mathbb{P}}$, the dynamics \eqref{dx_t} and \eqref{MF_dyn_int} are, respectively, expressed as
\begin{align}
    dx_t^i =& \left( A_kx_t^i+F_k^\pi\bar{x}_t+B_ku_t^{i,*}+b_k(t) \right) dt 
    + \sigma_k\left(d\hat{w}_t^i+ \frac{1}{\gamma_k}\sigma_k^\tra (\Pi_t^k x_t +\Lambda_t^k \bar{x}_t +\Upsilon_t^k)dt \right) \nnb\\ &+\sigma_0\Bigl(d\hat{w}_t^0 + \frac{1}{\gamma_k}\sigma_0^\tra(\Pi_t^k x_t^i +\Lambda_t^k \bar{x}_t+\pmb{1}_{n \times Kn}(\Lambda_t^k)^\tra x_t^i\nnb\\&+\pmb{1}_{n \times Kn}\Delta_t^k \bar{x}_t +\Upsilon_t^k+\pmb{1}_{n \times Kn}\Gamma_t^k)dt \Bigr) \text{\hspace{5 mm}$\hat{\mathbb{P}}$--a.s.},\\
    d\bar{x}_t = &(\br{A}\bar{x}_t  + \br{B} \bar{u}_t^{*}  + \br{m}_t)dt + \pmb{1}_{Kn \times n}\sigma_0\Big(d\hat{w}_t^0 + \frac{1}{\gamma_k}\sigma_0^\tra(\Pi_t^k x_t^i +\Lambda_t^k \bar{x}_t \nnb\\&+\pmb{1}_{n \times Kn}(\Lambda_t^k)^\tra x_t^i+\pmb{1}_{n \times Kn}\Delta_t^k \bar{x}_t +\Upsilon_t^k+\pmb{1}_{n \times Kn}\Gamma_t^k)dt \Big) \text{\hspace{5 mm}$\hat{\mathbb{P}}$--a.s.}
\end{align} 
Now, we substitute the control action \eqref{optimal-u2-neutral} and the mean field of control actions \eqref{optimal-bar-u2-neutral} in the above agent and mean field dynamics under $\hat{\mathbb{P}}$ to obtain
\begin{align} 
    dx_t^i =&  \left[A_k-B_kR_k^{-1}(B_k^\tra\Pi^k_t+S_k^\tra)\right] x_t^idt +\left[F_k^\pi-B_kR_k^{-1}(B_k^\tra\Lambda^k_t-S_k^\tra H_k^{\pi})\right]\bar{x}_tdt \nnb\\&+\left[-B_kR_k^{-1}\left[B_k^\tra\Upsilon^k_t-S_k^\tra \eta_k \right]+b_k(t)\right] dt + \sigma_k\left(d\hat{w}_t^i+ \frac{1}{\gamma_k}\sigma_k^\tra ((\Pi_t^k)^\tra x_t^i +\Lambda_t^k \bar{x}_t +\Upsilon_t^k)dt \right) \nnb\\& +\sigma_0\left(d\hat{w}_t^0 + \frac{1}{\gamma_k}\sigma_0^\tra(\Pi_t^k x_t^i +\Lambda_t^k \bar{x}_t+\pmb{1}_{n \times Kn}(\Lambda_t^k)^\tra x_t^i+\pmb{1}_{n \times Kn}\Delta_t^k \bar{x}_t +\Upsilon_t^k+\pmb{1}_{n \times Kn}\Gamma_t^k)dt \right) 
\end{align}
and 
\begin{align}
    d\bar{x}_t =& (\bar{A}_t\bar{x}_t + \bar{m}_t)dt\nnb \\& + \pmb{1}_{Kn \times n}\sigma_0\Bigl(d\hat{w}_t^0 + \frac{1}{\gamma_k}\sigma_0^\tra(\Pi_t^k x_t^i +\Lambda_t^k \bar{x}_t+\pmb{1}_{n \times Kn}(\Lambda_t^k)^\tra x_t^i \nnb \\ & +\pmb{1}_{n \times Kn}\Delta_t^k \bar{x}_t +\Upsilon_t^k+\pmb{1}_{n \times Kn}\Gamma_t^k)dt \Bigr)
\end{align}
where 
\begin{equation}\label{bar_A & bar_M}
    \bar{A}_t = \begin{bmatrix} \bar{A}_1\\ \vdots\\\bar{A}_K\end{bmatrix}\in \mathbb{R}^{Kn \times Kn}, \quad 
    \bar{m}_t = \begin{bmatrix} \bar{m}_1\\ \vdots \\ \bar{m}_K \end{bmatrix} \in \mathbb{R}^{Kn \times 1},
\end{equation}
and for $k \in \{1,2,...,K\}$
\begin{align}
    \bar{A}_k  =& \left[A_k - B_k R_k^{-1}(B_k^\tra\Pi^k_t+S_k^\tra)\right] \pmb{e}_k+ F^{\pi}_k - B_k R_k^{-1} (B_k^\tra\Lambda^k_t-S_k^\tra H_k^{\pi}),\\
    \bar{m}_k =& b_k+B_kR_k^{-1}S_k^\tra \eta_k-B_k R_k^{-1} B_k^\tra\Upsilon^k_t.
\end{align}
Finally, we substitute the derived agent dynamics and mean field dynamics under $\hat{\mathbb{P}}$ in \eqref{dp_t2} to obtain the dynamics that $p^i_t$ satisfies as 
\begin{align}
    dp_t^i=&d\Pi_t^kx_t^i+\Bigl\{\Pi_t^kA_k-\Pi_t^kB_kR_k^{-1}(B_k^\tra\Pi_t^k+S_k^\tra)\nnb\\&+\frac{1}{\gamma_k}\left[\Pi_t^k\sigma_k\sigma_k^\tra\Pi_t^k+(\Pi_t^k+\Lambda_t^k\pmb{1}_{Kn \times n})\sigma_0\sigma_0^\tra(\Pi_t^k+\pmb{1}_{n \times Kn}(\Lambda_t^k)^\tra)\right]\Bigr\}x_t^idt\nnb\\&+d \Lambda_t^k\bar{x}_t+\Bigl\{\Pi_t^kF_k^\pi -\Pi_t^kB_kR_k^{-1}(B_k^\tra \Lambda_t^k-S_k^\tra H_k^{\pi})+\Lambda_t^k\bar{A}_t\nnb\\&+\frac{1}{\gamma_k}\left[\Pi_t^k\sigma_k\sigma_k^\tra\Lambda_t^k+(\Pi_t^k+\Lambda_t^k\pmb{1}_{Kn \times n})\sigma_0\sigma_0^\tra(\pmb{1}_{n \times Kn}\Delta_t^k+\Lambda_t^k)\right]\Bigr\}\bar{x}_tdt \nnb\\&+ d\Upsilon_t^k +\Bigl\{-\Pi_t^kB_kR_k^{-1}B_k^\tra \Upsilon_t^k+\Pi_t^kB_kR_k^{-1}S_k^\tra \eta_k +\Pi_t^kb_k(t)+\Lambda_t^k\bar{m}_t\nnb\\&+\frac{1}{\gamma_k}\left[\Pi_t^k\sigma_k\sigma_k^\tra\Upsilon_t^k+(\Pi_t^k+\Lambda_t^k\pmb{1}_{Kn \times n})\sigma_0\sigma_0^\tra(\Upsilon_t^k+\pmb{1}_{n \times Kn}\Gamma_t^k)\right]    \Bigr\}dt\nnb\\&+\Pi_t^k\sigma_kd\hat{w}_t^i+(\Pi_t^k+\Lambda_t^k\pmb{1}_{Kn \times n})\sigma_0d\hat{w}_t^0,\label{dp-sol1-neutral} \text{\hspace{5 mm}$\hat{\mathbb{P}}$--a.s.}
\end{align}
Since the two SDEs \eqref{dp-sol2-neutral} and \eqref{dp-sol1-neutral} that $p^i_t$ satisfies must align for every sample path of the Wiener processes, it is necessary for both the drift coefficients and the diffusion coefficients to be identical. Equating the drift coefficients of \eqref{dp-sol2-neutral} and \eqref{dp-sol1-neutral}, we have

\begin{align}
&\begin{cases}\label{Pi2}
    &d\Pi_t^k=\Biggl\{-\Pi_t^kA_k-A_k^\tra\Pi_t^k- Q_k+(\Pi_t^kB_k+S_k)R_k^{-1}(B_k^\tra\Pi_t^k+S_k^\tra)\\&\hspace{1.2cm}-\frac{1}{\gamma_k}\biggl[\Pi_t^k\sigma_k\sigma_k^\tra\Pi_t^k+(\Pi_t^k+\Lambda_t^k\pmb{1}_{Kn \times n})\sigma_0\sigma_0^\tra(\Pi_t^k+\pmb{1}_{n \times Kn}(\Lambda_t^k)^\tra)\biggr]\Biggr\}dt\\ 
    &\Pi_T^k=\hat{Q}_k\\
\end{cases}\allowdisplaybreaks\\
&\begin{cases}\label{Lambda2}
    &d \Lambda_t^k=\Biggl\{-\Pi_t^kF_k^\pi-\Lambda_t^k\bar{A}_t-A_k^\tra\Lambda_t^k+Q_kH_k^{\pi} +(\Pi_t^kB_k+S_k)R_k^{-1}(B_k^\tra \Lambda_t^k-S_k^\tra H_k^{\pi})\\&\hspace{1.2cm}-\frac{1}{\gamma_k}\biggl[\Pi_t^k\sigma_k\sigma_k^\tra\Lambda_t^k+(\Pi_t^k+\Lambda_t^k\pmb{1}_{Kn \times n})\sigma_0\sigma_0^\tra(\pmb{1}_{n \times Kn}\Delta_t^k+\Lambda_t^k)\biggr]\Biggr\}dt\\
    &\Lambda_T^k=-\hat{Q}_kH_k^\pi
\end{cases}\\
&\begin{cases}\label{Upsilon2}
    &d\Upsilon_t^k= \Biggl\{-\Pi_t^kb_k(t)-\Lambda_t^k \bar{m}_t-A_k^\tra \Upsilon_t^k+Q_k \eta_k+(\Pi_t^kB_k+S_k)R_k^{-1}(B_k^\tra \Upsilon_t^k-S_k^\tra \eta_k)\\&\hspace{1.2cm}-\frac{1}{\gamma_k}\biggl[\Pi_t^k\sigma_k\sigma_k^\tra\Upsilon_t^k+(\Pi_t^k+\Lambda_t^k\pmb{1}_{Kn \times n})\sigma_0\sigma_0^\tra(\Upsilon_t^k+\pmb{1}_{n \times Kn}\Gamma_t^k)\biggr]\Biggr\}dt\\ 
    &\Upsilon_T^k=-\hat{Q}_k\eta_k.
\end{cases}
\end{align}
By equating the diffusion coefficients of \eqref{dp-sol2-neutral} and \eqref{dp-sol1-neutral}, we obtain 
\begin{align}
    &\Pi_t^k\sigma_k=\exp{\left(-A_k^\tra t\right)} Z_t^i,\\
    &(\Pi_t^k+\Lambda_t^k\pmb{1}_{Kn \times n})\sigma_0=\exp{\left(-A_k^\tra t\right)} Z_t^0.
\end{align}
Now, our focus turns to characterizing $\Delta^k_t$, $\Gamma^k_t$, and $\Psi^k_t$. For the change of measure to be valid, and consequently, the obtained equations \eqref{Pi2}--\eqref{Upsilon2} to hold, it is essential to satisfy the condition \eqref{Condition change of measure}. To do so, we substitute the control action \eqref{optimal-u2-neutral}, the mean field of control actions \eqref{optimal-bar-u2-neutral}, and equations \eqref{Pi2}--\eqref{Upsilon2} into condition \eqref{Condition change of measure} under $\hat{\mathbb{P}}$, resulting in

\begin{align}
    \zeta(u)= &\int_0^T \Biggl(\frac{1}{2\gamma_k}(\bar{x}_t)^\tra (H_k^\pi)^\tra Q_kH_k^\pi \bar{x}_t+\frac{1}{\gamma_k}\eta^\tra_k Q_k H_k^\pi\bar{x}_t +\frac{1}{2\gamma_k}\eta^\tra_k Q_k \eta_k
    + \frac{1}{2\gamma_k}   \bigl[(\bar{x}_t)^\tra((H_k^{\pi})^\tra S_k \nnb\allowdisplaybreaks\\&- (\Lambda_t^k)^\tra B_k)  - (\Upsilon^k_t)^\tra B_k + (\eta_k)^\tra S_k \bigr]R_k^{-1} \bigl[(B_k^\tra\Lambda^k_t-S_k^\tra H_k^{\pi})\bar{x}_t+B_k^\tra\Upsilon^k_t-S_k^\tra \eta_k \bigr] \nnb\\& + \frac{1}{\gamma_k}\Bigl[(\bar{x}_t)^\tra((F_k^\pi)^\tra\Lambda_t^k + \Delta_t^k\bar{A}_t ) \bar{x}_t +((\Upsilon_t^k)^\tra F_k^\pi +(\Gamma_t^k)^\tra \bar{A}_t )\bar{x}_t) + (\bar{x}_t)^\tra (\Lambda_t^k)^\tra b_k(t) \nnb\\&+ (\bar{x}_t)^\tra \Delta_t^k \bar{m}_t+ (\Upsilon_t^k)^\tra b_k(t)+ (\Gamma_t^k)^\tra\bar{m}_t \Bigr] +\frac{1}{2\gamma_k^2}(\bar{x}_t)^\tra\Bigl( (\Lambda_t^k)^\tra\sigma_k\sigma_k^\tra\Lambda_t^k+((\Lambda_t^k)^\tra\nnb\\&+\Delta_t^k\pmb{1}_{Kn \times n})\sigma_0\sigma_0^\tra(\Lambda_t^k+\pmb{1}_{n \times Kn}\Delta_t^k) \Bigr) \bar{x}_t
    +\frac{1}{2\gamma_k^2}tr\Bigl((\pmb{\Sigma}^k)^\tra\pmb{H}_t^k\pmb{X}_t(\pmb{C}_t^k)^\tra\pmb{\Sigma}^k \nnb\\&+(\pmb{\Sigma}^k)^\tra\pmb{C}_t^k(\pmb{X}_t)^\tra\pmb{H}_t^k\pmb{\Sigma}^k+(\pmb{\Sigma}^k)^\tra\pmb{C}_t^k(\pmb{C}_t^k)^\tra\pmb{\Sigma}^k\Bigr) dt 
    +\frac{1}{2\gamma_k}\int_0^T tr(\sigma_k^\tra \Pi_t^k \sigma_k\nnb \\&+\sigma_0^\tra\left(\Pi_t^k+2\pmb{1}_{n \times Kn}(\Lambda_t^k)^\tra+\pmb{1}_{n \times Kn}\Delta_t^k\pmb{1}_{Kn \times n})\sigma_0\right)dt +\frac{1}{2\gamma_k} \int_0^T (\bar{x}_t)^\tra  d\Delta_t^k \bar{x}_t \nnb \\& +\frac{1}{\gamma_k}\int_0^T (d\Gamma_t^k)^\tra
   \bar{x}_t+\frac{1}{2 \gamma_k}\int_0^T d\Psi_t^k = 0, \text{\hspace{5 mm}$\hat{\mathbb{P}}$--a.s.}
\end{align}
To satisfy the above condition, we further impose the following constraints on the coefficients of the control action: 
\begin{align}
&\begin{cases}
    &d\Delta_t^k=\Biggl\{-(H_k^\pi)^\tra Q_k H_k^\pi +\Delta_t^k\bar{A}_t+\bar{A}_t^\tra\Delta_t^k-2(F_k^\pi)^\tra\Lambda_t^k\\&\hspace{1.2cm}+\left((\Lambda_t^k)^\tra B_k - (H_k^\pi)^\tra S_k\right)R_k^{-1}\left(B_k^\tra \Lambda_t^k - S_k^\tra H_k^\pi\right)\\&\hspace{1.2cm}-\frac{1}{\gamma_k}\biggl[(\Lambda_t^k)^\tra\sigma_k\sigma_k^\tra\Lambda_t^k+((\Lambda_t^k)^\tra+\Delta_t^k\pmb{1}_{Kn \times n})\sigma_0\sigma_0^\tra(\Lambda_t^k+\pmb{1}_{n \times Kn}\Delta_t^k)\biggr]\Biggr\}dt\\
    &\Delta_T^k=-(H_k^\pi)^\tra\Lambda_T^k
\end{cases}\allowdisplaybreaks\\
&\begin{cases}
    &d\Gamma_t^k=\Biggl\{-(H_k^\pi)^\tra Q_k \eta_k- (F_k^\pi)^\tra \Upsilon_t^k-(\Lambda_t^k)^\tra b_k(t) -\Delta_t^k\bar{m}_t-(\bar{A}_t)^\tra\Gamma_t^k \\&\hspace{1.2cm}+\left((\Lambda_t^k)^\tra B_k - (H_k^\pi)^\tra S_k\right)R_k^{-1}\left(B_k^\tra \Upsilon_t^k - S_k^\tra \eta_k \right)\\&\hspace{1.2cm}-\frac{1}{\gamma_k}\biggl[(\Lambda_t^k)^\tra\sigma_k\sigma_k^\tra\Upsilon_t^k+((\Lambda_t^k)^\tra+\Delta_t^k\pmb{1}_{Kn \times n})\sigma_0\sigma_0^\tra(\Upsilon_t^k+\pmb{1}_{n \times Kn}\Gamma_t^k)\biggr]\Biggr\}dt\\
    &\Gamma_T^k=-(H_k^\pi)^\tra\Upsilon_T^k
\end{cases}\allowdisplaybreaks\\
&\begin{cases}
    &d\Psi_t^k=\Biggl\{-\eta^\tra_k
    Q_k \eta_k - 2(\Upsilon_t^k)^\tra b_k(t) - 2 (\Gamma_t^k)^\tra\bar{m}_t- tr(\sigma_0^\tra(\Pi_t^k+2\pmb{1}_{n \times Kn}(\Lambda_t^k)^\tra\\&\hspace{1.2cm}+\pmb{1}_{n \times Kn}\Delta_t^k\pmb{1}_{Kn \times n})\sigma_0)-tr(\sigma_k^\tra \Pi_t^k \sigma_k) + ((\Upsilon_t^k)^\tra B_k - \eta^\tra_k S_k)R^{-1}_k(B_k^\tra\Upsilon_t^k-S_k^\tra \eta_k)\\&\hspace{1.2cm}-\frac{1}{\gamma_k}\biggl[(\Upsilon_t^k)^\tra\sigma_k\sigma_k^\tra\Upsilon_t^k+((\Upsilon_t^k)^\tra+(\Gamma_t^k)^\tra\pmb{1}_{Kn \times n})\sigma_0\sigma_0^\tra(\Upsilon_t^k+\pmb{1}_{n \times Kn}\Gamma_t^k)\biggr]\Biggr\}dt\\
    &\Psi_t^k=-\eta^\tra_k\Upsilon_T^k.
\end{cases}
\end{align}
\end{proof}

We have derived the optimal control under $\mathbb{\hat{P}}$. Now, we investigate its relationship with the optimal control action under the original measure $\mathbb{P}$.

\begin{theorem}\label{Optimal control risk sensitive}
Under the risk-sensitive measure $\mathbb{P}$, the optimal control action for the LQG risk-sensitive system, described by \eqref{dx_t}--\eqref{cost_comp_inf}, admits the linear state feedback representation 
\begin{equation}\label{optimal_u_under_P}
u_t^{i,*}=-R_k^{-1}\left[(B_k^\tra\Pi^k_t+S_k^\tra) x_t^i+(B_k^\tra\Lambda^k_t-S_k^\tra H_k^{\pi})\bar{x}_t+B_k^\tra\Upsilon^k_t-S_k^\tra \eta_k \right],
\end{equation}
where $\Pi^k_t$, $\Lambda^k_t$, and $\Upsilon^k_t$ are characterized by \eqref{Pi}--\eqref{Psi} given in Theorem \ref{Explicit solution for control}.
\end{theorem}
\begin{proof}
Consider the sample space $\Omega$. Then, $u_t^{i,*}$ admits the  representation 
\begin{equation}
    u_t^{i,*}=-R_k^{-1}\left[(B_k^\tra\Pi^k_t+S_k^\tra) x_t^i+(B_k^\tra\Lambda^k_t-S_k^\tra H_k^{\pi})\bar{x}_t+B_k^\tra\Upsilon^k_t-S_k^\tra \eta_k \right], \text{\hspace{5 mm}$\hat{\mathbb{P}}$--a.s.}
\end{equation}
if and only if
\begin{equation}\label{a.s. equality in risk neutral}
    \hat{\mathbb{P}}\left(\left\{\nu |u_t^{i,*}(\nu) \neq -R_k^{-1}\left[(B_k^\tra\Pi^k_t+S_k^\tra) x_t^{i}(\nu)+(B_k^\tra\Lambda^k_t-S_k^\tra H_k^{\pi})\bar{x}_t(\nu)+B_k^\tra\Upsilon^k_t-S_k^\tra \eta_k \right]\right\}\right)= 0,
\end{equation}
where $\nu \in \Omega$ represents a state of the world.
By Girsanov theorem, $\hat{\mathbb{P}}$ is a measure equivalent to $\mathbb{P}$.
Thus, by the equivalence of measure, \eqref{a.s. equality in risk neutral} implies that
\begin{equation}
    \mathbb{P} \left(\left\{\nu |u_t^{i,*}(\nu) \neq -R_k^{-1}\left[(B_k^\tra\Pi^k_t+S_k^\tra) x_t^i(\nu)+(B_k^\tra\Lambda^k_t-S_k^\tra H_k^{\pi})\bar{x}_t(\nu)+B_k^\tra\Upsilon^k_t-S_k^\tra \eta_k \right]\right\}\right)= 0.
\end{equation}
Thus, under $\mb{P}$, the optimal control action admits the representation \eqref{optimal_u_under_P} with the respective control coefficients. In other words,  \eqref{optimal_u_under_P} makes the Gâteaux derivative \eqref{Gâteaux derivative} zero. Therefore, it is the optimal control action that minimizes the cost functional \eqref{cost integral_inf} given the dynamics of the system \eqref{dx_t}.
\end{proof}

\subsection{Nash Equilibrium}
\begin{definition}[Nash Equilibrium]
A set of strategies $\{u^i, i=1,2,\dots\} \in \mc{U}^1 \times \dots \times \mc{U}^N$ achieves the Nash equilibrium for all $N$ plays given the cost functional $J^i$ for each if for every agent $i \in \mfN$ with any admissible strategy $u \in \mc{U}^i$
\begin{equation}
    J^i(u^1,\dots,u^i,\dots,u^N) \leq J^i(u^1,\dots,u^{i-1},u,u^{i+1},\dots,u^N).
\end{equation}
\end{definition}

In other words, in the Nash equilibrium, no agent will be better off, specifically in this case, with a smaller cost, if it unilaterally deviates from the strategies established by the equilibrium.

\begin{theorem}
Consider the optimal control \eqref{optimal_u_under_P} obtained in Theorem \ref{Optimal control risk sensitive} for LQG risk-sensitive system. For the infinite-population model given the system described by the dynamics \eqref{dx_t} and the cost functional \eqref{cost integral_inf}, the set of the optimal controls $\{u^{i,*}, i=1,2,\dots\}$ for agents yields a Nash equilibrium.
\end{theorem}
\begin{proof}
    Considering that all agents adhere to the optimal strategies outlined in Theorem \ref{Optimal control risk sensitive}, we can establish the validity of the theorem statement. In situations where an individual agent $i$ chooses to diverge from the set of strategies unilaterally, the influence on the mean field will be insignificant. Consequently, on the one hand, this prompts the remaining agents to execute the original control, with the aim of minimizing the cost functional. On the other hand, as the mean field state is unchanged, by Theorem \ref{Optimal control risk sensitive}, the optimal control of the agent $i$ in question remains to be $u^{i,*}$. Therefore, any deviation of the agent $i$ from this optimal control $u^{i,*}$ will not lead to a cost reduction. 
\end{proof}
In this work, our focus is on the infinite-population scenario. The connection between the obtained Nash equilibrium strategies and the original finite-population system may be established by following along the lines of proof in \citep{liu2023}. More specifically, it can be shown that these strategies yield an approximate Nash ($\epsilon$-Nash) equilibrium for the finite-population system.

\section{Application: An Interbank Market Model}\label{sec: Application}
In the context of the interbank market, we undertake a study utilizing the LQG risk-sensitive model introduced in Section \ref{sec: Infinite Population Model}. Our objective is to acquire a deeper understanding of the dynamics involved in interbank lending and borrowing. In this context, agents represent banks and their state represents the logarithmic monetary reserve (log-reserve) of the bank. A representative bank, driven by its financial requirements in different periods, engages in lending activities by purchasing bonds from the central bank and lending to other banks, or engages in borrowing activities with the central bank and other banks, all while striving to minimize operational costs. Within the same framework, the mean field state is illustrated by the limiting average of the log-reserves held by all banks participating in the market. Subsequently, we will henceforth denote this mean field state as the market state. In this section, we introduce a simplified version of the model to give an example. However, the general model can also be used similarly when there is a demand.

We consider log-reserves of banks and of the market and their control action to be scalars and reduce the dimension of the matrices by setting $K=r=1$. Consequently, the market exclusively comprises homogeneous banks sharing the same model parameters each subject to an idiosyncratic shock and a common noise. The common noise in each case can be viewed as the common impact of the market environment at a macro level on the banks. In this setting, the banks are correlated due to being impacted by the common noise as presented in Section \ref{sub: dynamics_infinite}. In addition, although independent of each other, the idiosyncratic and the common shocks will affect the banks by the same factor $\rho$. Consequently, the interplay between these shocks has a combined effect on the log-reserve of an individual bank and the market.

In this section, we begin by presenting an optimization problem in the context of interbank transactions. We consider first the model parameters,  presented in Section \ref{sec: Finite Population Model} and \ref{sec: Infinite Population Model}, as in Table \ref{Coefficients for Application}. As we are in a homogeneous setting, we consider the same $\sigma$ as part of the multiplier for both individual and market shock.
\begin{table}[h]
    \centering
    \label{Model parameters in the interbank market model}
    \begin{adjustbox}{width=\textwidth}
        \begin{tabular}{|*{12}{c|}}
        \hline
        General model & $A_k$  & $F_k$ &  $B_k$ & $H_k$ & $\eta_k$ & $\hQ_k$   &$Q_k$ &$S_k$ & $R_k$ &$\sigma_0$ &$\sigma_k$\\
        \hline
        Interbank market model & $-a$  & $a $   & $1$    & 1  & $0$     & $\hat{q}$ &$q$  &$\xi$  &$1$ &$\sigma \rho$ &$\sigma \sqrt{1-\rho^2}$\\
        \hline
        \end{tabular}
    \end{adjustbox}
    \caption{Model parameters in the interbank market model.}
    \label{Coefficients for Application}
\end{table}
Then, we provide an interpretation for each parameter based on \citet{Carmona2013} and \citet{YY2023}. Next, we introduce the solution to the problem based on the theorems presented in Section \ref{sec: sols to inf}. We solve the system of control coefficients numerically and provide an analytical solution for a simpler case. Subsequently, we define the total and conditional default probability and proceed to address it utilizing respectively the classical and stochastic Fokker-Planck equations, drawing inspirations from \citet{wille2004new} and \citet{Carmona2013}, by considering the first hitting time of the market and agent state falling below a default threshold. Next, we employ the forward explicit finite differences method to tackle the probability of default concerning both the individual bank and the entire market. Then, we examine the influence of parameter variations on the probabilities of default. Notably, we consider the effects of the common factor $\rho$, risk-sensitivity, and liquidity parameters on the reserve of the bank and of the market at equilibrium. In the end, a comprehensive analysis of the bank's conditional probability of default will follow, considering the presence of two distinct trajectories of common noise. 

The terminology employed in this section pertains to interbank transactions. Specifically, the concepts of lending and borrowing from the central bank correspond to the acquisition and sale, respectively, of government-issued bonds. Moreover, the transaction rate denotes the controlled measures that a bank undertakes in this process to effectively manage reserve prerequisites, enhance liquidity, and fulfill regulatory mandates.

Remark that despite our efforts to obtain data for parameter calibration, we were unable to access the necessary information due to the confidentiality protocols regarding monetary reserves held by various institutions. Consequently, we will assign values to parameters inspired from \citet{Carmona2013} in the application sections.

\subsection{Finite-Population Model}
\subsubsection{Dynamics}\label{sub: dynamics}
On the probability space $(\Omega, \pmb{F}, (\mathcal{F}^{[N]}_t)_{t\in \mfT}, \mathbb{P})$, for bank $i, i \in \mfN$, the finite population dynamics is given as
\begin{equation}\label{dx_t_app}
    dx_t^i=\{a (x_t^{[N]}-x_t^i)+u_t^i+b(t)\}dt+\sigma \sqrt{1- \rho^2} dw_t^i+\sigma \rho dw_t^0
\end{equation}
where $t \in \mfT$. We denote the variable $x_t^i \in \mathbb{R}$ as the log-reserve of the bank at the time $t$. The transaction rate $u_t^i \in \mathbb{R}$ represents the money that the bank lends to or borrows from the central bank during the market activity at each time $t$. As in the general model, the market shock is characterized by $w_t^0 \in \mathbb{R}$ which is independent of the shock received by the bank $w_t^i \in \mathbb{R}$ through $t \in \mfT$.The average log-reserve of all the banks in the market at the time $t$ represents the market state and is captured by $x_t^{[N]} \in \mathbb{R}$ with dynamics
\begin{equation}\label{aver_dyn_app}
    d x_t^{[N]} = (u_t^{[N]}  + b(t))dt + \frac{\sigma \sqrt{1- \rho^2}}{N}\sum_{i \in \mc{I}} dw_t^i+ \sigma \rho dw^0_t
\end{equation}
\begin{equation}
x_t^{[N]}=\frac{1}{N}\sum_{i \in \mc{I}} x_t^i,\quad
u_t^{[N]}=\frac{1}{N}\sum_{i \in \mc{I}} u_t^i .
\end{equation} 

In addition, the parameter $a \in \mathbb{R}$ is the mean reversion rate of the bank's reserve towards the market state. The liquidity of the bank before market activity at each time $t$ is represented by $b(t)$. The volatility of the log-reserve of the bank with respect to its own local shock (underlying uncertainty source) is denoted by $\sigma \rho \in \mathbb{R}$. The volatility of the log-reserve with respect to the global shock that affects the market (i.e. the macroeconomic factors), is characterized by $\sigma \sqrt{1-\rho^2} \in \mathbb{R}$. As can be seen from above equation, an instantaneous coefficient $0 \leq \rho \leq 1$ is a common multiplier factor for the shock delivered by the bank itself and by the environment.

In addition, the equivalent assumptions and $\sigma$-fields as for the general model in Section \ref{sec: Finite Population Model} are considered.

\subsubsection{Cost Functional}\label{sub:cost functional fin}
The operational cost of a representative bank to be minimized is modeled by the functional 
\begin{equation}\label{Cost Functional_app_finite}
  J^{i,[N]} = \gamma\log\mb{E} \left\{\exp\left(\frac{1}{\gamma}\bigl(g(x_T^i,x_T^{[N]})+\int_0^Tf(x^i,x_t^{[N]},u_t^i)dt\bigr)\right)\right\} 
\end{equation} where 
\begin{equation}
    g(x_T^i,\bar{x}_T)=\frac{1}{2} (x_T^{[N]}-x_T^i)^2 \hat{q}
\end{equation}
\begin{equation} 
    f(x^i,x_t^{[N]},u_t^i)=\frac{1}{2} \left\{ (x_t^{[N]}-x_t^i)^2 q-2(x_t^{[N]} -x_t^i) \xi u_t^i+(u_t^i)^2 \right\}
\end{equation}
with $\hat{q}$, $q$, $\xi \in \mathbb{R}$.

The costs specified for the bank are composed of the terminal $ g(x_T^i,x_T^{[N]})$ and running $f(x^i,x_t^{[N]},u_t^i)$ costs. The degree of risk-sensitivity for bank-$i$ is represented by $\tfrac{1}{\gamma} \in (0,\infty)$ and models a risk-averse behaviour. Specifically, the larger $\tfrac{1}{\gamma}$, the more risk-averse is the bank. In the limit, where $\tfrac{1}{\gamma} \rightarrow 0$, the cost functional reduces to a risk-neutral one. The terminal cost consists of only a quadratic term associated with the risk undertaken in connection with the market state at the time $T$. There are three running cost components associated with the state of the bank and the market state as well as the control action at time $t$. When the log-reserve of the bank significantly differs from the market state, the penalty for deviation is conveyed through the quadratic cost $(x_t^{[N]}-x_t^i)^2 q$. The bank's incentive to borrow from or lend to the central bank in relation to the market state is modeled by $-2(x_t^{[N]} -x_t^i) \xi u_t^i$. Remark that $\xi > 0$ represents the bank's borrowing or lending fees for the adjustments in the monetary reserve, guided by the control $u_t^i$. In other words, if $x_t^{[N]}>x_t^i$, the bank wishes to have $u_t^i > 0$ (i.e. borrowing money). Then, the borrowing cost will be added to the running cost (i.e. $-2(x_t^{[N]} -x_t^i) \xi u_t^i > 0$).  If $x_t^{[N]}<x_t^i$, the bank wishes to have $u_t^i < 0$ (i.e. lending money). Subsequently, the gain from lending will be deduced from the running cost  (i.e. $-2(x_t^{[N]}-x_t^i) \xi u_t^i < 0$). The transaction cost or market friction is modeled by the quadratic term $(u_t^i)^2$.

In short, through the trading horizon $\mfT$, a representative bank wants to minimize its expected cost \eqref{Cost Functional_app_finite} while being risk-averse and its log-reserve is governed by \eqref{dx_t_app}.

\subsection{Infinite-Population Model}
\subsubsection{Dynamics}\label{sub: dynamics_infinite}
In the infinite population limit, where $N \rightarrow \infty$ (see Section \ref{sec: Infinite Population Model}), the log-reserve of the bank  $i \in \mfN$ at the time $t$ satisfies
\begin{equation}\label{dx_inf_app} 
    dx_t^i=\{a (\bar{x}_t-x_t^i)+u_t^i+b(t)\}dt+\sigma \sqrt{1- \rho^2} dw_t^i+\sigma \rho dw_t^0
\end{equation}
where the mean field, $\bar{x}_t = \lim_{N \rightarrow \infty} \tfrac{1}{N}\sum_{i \in \mc{I}} x_t^i$, represents the limiting market state satisfying 
\begin{equation}\label{dMF_inf_app} 
d\bar{x}_t = (\bar{u}_t  + b(t))dt +\sigma \rho dw_t^0
\end{equation}
with 
\begin{equation}
\bar{u}_t=\lim_{N \rightarrow \infty}\frac{1}{N}\sum_{i \in \mc{I}} u_t^i.
\end{equation} 
From this point forward, we will refer the state $\bar{x}_t$ as the market state. Other coefficients and variables are the same as the ones defined in Section \ref{sub: dynamics}.

Remark that for two banks in the market, bank-$i$ and bank-$j$ with $i,j \in \mfN$ such that $ i \neq j$. As the Brownian motions $w_t^i$, $w_t^j$ and $w_t^0$ are independent of each other but the bank states $x_t^i$ and $x_t^j$ are influenced by the same common noise $w_t^0$, $corr(x_t^i,x_t^j)=(\sigma\rho)^2$. In other words, the banks are correlated. In addition, for any bank-$i$, $i \in \mfN$, $corr(x_t^i,\bar{x}_t)=(\sigma\rho)^2$.

The same assumptions and $\sigma$-fields as for the general model in Section \ref{sec: Infinite Population Model} in dimension-reduced form are considered. 

\subsubsection{Cost Functional}
The operational cost of a representative bank that needs to be minimized is structured using the identical parameters and variables as described in \ref{sub:cost functional fin}. The only alteration is the substitution of the state $x_t^{[N]}$ with the market one $\bar{x}_t$ to account for the scenario involving an infinite population of small banks. This cost is represented by the functional 
\begin{equation}\label{Cost Functional_app}
   J^{i,\infty} = \gamma\log\mb{E} \left\{\exp\left(\frac{1}{\gamma}\bigl(g(x_T^i,\bar{x}_T)+\int_0^Tf(x^i,\bar{x}_t,u_t^i)dt\bigr)\right)\right\} 
\end{equation} where 
\begin{equation}
    g(x_T^i,\bar{x}_T)=\frac{1}{2} (\bar{x}_T-x_T^i)^2 \hat{q}
\end{equation}
\begin{equation} 
    f(x^i,\bar{x}_t,u_t^i)=\frac{1}{2} \left\{ (\bar{x}_t-x_t^i)^2 q-2(\bar{x}_t-x_t^i) \xi u_t^i+(u_t^i)^2 \right\}.
\end{equation}

In order to ensure the convexity of the cost functional, we impose the equivalent conditions as in Assumption \ref{ass:ConvCond}, i.e.
\begin{equation}
    \hat{q}\geq0,\quad q-\xi^2\geq0.
\end{equation}

\subsection{Optimal Transaction Rate for Infinite-Population Model}
From Theorem \ref{Optimal control risk sensitive} and the model described by \eqref{dx_t_app}, \eqref{dMF_inf_app} and \eqref{Cost Functional_app}, the optimal transaction rates $\{u^{i,*}, i=1,2,\dots\}$ for individual banks achieving a Nash equilibrium are characterized by
\begin{align}\label{u_app}
    u_t^{i,*}=-\left[(\Pi_t+\xi) x_t^i+(\Lambda_t-\xi)\bar{x}_t+\Upsilon_t \right], \text{\hspace{5 mm}$i \in \mfN$}
\end{align}
where 
\begin{align}
&\begin{cases}\label{Pi_app}
    &d\Pi_t=\Bigl\{\left(1-\frac{\sigma^2}{\gamma}\right)\Pi_t^2+\left(2a+2\xi-\frac{2}{\gamma}(\sigma \rho)^2\Lambda_t\right)\Pi_t\\&\hspace{1.2cm}-\frac{1}{\gamma}(\sigma \rho)^2(\Lambda_t)^2+\xi^2-q\Bigr\}dt\\
    &\Pi_T=\hat{q}\\
\end{cases}\\
&\begin{cases}\label{Lambda_app}
    &d \Lambda_t=\Bigl\{(1-\frac{1}{\gamma}(\sigma \rho)^2)\Lambda_t^2+\left( a+2\Pi_t+\xi -\frac{\sigma^2}{\gamma} \Pi_t-\frac{1}{\gamma}(\sigma \rho)^2 \Delta_t \right)\Lambda_t\\&\hspace{1.2cm}-\bigl(\xi +a +\frac{1}{\gamma}(\sigma \rho)^2\Delta_t\bigr)\Pi_t+q-\xi^2\Bigr\}dt\\
    &\Lambda_T=-\hat{q} 
\end{cases}\\
&\begin{cases}\label{Upsilon_app}
    &d\Upsilon_t= \Bigl\{\left(\Pi_t+\Lambda_t+a+\xi-\frac{\sigma^2}{\gamma} \Pi_t-\frac{1}{\gamma}(\sigma \rho)^2 \Lambda_t\right)\Upsilon_t-(\frac{1}{\gamma}(\sigma \rho)^2\Gamma_t+b(t))\Pi_t\\&\hspace{1.2cm}-(\frac{1}{\gamma}(\sigma \rho)^2\Gamma_t+b(t))\Lambda_t \Bigr\}dt\\ 
    &\Upsilon_T=0
\end{cases}\\
&\begin{cases}\label{Delta_app}
    &d\Delta_t=\Bigl\{-\frac{1}{\gamma}(\sigma \rho)^2\Delta_t^2- 
    2 (\Pi_t + \Lambda_t +\frac{1}{\gamma}(\sigma \rho)^2 \Lambda_t)\Delta_t+(1-\frac{\sigma^2}{\gamma})\Lambda_t^2\\&\hspace{1.2cm} - 2(
    \xi +a ) \Lambda_t- q + \xi^2\Bigr\}dt\\
    &\Delta_T=-\Lambda_T
\end{cases}\\
&\begin{cases}\label{Gamma_app}
    &d\Gamma_t=\Bigl\{\left(\Pi_t + \Lambda_t  - \frac{1}{\gamma}(\sigma\rho)^2(\Lambda_t+\Delta_t)\right)\Gamma_t+\Bigl(\Upsilon_t - b(t) -\frac{\sigma^2}{\gamma}\Upsilon_t \Bigr)\Lambda_t \\&\hspace{1.2cm}+ \Bigl(- \xi - a  +\Delta_t-\frac{1}{\gamma}(\sigma \rho)^2\Delta_t\Bigr)\Upsilon_t   -b(t)\Delta_t \Bigr\}dt\\
    &\Gamma_T=0
\end{cases}\\
&\begin{cases}\label{Psi_app}
    &d\Psi_t=\Bigl\{\left(1-\frac{\sigma^2}{\gamma}\right)\Upsilon_t^2-\frac{1}{\gamma}(\sigma \rho)^2\Gamma_t^2-\sigma^2 \Pi_t-2(\sigma\rho)^2\Lambda_t\\&\hspace{1.2cm}+2\left(\Gamma_t -\frac{1}{\gamma}(\sigma \rho)^2\Gamma_t-b(t)\right)\Upsilon_t-(\sigma \rho)^2\Delta_t-2b(t)\Gamma_t \Bigr\}dt\\
    &\Psi_T=0.
\end{cases}
\end{align}
The resulting market transaction rate $\bar{u}_t^*$ in the infinite-population model is given by
\begin{align}\label{u_bar_app}
    \bar{u}_t^* =-\left[(\Pi_t+\xi) \bar{x}_t+(\Lambda_t-\xi )\bar{x}_t+\Upsilon_t\right].
\end{align}
Consequently, the following dynamics for individual banks and the market state emerge
\begin{equation}\label{dx_op}
    dx_t^i=\{ \left( a  -\Lambda_t + \xi  \right) \bar{x}_t-\left(a+\Pi_t+\xi\right)x_t^i-\Upsilon_t +b(t)\}dt+\sigma \sqrt{1- \rho^2} dw_t^i+\sigma \rho dw_t^0
\end{equation}
\begin{equation}\label{dMF_op}
    d\bar{x}_t = (\bar{A}_t\bar{x}_t + \bar{m}_t)dt+\sigma \rho d w_t^0
\end{equation}
where
\begin{align}
    \bar{A}_t&=-\Pi_t- \Lambda_t \label{A_bar_app}\\
    \bar{m}_t&=b(t)-\Upsilon_t.
\end{align}

We provide an example of a simplified optimization problem and provide analytically the optimal transaction rate of the bank and of the market.

\subsubsection{Analytical Solutions for a Specific Scenario}\label{simplified ex}
It is interesting to explore the analytical solution to a special case of the model under consideration. We will give an example there. Consider the question with parameters of value $1$ except $a = 10$ and we are interested in the analytical solution of the optimal transaction rate of the bank and of the market.

Consider the dynamics of the bank as
\begin{equation}\label{dx_t_app2}
    dx_t^i=\{10(\bar{x}_t-x_t^i)+u_t^i+1\}dt+dw_t^0
\end{equation}
\begin{equation}\label{dMF_t_app2}
d\bar{x}_t = (\bar{u}_t  + 1)dt +dw^0_t
\end{equation}
where
\begin{equation}
\bar{x}_t=\frac{1}{N}\sum_{i \in I} x_t^i \in \mathbb{R},\quad
\bar{u}_t=\frac{1}{N}\sum_{i \in I} u_t^i \in \mathbb{R}.
\end{equation}
Moreover, consider the cost functional is given by
\begin{equation}\label{Cost Functional_app2}
    \lim_{N\rightarrow \infty} J^{i,[N]} = \log\mb{E} \left\{\exp\left(\bigl(g(x_T^i,\bar{x}_T)+\int_0^Tf(x^i,\bar{x}_t,u_t^i)dt\bigr)\right)\right\} 
\end{equation} where 
\begin{equation}
    g(x_T^i,\bar{x}_T)=\frac{1}{2} (\bar{x}_T-x_T^i)^2 
\end{equation}
\begin{equation} 
    f(x^i,\bar{x}_t,u_t^i)=\frac{1}{2} \left\{ (\bar{x}_t-x_t^i)^2-2(\bar{x}_t-x_t^i) u_t^i+(u_t^i)^2 \right\}.
\end{equation}
\begin{proposition}
The optimal control of the LQG risk-sensitive system with the dynamics \eqref{dx_t_app2} and the cost functional \eqref{Cost Functional_app2} is given by
\begin{align}\label{u_t with values}
    u_t^{i,*} = \left(1-\frac{22\exp{(22t)}}{\exp{(22t)}-23\exp{(22)}}\right)(\bar{x}_t-x_t^i).
\end{align}
\end{proposition}
\begin{proof}
Considering the optimal control of the bank based on the equation \eqref{u_app} with defined parameters, namely
\begin{align}
    u_t^{i,*}=-\left[(\Pi_t+1) x_t^i+(\Lambda_t-1)\bar{x}_t+\Upsilon_t \right].
\end{align}
Based on the Section \ref{Optimal control risk sensitive} and the parameters defined in this specific case, for the system of ordinary differential equations (ODES) for control coefficients $\Pi_t, \Lambda_t, \Upsilon_t, \Delta_t, \Gamma_t$ and $\Psi_t$, we can see that $\Pi_t=-\Lambda_t$ leading
\begin{align}
&\begin{cases}
    &d\Pi_t=-d\Lambda_t=22\Pi_t+\Pi_t^2dt\\
    &\Pi_T=-\Lambda_T=1.
\end{cases}
\end{align}
By solving this ODE,
\begin{align}
    \Pi_t=\frac{-22\exp{(22c_1+22t)}}{\exp{(22c_1+22t)}-1}, c_1\in \mathbb{R}.
\end{align}
We can then solve $c_1$ by considering the terminal condition
\begin{align}
    \Pi_T=\frac{-22\exp{(22c_1+22T)}}{\exp{(22c_1+22T)}-1}=1.
\end{align}
We obtain
\begin{align}
    \Pi_t=\frac{-22\exp{(22t)}}{\exp{(22t)}-23\exp{(22)}}.
\end{align}
For $\Upsilon_t$,
\begin{align}
&\begin{cases}
    &d\Upsilon_t=11\Upsilon_tdt\\
    &\Upsilon_T=0.
\end{cases}
\end{align}
However, when solving the above ODE, we obtain
\begin{align}
    \Upsilon_t=c_2 \exp{(11t)},\,\, c_2 \in \mathbb{R}
\end{align}
which at terminal time, $T$, is equal to
\begin{align}
    \Upsilon_T=c_2 \exp{(11T)}=0.
\end{align}
Thus, $c_2=0$ and $\Upsilon_t=0$.\\
As a result, the optimal control is
\begin{align}\label{u_t with values simplified}
    u_t^{i,*}&= (\Pi_t+1)(\bar{x}_t-x_t^i)\nnb
           \\& = \left(1-\frac{22\exp{(22t)}}{\exp{(22t)}-23\exp{(22)}}\right)(\bar{x}_t-x_t^i).
\end{align}
\end{proof}

The rest of the paper delves into the analysis of the likelihood of default concerning both the bank's and the market's log-reserve in the equilibrium which we refer thereby as the individual and systemic defaults. The interdependence of the banks is articulated in the Section \ref{sub: dynamics_infinite}. The correlation between banks imposes a risk to the entire market, identified as the systemic risk. Namely, the systemic risk refers to the probability of the market default given such relationship between banks. This scrutiny is supplemented by analyzing the effects of various model parameters on the default probability. Additionally, the influence of particular trajectories of common noise on default is showcased.

\subsection{Individual Default and Systemic Risk}\label{sec: Individual Default and Systemic Risk}

In this section, we investigate the default probability of a representative bank $i \in \mfN$ and the systemic risk. We first define these notions by the likelihood of the respective states dipping below a specific threshold based on \citet{Carmona2013}. We first derive the Fokker-Planck equation that the respective probability density function satisfies in each case based on \citet{e2019applied} and \citet{Carmona2013}. Then, to compute the default probabilities, we use the analysis of first hitting time and the obtained Fokker-Planck equations. We refer to \citet{wille2004new} for the calculation of the default probability of a general diffusion process using this method. 

\subsubsection{Definition of Default Probability and First Hitting Time}\label{def: default}

The default event can be interpreted as an occurrence wherein either the market or the agent fails to fulfill the minimum reserve requirements stipulated by the regulator or the conditions necessary to sustain the functionality of daily operations. As \citet{Carmona2013}, we consider the same constant default threshold for both the market and the agent. In this context, the market default can also be seen as the default of a representative bank that holds the limiting average of the log-reserves of all banks.

We define the probability of a systemic default event as the likelihood of the minimum market state, governed by the dynamics described in equation \eqref{dMF_op}, falling below the default threshold $\theta$ over the time horizon $\mfT$ as 
\begin{equation}\label{def: MFD}
    \mathds{P}(\min_{0 \leq t \leq T} \bar{x}_t \leq \theta).
\end{equation}
We define the probability of the default event of bank-$i$ as the likelihood of the bank's log-reserve, governed by the dynamics described in equation \eqref{dx_op}, falling below the default threshold $\theta$ over the time horizon $\mfT$ as
\begin{equation}\label{def: XD}
    \mathds{P}(\min_{0 \leq t \leq T} x_t^i \leq \theta).
\end{equation}
We define the conditional probability of the default event of bank-$i$ as the likelihood of the bank's log-reserve, governed by the dynamics described in equation \eqref{dx_op}, falling below the default threshold $\theta$ over the time horizon $\mfT$  given $(\mathcal{F}_t^{0})_{t\in \mfT}$ as  
\begin{equation}\label{def: XCD}
    \mathds{P}(\min_{0 \leq t \leq T} \bar{x}_t \leq \theta | \mathcal{F}_t^0).
\end{equation}
In this scenario, we will conduct an in-depth analysis of the individual default probability while considering a specific trajectory of common noise. This probability provides a clearer insight into the default event of bank-$i$ within the context of observed market shocks. 

Over the time horizon $\mfT$, the event that the minimum of the set of states governed by the corresponding dynamics falls below the threshold $\theta$ is equivalent to the first hitting time of the state when it reaches the predefined threshold $\theta$ \citep{wille2004new}.
Let us define the first hitting time for bank-$i$ as $t^*_{x^i} \coloneqq \min_{x^i_t = \theta} t$. Then, we have
\begin{equation}
    \mathds{P}(\min_{0 \leq t \leq T} x_t^i \leq \theta)=\mathds{P}(t^*_{\bar{x}} \leq T).
\end{equation}Similarly, we define the first hitting time for the mean-field as $t^*_{\bar{x}}\coloneqq \min_{\bar{x}_t = \theta} t$. The equivalent probability for the systemic event is then given by
\begin{equation}
    \mathds{P}(\min_{0 \leq t \leq T} \bar{x}_t \leq \theta) = \mathds{P}(t^*_x \leq T).
\end{equation}
The conditional default probability of a representative bank given $(\mathcal{F}_t^{0})_{t\in \mfT}$ is equvalently expressed as in 
\begin{equation}
    \mathds{P}(\min_{0 \leq t \leq T} x_t^i \leq \theta | \mathcal{F}_T^0)=\mathds{P}(t^*_{x_t^i} \leq T | \mathcal{F}_T^0).
\end{equation}

\subsubsection{Fokker-Planck Equation for Systemic Risk} 
The probability of market default is considered first, and then a similar approach is applied to analyze the probability of default for the individual bank. The analysis begins by investigating through the time horizon $\mfT$ the event of the minimum market state reaching a certain value at a specific time to determine the probability of the first hitting time. If the minimum market state reaches the predetermined threshold, the default event occurs. The probability of the default may be computed from the survival probability density function $\bar{p}(\bar{x},t)$ which captures the event in which the market default is not occurred through out the time horizon $\mfT$. In order to find $\bar{p}(\bar{x},t)$, as $\bar{x}_t$ is stochastic we employ the Fokker-Planck method  based on \citet{wille2004new}, where this method is used to calculate the probability of default of a diffusion process.

We solve first the Fokker-Planck partial differential equation (PDE) with respective boundaries for the probability of the systemic survival described as
\begin{align}
    \frac{\partial \bar{p}(\bar{x},t)}{\partial t} & = - \frac{\partial}{\partial\bar{x}}[(\bar{A}_t \bar{x} + b(t) - \Upsilon_t) \bar{p}(\bar{x},t)] + \frac{(\sigma \rho)^2}{2} \frac{\partial^2 \bar{p}(\bar{x},t)}{\partial \bar{x}^2}\nnb \\
     &= -\bar{A}_t \frac{\partial}{\partial\bar{x}}[ \bar{p}(\bar{x},t)]-(\bar{A}_t \bar{x} + b(t) - \Upsilon_t) \frac{\partial}{\partial\bar{x}}[ \bar{p}(\bar{x},t)] + \frac{(\sigma \rho)^2}{2} \frac{\partial^2 \bar{p}(\bar{x},t)}{\partial \bar{x}^2}. \label{FPMF}
\end{align}

 We consider the absorbing boundaries allowing $p(\bar{x},t)$ to vanish if it breaks the threshold. Moreover, we impose $\mathds{P}(\bar{x}=\infty) = 0$ almost surely. In addition, we define the boundary condition at initial time $t=0$ according to a standard normal distribution, denoted as $\mc{N}(0,1)$. Hence, the boundaries are 
\begin{equation}\label{FP MF boundaries}
    \bar{p}(\theta,t)=0, \hspace{1cm} 
    \bar{p}(\infty,t)=0, \hspace{1cm} 
    \bar{p}(\bar{x},0) \sim \mc{N}(0,1)\text{\hspace{1mm} with $\bar{x} \in (\theta, \infty)$}.
\end{equation}

It should be noted that the existence of the probability density function $p(\bar{x}_t,t)$ assumes that the market state does not break the threshold at time $t$. Therefore, the probability of the event that the first hitting time is beyond $T$ can be determined by integrating $p(\bar{x}_t, T)$ over all possible $\bar{x}$ within the boundary of existence. Hence,
\begin{equation}\label{last_eq_MFP}
    \mathds{P}(t^*_{\bar{x}} > T) = \int_a^\infty  \bar{p}(\bar{x},T) d\bar{x}.
\end{equation}

Hence, the probability of the event that the first hitting time is within the time interval $\mfT$ is given by 
\begin{align}\label{last_step_MFFP}
    \mathds{P}(t^*_{\bar{x}} \leq T) & = 1 - \int_a^\infty  \bar{p}(\bar{x},T) d\bar{x}.
\end{align}

\subsubsection{Fokker-Planck Equation for Individual Default Probability}
The probability of default of a representative bank can be solved in a similar way. We consider the event of the bank's state reaching a certain value at a specific time to determine the probability of the first hitting time. To keep notation concise, we adopt a matrix representation. The joint dynamics of bank $i$, \eqref{dx_op}-\eqref{dMF_op}, and the market state \eqref{dMF_op} is given by 
\begin{align}
    d\pmb{X}_t^i &= \begin{bmatrix} \pmb{\upsilon}_1 \\ \pmb{\upsilon}_2\end{bmatrix}+ \pmb{\Sigma} d\pmb{W}_t^i
\end{align}
where
\begin{equation}
\begin{bmatrix} \pmb{\upsilon}_1 \\ \pmb{\upsilon}_2\end{bmatrix} =\begin{bmatrix} -\Pi_t- a -\xi  && -\Lambda_t + a +\xi \\
                0 && -\Pi_t - \Lambda_t
\end{bmatrix}
\begin{bmatrix} x^i\\\bar{x}\end{bmatrix}+\begin{bmatrix} b(t)-\Upsilon_t \\ b(t)-\Upsilon_t \end{bmatrix}
\end{equation}
\begin{equation}
\pmb{\Sigma}=
\begin{bmatrix}\sigma \sqrt{(1-\rho^2)} && \sigma \rho\\
               0 && \sigma \rho 
\end{bmatrix},
\quad
\pmb{W}_t^i 
=\begin{bmatrix} w^i_t \\ w^0_t\end{bmatrix}.
\end{equation}

The analysis begins by examining the joint state of bank-$i$ and the market, denoted by $\pmb{X}^i$, reaching a certain set of values at a specific time to determine the probability of the first hitting time. The distribution of this state is described by the survival probability density function $p(\pmb{X}^i,t)$ satisfying the Fokker-Planck equation
\begin{align}
    \frac{\partial p(\pmb{X}^i,t)}{\partial t} &= - \frac{\partial \pmb{\upsilon}_1 p(\pmb{X}^i,t)}{\partial x^i} - \frac{\partial \pmb{\upsilon}_2 p(\pmb{X}^i,t)}{\partial \bar{x}} \nnb\\&+ \frac{1}{2}\Bigl\{\sigma^2 \frac{\partial^2p(\pmb{X}^i,t)}{\partial (x^i)^2}+\sigma^2 \rho^2\frac{\partial^2p(\pmb{X}^i,t)}{\partial (x^i)(\bar{x})}+\sigma^2 \rho^2\frac{\partial^2p(\pmb{X}^i,t)}{\partial (\bar{x})(x^i)}+ \sigma^2 \rho^2 \frac{\partial^2p(\pmb{X}^i,t)}{\partial \bar{x}^2} \Bigr\} \label{FPA}
\end{align}
subject to the boundary conditions 
\begin{equation*}
    p\left(\begin{bmatrix}
        \theta \\
        \bar{x}
    \end{bmatrix},t\right)=0, \text{\hspace{3cm}}
    p\left(\begin{bmatrix}
        \infty \\
        \infty
    \end{bmatrix},t\right)=0,
\end{equation*}
\begin{equation}\label{FPX boundaries}
    p(\pmb{X},0) \sim \mc{N} \left(\begin{bmatrix}
        0 \\0
    \end{bmatrix},\begin{bmatrix}
        1&0 \\0 & 1\end{bmatrix} \right) \text{\hspace{1mm} with $(x,\bar{x}) \in (\theta, \infty)\times(-\infty,\infty)$}.
\end{equation}

We consider the absorbing boundaries make $p(\pmb{X}^i,t)$ vanish if it breaks the threshold. Moreover, we impose $\mathds{P}(\pmb{X}^i=\infty) = 0$ almost surely. 
The boundary condition at initial time, $t=0$, is defined according to a bivariate standard normal distribution with a zero correlation matrix.

We note that the existence of the probability density function $p(\pmb{X}^i,t)$ assumes that $\pmb{X}^i$ does not break the threshold at time $t$. Therefore, the probability of the bank's survival given a specific market state at time $T$ can be determined by integrating $p(\pmb{X}^i,T)$ over all possible $\bar{x}$ within the boundary of existence. Hence,
\begin{align}
    p(x^i,T) & =\int_{-\infty}^{\infty} p(\pmb{X}^i,T) d{\bar{x}}.
\end{align}

Following a similar procedure as used for determining the market default probability, we can determine the probability of the bank experiencing default within the time interval $\mathcal{T}$ as
\begin{align}\label{last_step_XFP}
    \mathds{P}(t^*_{x^i} \leq T) &= 1 - \int_a^\infty  p(x^i,T) d{x^i}.
\end{align}

\subsubsection{Stochastic Fokker-Planck Equation for Individual Default Probability under Specific Common Shock}
The conditional probability of default of a representative bank consists of analyzing the default event given the common noise. The distribution of the conditional default of the bank may be calculated using the survival probability density function $p(x^i,t | w^0_t)$, which, in turn, can be computed using the Fokker-Planck method as in the previous section. However, rather than examining a classical PDE as discussed in the previous section, our focus now shifts to solving a stochastic PDE to take the filtration $(\mathcal{F}_t^{0})_{t\in \mfT}$ into consideration.

For the agent's dynamics \eqref{dx_op} with the optimal control \eqref{u_app}, the stochastic Fokker-Planck equation generating  $p(x^i ,t | w^0_t)$ is given by 
\begin{align}
    \partial p(x^i,t | \mathcal{F}_t^0) = \biggl\{&- \frac{\partial\{(-\xi-a-\Pi_t)x^i+(a-\Lambda_t+\xi )\bar{x}_t+b(t)-\Upsilon_t\}p(x^i,t | \mathcal{F}_t^0)}{\partial x_t^i}  \nnb \\&+ \frac{\sigma^2(1-\rho^2)}{2} \frac{\partial^2 p(x^i,t | \mathcal{F}_t^0)}{\partial (x_t^i)^2}\biggr\} dt- \sigma^2 \rho^2 \frac{\partial{p(x^i,t | \mathcal{F}_t^0)}}{\partial{x^i}} dw_t^0 \label{FPCA}
\end{align}
with the boundary conditions 
\begin{equation}\label{CFP boundaries}
    p(\theta,t | \mathcal{F}_t^0)=0, \hspace{0.5cm} 
    p(\infty,t | \mathcal{F}_t^0)=0, \hspace{0.5cm} 
    p(x^i,0 | \mathcal{F}_0^0) \sim \mc{N} (0,1) \text{\hspace{1mm} with $x \in (\theta, \infty)$}.
\end{equation}

Following a similar procedure as in previous sections, the conditional probability of the bank being defaulted within the time interval $\mfT$ is computed via
\begin{align}\label{last_step_CXFP}
    \mathds{P}(t^*_{x^i} \leq T | \mathcal{F}_T^0) &= 1 - \int_a^\infty  p(x^i,T | \mathcal{F}_T^0) d{x^i}.
\end{align}

\subsection{Numerical Experiments}\label{numerical methods}
Given the complexity inherent in specifying the probability of default based on the Fokker-Plack equations, we employ numerical techniques to adeptly tackle various aspects. Thisincludes solving the system of ODEs that the coefficients of optimal control satisfy and discerning both systemic and bank-specific conditional and unconditional defaults. We use numerical solutions to find the probability of default using the Fokker-Planck equations. Additionally, we carry out a sensitivity analysis by integrating coefficient values into the equation.

\subsubsection{Numerical Method for Control Coefficients}\label{NM Control}
To achive this goal, we  utilize the discretization of the time interval $\mathcal{T}$ into smaller segments $\pmb{\Delta t}$. Then, for each coefficient of the optimal control (i.e. $\Pi_t,\Lambda_t,\Upsilon_t,\Delta_t,\Gamma_t$ and $\Psi_t$), we discretize the respective ODE (i.e. \eqref{Pi_app}-\eqref{Psi_app}). For example, the discretization of the ODE that $\Pi_t$ satisfy is given by 
\begin{align}
&\begin{cases}
    &\frac{\Pi_{\pmb{\Delta (t+1)}}-\Pi_{\pmb{\Delta t}}}{\pmb{\Delta t}}=  \left(1-\frac{\sigma^2}{\gamma}\right)\Pi_{\pmb{\Delta t}}^2+\left(2a+2\xi-\frac{2}{\gamma}(\sigma \rho)^2\Lambda_{\pmb{\Delta t}}\right)\Pi_{\pmb{\Delta t}}-\frac{1}{\gamma}(\sigma \rho)^2(\Lambda_{\pmb{\Delta t}})^2+\xi^2-q\\
    &\Pi_T=\hat{q}.
\end{cases}
\end{align}
As the six ODEs, that the control coefficients satisfy, are coupled with each other, we solve a system of six ODEs to $\Pi_t,\Lambda_t,\Upsilon_t,\Delta_t,\Gamma_t,\Psi_t$. Specifically, we use backward differentiation with Python library \texttt{solve\_ivp} in \texttt{scipy.integrate}.

\subsubsection{Numerical Method for Fokker-Planck Equations}\label{FPEPD_D}
In order to solve the partial differential equations \eqref{FPMF}, \eqref{FPA} and \eqref{FPCA}, we need to first discretize them. To this purpose, we employ the forward explicit finite differences method. The probability of default is then calculated using numerical methods for integration.
\begin{itemize}
    \item[1.] \textit{Systemic Risk} 
    
    To solve for the probability of the market default \eqref{def: MFD} using the finite differences method, we employ a two-dimensional grid defined over the underlying variables time $t$ and market state $\bar{x}$. We discretize these variables within ranges $[t_0=0, \pmb{\Delta t},2\pmb{\Delta t},\dots,T]$ and $[\theta,\theta+\pmb{\Delta \bar{x}},\dots,\theta+\pmb{\bar{M} \Delta \bar{x}}]$, respectively, where $\pmb{\bar{M}} \in \mathbb{N}$ is chosen to be sufficiently large and the discretization of the variables $t$ and $\bar{x}$ are sufficiently small. At each grid point, we denote the probability as $\bar{p}_{\pmb{j}}^{\pmb{i}}$, where $\pmb{i} \in \mathbb{N}$ indicates the time position $\pmb{i} \pmb{\Delta t}$ and $\pmb{j} \in \mathbb{N}$  denotes the market state position  $\pmb{j} \pmb{\Delta \bar{x}}$. Consider the Fokker-Planck equation for the systemic survival \eqref{FPMF}, its respective discretization is
    \begin{align}\label{FPMFD}
        \bar{p}^{\pmb{i}}_{\pmb{j}}=\bar{p}^{\pmb{i}-1}_{\pmb{j}}\nnb&+\pmb{\Delta t}\Bigg\{\left(-\bar{A}^{\pmb{i}-1}(1+\bar{x}_{\pmb{j}})- b^{\pmb{i}-1} + \Upsilon^{\pmb{i}-1}\right) \left(\frac{\bar{p}^{\pmb{i}-1}_{\pmb{j}+1}-\bar{p}^{\pmb{i}-1}_{\pmb{j}-1}}{2\pmb{\Delta \bar{x}}} \right)\\&+\frac{(\sigma \rho)^2}{2}\left(\frac{\bar{p}^{\pmb{i}-1}_{\pmb{j}+1}-2\bar{p}^{\pmb{i}-1}_{\pmb{j}}+\bar{p}^{\pmb{i}-1}_{\pmb{j}-1}}{\pmb{\Delta \bar{x}}^2} \right)\Bigg\}
    \end{align}
    where $\bar{A}^{\pmb{i}}=-\Pi^{\pmb{i}} - \Lambda^{\pmb{i}} $. Remark that $\bar{A}$ depends only on time. The forward method begins with the initial point $\bar{p}^0_{\pmb{j}}$ which follows a standard normal distribution $\mc{N} (0,1)$ restricted on the space generated by the market state $(\theta,\theta+\pmb{\bar{M} \Delta \bar{x}}]$. Remark that in order to satisfy the absorbing condition at the threshold, we consider $\bar{p}^{\pmb{i}}_{\theta} = 0$ for all ${\pmb{i}}$, representing condition $\bar{p}(\theta,t)=0$. Then, the probability $\bar{p}$ is incremented at each time and market state step up to the end of the time horizon $T$.
    
    \item[2.] \textit{Default Probability of Bank-$i$} 
    
    To simplify the notation, the individual bank state will be denoted as $x$. To solve for the probability of the individual bank default \eqref{def: XD} using the finite differences method, we employ a three-dimensional grid defined over the underlying variables time $t$, bank state $x$ and market state $\bar{x}$. We discretize these variables within their respective as  $[t_0=0, \pmb{\Delta t},2\pmb{\Delta t},\dots,T]$, $[\theta,\theta+\pmb{\Delta x},\dots,\theta+\pmb{N_x \Delta x}]$ and the last one
    $[-\pmb{\bar{M} \Delta \bar{x}},(-\pmb{(\bar{M}+1) \Delta \bar{x}},\dots,\pmb{\bar{M} \Delta \bar{x}}]$, where $\pmb{N_x},\pmb{\bar{M}} \in \mathbb{N}$ are chosen to be sufficiently large the discretization of the variables $t$, $x$  and $\bar{x}$ are sufficiently small. At each grid point, we denote the probability as $p^{\pmb{i}}_{\pmb{j},\pmb{m}}$, where $\pmb{i}$ indicates the time position $\pmb{i}\pmb{\Delta t}$, $\pmb{j} \in \mathbb{N}$ denotes the bank state position  $\pmb{j}\pmb{\Delta x}$ and $\pmb{m}$ denotes the market state position  $\pmb{m}\pmb{\Delta \bar{x}}$. Subsequently, the discretization of the Fokker-Planck equation \eqref{FPA} that the default probability of an individual bank satisfies is given by
    \begin{align}\label{FPAD}
        p^{\pmb{i}}_{\pmb{j},\pmb{m}}&=p^{\pmb{i}-1}_{\pmb{j},\pmb{m}}+\pmb{\Delta t}\Biggl\{\Bigl(\left(\Pi^{\pmb{i}-1} + a +\xi\right)(1+x_{\pmb{j}})+\left(\Lambda^{\pmb{i}-1} - a -\xi \right)\bar{x}_{\pmb{m}}- b^{\pmb{i}-1} + \Upsilon^{\pmb{i}-1}\Bigr) \nnb\\& \times \left(\frac{ p^{\pmb{i}-1}_{{\pmb{j}+1},\pmb{m}}- p^{\pmb{i}-1}_{{\pmb{j}-1},\pmb{m}}}{2\pmb{\Delta \bar{x}}} \right)+\left(-\bar{A}^{\pmb{i}-1}(1+\bar{x}_{\pmb{m}})- b^{\pmb{i}-1} + \Upsilon^{\pmb{i}-1}\right) \left(\frac{p^{\pmb{i}-1}_{\pmb{j},\pmb{m}+1}-p^{\pmb{i}-1}_{\pmb{j},\pmb{m}-1}}{2\pmb{\Delta \bar{x}}} \right)\nnb\\&+\frac{1}{2}\biggr\{\sigma^2\left(\frac{p^{\pmb{i}-1}_{\pmb{j}+1,\pmb{m}}-2p^{\pmb{i}-1}_{\pmb{j},\pmb{m}}+p^{\pmb{i}-1}_{\pmb{j}-1,\pmb{m}}}{\pmb{\Delta \bar{x}}^2}\right)+\sigma^2\rho^2 \left(\frac{p^{\pmb{i}-1}_{\pmb{j},\pmb{m}+1}-2p^{\pmb{i}-1}_{\pmb{j},\pmb{m}}+p^{\pmb{i}-1}_{\pmb{j},\pmb{m}-1}}{\pmb{\Delta \bar{x}}^2}\right)\nnb\\&+2\sigma^2\rho^2 \left(\frac{p^{\pmb{i}-1}_{\pmb{j}+1,\pmb{m}+1}-p^{\pmb{i}-1}_{\pmb{j}+1,\pmb{m}-1}-p^{\pmb{i}-1}_{\pmb{j}-1,\pmb{m}+1}+p^{\pmb{i}-1}_{\pmb{j}-1,\pmb{m}-1}}{4\pmb{\Delta \bar{x}}\pmb{\Delta x}}\right)\biggl\}\Biggr\}
    \end{align}
    where $\bar{A}^{\pmb{i}}=-\Pi^{\pmb{i}} - \Lambda^{\pmb{i}}$. The forward method begins with the initial point $p^0_{\pmb{j},\pmb{m}}$ which follows a standard bivariate normal distribution $\mc{N} \left(\begin{bmatrix} 0 \\0 \end{bmatrix},\begin{bmatrix}
            1&0 \\0 & 1\end{bmatrix} \right)$ restricted on space generated by the market and bank's state, namely  $(\theta,\theta+\pmb{\Delta x},\dots,\theta+\pmb{N_x \Delta x}] \times [-\pmb{\bar{M} \Delta \bar{x}},(-\pmb{\bar{M}+1) \Delta \bar{x}},\dots,\pmb{\bar{M} \Delta \bar{x}}]$. Remark that in order to satisfy the absorbing condition at the threshold, we consider $p^{\pmb{i}}_{\theta, \pmb{m}} = 0$ for all ${\pmb{i}}$ and ${\pmb{m}}$, representing condition $p(\begin{bmatrix}
        \theta \\
        \bar{x}
    \end{bmatrix},t)=0$. Then, the probability $p$ is incremented at each time, bank state and market state step up to the end of the time horizon $T$.
    
    \item[3.] \textit{Conditional Default Probability of Bank-$i$ given a Specific Common Shock}
    
    To solve the probability of conditional bank default \eqref{def: XCD} using the finite differences method, we employ a two-dimensional grid defined over the underlying variables time $t$ and bank state $x$ under the filtration $(\mathcal{F}_t^{0})_{t\in \mfT}$. We discretize these variables within ranges $[t_0=0, \pmb{\Delta t},2\pmb{\Delta t},\dots,T]$ and $[\theta,\theta+\pmb{\Delta x},\dots,\theta+\pmb{N_x \Delta x}]$ respectively, where $\pmb{N_x} \in \mathbb{N}$ is chosen to be sufficiently large and the discretization of the variables $t$ and $x$ are sufficiently small. At each grid point, we denote the probability as $(p_{|\mathcal{F}_{\pmb{i}}^0})^{\pmb{i}}_{\pmb{j}}$, where $\pmb{i}$ indicates the time position $\pmb{i}\pmb{\Delta t}$ and $\pmb{j}$ denotes the bank state position  $\pmb{j}\pmb{\Delta x}$. In this specific case, as the common noise $w^0_t$, namely $w^{0, \pmb{i}\pmb{\Delta t}}$ in the discretization, is known at time $\pmb{i}\pmb{\Delta t}$, we consider the market state at each time $\pmb{i}\pmb{\Delta t}$ for all $x^{\pmb{i}\pmb{\Delta t}}$ from the discretization of its dynamics.
    \begin{equation}
        \bar{x}^{\pmb{i}\pmb{\Delta t}}=\bar{x}^{(\pmb{i}-1)\pmb{\Delta t}}+ (\bar{A}^{(\pmb{i}-1)\pmb{\Delta t}}\bar{x}^{(\pmb{i}-1)\pmb{\Delta t}} + \bar{m}^{(\pmb{i}-1)\pmb{\Delta t}})\pmb{\Delta t}+\sigma \rho w^{0, \pmb{\Delta t}} \text{\hspace{5 mm} for all $i$}
    \end{equation}
    where
    \begin{align}
        \bar{A}^{(\pmb{i}-1)\pmb{\Delta t}}&=-\Pi^{(\pmb{i}-1)\pmb{\Delta t}} - \Lambda^{(\pmb{i}-1)\pmb{\Delta t}} \\
        \bar{m}^{(\pmb{i}-1)\pmb{\Delta t}}&=b^{(\pmb{i}-1)\pmb{\Delta t}}-\Upsilon^{(\pmb{i}-1)\pmb{\Delta t}}
    \end{align}
    with starting point $\bar{x}^0$. For simplicity, we denote $\bar{x}^{\pmb{i}\pmb{\Delta t}}$ as $\bar{x}^{\pmb{i}}$ in the following text. Consider the Fokker-Planck equation for the conditional bank's survival \eqref{FPCA}, its respective discretization is
    \begin{align}\label{FPACD}
        (p_{|\mathcal{F}_{\pmb{i}}^0})^{\pmb{i}}_{\pmb{j}}&=(p_{|\mathcal{F}_{\pmb{i}-1}^0})^{\pmb{i}-1}_{\pmb{j}}\nnb\\&+\pmb{\Delta t}\Biggl\{\left(\left(\Pi^{\pmb{i}-1} + a +\xi\right)(1+x_{\pmb{j}})+\left(\Lambda^{\pmb{i}-1} - a -\xi \right)\bar{x}^{\pmb{i}-1}- b^{\pmb{i}-1} + \Upsilon^{\pmb{i}-1}\right) \nnb\\&\times \left(\frac{(p_{|\mathcal{F}_{\pmb{i}-1}^0})^{\pmb{i}-1}_{\pmb{j}+1}- (p_{|\mathcal{F}_{\pmb{i}-1}^0})^{\pmb{i}-1}_{\pmb{j}-1}}{2\pmb{\Delta x}} \right)- \sigma^2 \rho^2 \frac{{(p_{|\mathcal{F}_{\pmb{i}-1}^0})^{\pmb{i}-1}_{\pmb{j}+1}}-(p_{|\mathcal{F}_{\pmb{i}-1}^0})^{\pmb{i}-1}_{\pmb{j}-1}}{2\pmb{\Delta x}} w^{0, \pmb{\Delta t}}\\&+\frac{\sigma^2(1-\rho^2)\left((p_{|\mathcal{F}_{\pmb{i}-1}^0})^{\pmb{i}-1}_{\pmb{j}+1}-2(p_{|\mathcal{F}_{\pmb{i}-1}^0})^{{\pmb{i}}-1}_{\pmb{j}}+(p_{|\mathcal{F}_{\pmb{i}-1}^0})^{\pmb{i}-1}_{\pmb{j}-1}\right)}{2\pmb{\Delta \bar{x}}^2}\Biggr\}.
    \end{align}
    
    The forward method begins with the initial point $(p_{|\mathcal{F}_{\pmb{0}}^0})^{\pmb{0}}_{\pmb{j}}$ which follows a standard normal distribution $\mc{N} (0,1)$ restricted on the space generated by bank's state $(\theta,\theta+\pmb{\Delta x},\dots,\theta+\pmb{N_x \Delta x}]$. Remark that in order to satisfy the absorbing condition at the threshold, we consider $(p_{|\mathcal{F}_{\pmb{i}}^0})^{\pmb{i}}_{\theta} = 0$ for all ${\pmb{i}}$, representing condition $p(\theta,t | \mathcal{F}_t^0)=0$. Then, the probability $(p_{|\mathcal{F}_{\pmb{i}}^0})$ is incremented at each time and each bank's state step up to the end of the time horizon $T$.
\end{itemize}
        
Note that the forward explicit finite differences method represents a straightforward yet inherently unstable approach for discretizing and solving PDEs. This instability arises from its tendency to amplify small discretization errors as they propagate across the grid. Achieving reliable outcomes necessitates employing a finer grid. Especially, a more refined time discretization is crucial. Consequently, emphasizing the significance of adopting a meticulously crafted grid becomes paramount. Alternative numerical techniques, such as implicit finite differences or the alternating directions implicit method, may be useful for enhancing the stability and accuracy of implementations \citep{PMB2013}.

The default probability for the three cases is calculated by evaluating the incremented probability at time $T$ and employing the trapezoidal rule across the generated grid with respect to the relevant variables with the use of \texttt{numpy.trapz}. The trapezoidal rule is a numerical method to approximate the integral using left and right Riemann sums over the probability curve. The default probability is retrieved at the end by performing a subtraction as described in equations \eqref{last_step_MFFP}, \eqref{last_step_XFP} and \eqref{last_step_CXFP}.

\subsubsection{Results and Interpretations: Systemic and Individual Default Probability}\label{SIDP}

In this section, we conduct a comprehensive analysis on the impact of various parameters on systemic and individual default probabilities based on the outcomes generated by numerical methods. These tests specifically pertain to unconditional default probabilities defined in Section \ref{def: default}. The baseline scenario is defined by the parameter values $ \sigma = 0.3,\, \rho = 0.4,\, \xi = 1,\, q = 10,\, \gamma = 0.2,\, a = 2.5,\, b(t) = 1,$ for all t, $\hat{q} = 0,$ the default threshold $\theta = -0.7$, and the time horizon $T=0.25$. Remark that due to the requirement of the appropriate grid as mentioned in Section \ref{FPEPD_D}, finding the probability given an extended time horizon requires a finer grid and thus higher computational time. We choose this restrained time horizon to reduce the processing time. The results presented in this subsection and the next one have been carefully selected through rigorous testing of various parameters and grid settings, identifying the appropriate grid configuration for the given parameters, and thereby ensuring the attainment of stable outputs. We present three cases using the parameters of the baseline scenario in which we change one of the parameters for the numerical analysis unless otherwise mentioned.
\\
\begin{itemize}
    \item[Case 1] \textit{Impact of Correlation Coefficient $\rho$} 
    
    The magnitude of the shocks that affect both the reserve of the market and bank-$i,\, i\in \mfN$, is expressed by $\sigma$. For the market reserve, this magnitude is multiplied by a factor of $\rho$. For the bank-$i$ reserve, the magnitude is multiplied by $\rho$ for the common shocks and $\sqrt{1-\rho^2}$ for the idiosyncratic shocks. As $\rho$ increases, the impact of the common noise on the overall market increases, leading to a higher probability of the market default. While this effect on the market is present, the impact of the idiosyncratic shock on the bank decreases as the associated multiplier of this shock is $\sqrt{1-\rho^2}$. Thus, the common and the idiosyncratic shocks affect the probability of individual default simultaneously and differently. In addition, as demonstrated in Section \ref{sub: dynamics_infinite}, the correlation between banks is quantified by the factor $(\sigma \rho)^2$. On the one hand, as the parameter $\rho$ is increasing while maintaining other parameters constant, banks exhibit higher degrees of correlation among themselves. In consequence, the default of one individual bank will lead to a more probable market default. On the other hand, when $\rho$ is large, the bank is subject to a lower individual risk but a higher systemic risk. However, because of the strong correlation, this common market risk is shared more extensively among banks. The current question revolves around determining the extent of systemic risk that individual banks are exposed to after sharing. The key consideration is whether the benefits of risk sharing outweigh the challenges posed by a potentially more volatile market environment on banks. Based on the numerical analysis, we observe that given other parameters as in the baseline scenario but with $\sigma=0.2$, as $\rho$ increases, the probability of the systemic default increases while the bank's default probability decreases as shown in Figure \ref{fig: Sensitivity Test over rho for Unconditional Probability Default}. In this scenario, effective sharing of systemic risk among agents occurs when the correlation coefficient $\rho$ is high. Moreover, along with reduced individual risk, there is a decrease in the likelihood of individual default. Finally, we observe that a higher risk-aversion degree, i.e. $\tfrac{1}{\gamma} = 80$, reduces the systemic risk for any correlation strength among agents.\\ \\
    
    \begin{figure}[h]
        \centering
        \begin{adjustbox}{width=\textwidth}
            \includegraphics{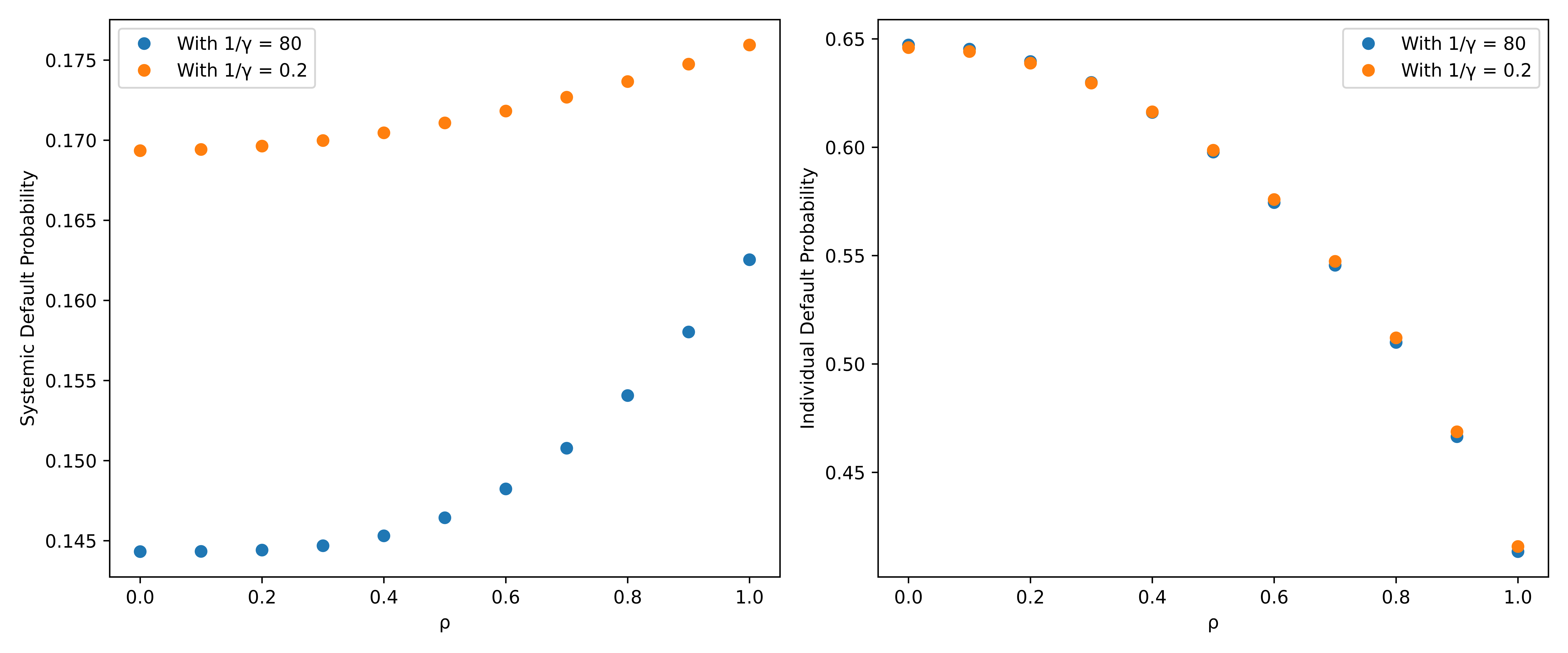}
        \end{adjustbox}
        \caption{The impact of correlation coefficient $\rho$ on individual and systemic default probabilities for two degrees of risk sensitivity $\tfrac{1}{\gamma} = 0.2,\, 80$, with the parameter values $ \sigma = 0.2,\, \rho = 0.4,\, \xi = 1,\, q = 10,\, \gamma = 0.2,\, a = 2.5,\, b(t) = 1,$ for all t, $\hat{q} = 0,$ the default threshold $\theta = -0.7$, and the time horizon $T=0.25$.}
        \label{fig: Sensitivity Test over rho for Unconditional Probability Default}
    \end{figure}
    \item[Case 2] \textit{Impact of Risk-Sensitivity Degree $\tfrac{1}{\gamma}$}
    
    The degree of risk sensitivity of a representative bank is expressed by $\tfrac{1}{\gamma}$. When $\tfrac{1}{\gamma} > 0$, the bank is risk-averse. In addition, the value of $\tfrac{1}{\gamma}$ expresses the magnitude of the risk aversion. Thus, a large $\tfrac{1}{\gamma}$ characterizes the behavior of the bank as excessively risk averse. As shown in Figure \ref{fig: Sensitivity Test over Inversed Gamma for Unconditional Probability Default} simulated from the baseline scenario but with changing $\tfrac{1}{\gamma}$, the probability of individual default diminishes when the bank exhibits a higher risk-aversion. As a result, for the market setup under study, where the banks share the same risk-aversion degree, the probability of systemic default follows a similar pattern and decreases by risk-aversion.
    
    \begin{figure}[h]
        \centering
        \begin{adjustbox}{width=\textwidth}
            \includegraphics{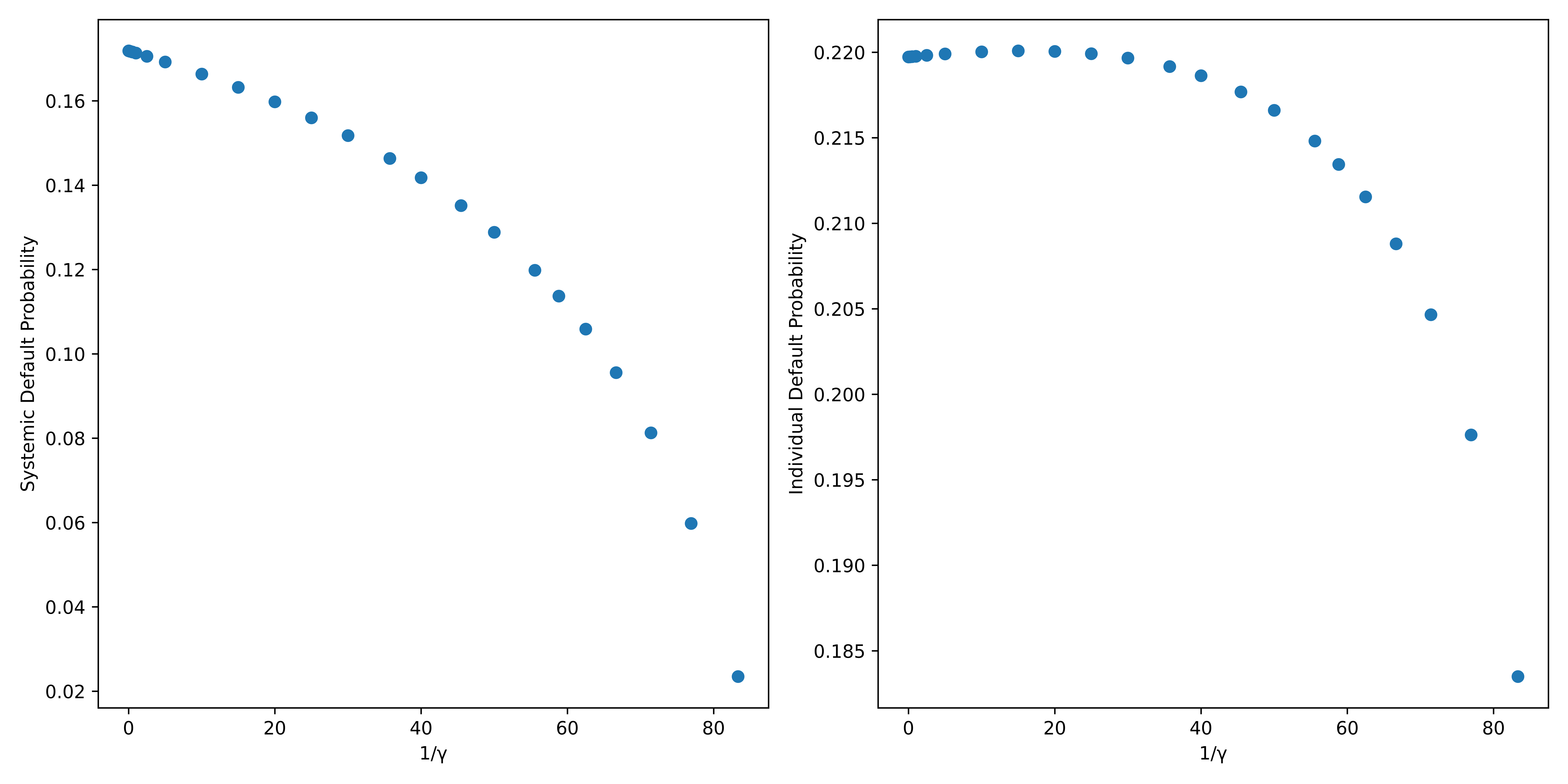}
        \end{adjustbox}
        \caption{The impact of risk-aversion degree $\tfrac{1}{\gamma}$ on individual and systemic default probabilities with the parameter values $ \sigma = 0.3,\, \rho = 0.4,\, \xi = 1,\, q = 10,\, \gamma = 0.2,\, a = 2.5,\, b(t) = 1,$ for all t, $\hat{q} = 0,$ the default threshold $\theta = -0.7$, and the time horizon $T=0.25$.}
        \label{fig: Sensitivity Test over Inversed Gamma for Unconditional Probability Default}
    \end{figure}

    \item[Case 3] \textit{Impact of Liquidity $b(t)$}
    
    We consider the case where the liquidity process, $b(t)=b$, is constant throughout time. As all banks are homogeneous, by increasing $b$, both the bank and the system enhance their liquidity positions, thereby reducing the level of risk they undertake. Conversely, reducing $b$ signifies a decrease in liquidity, introducing additional risk for both the bank and the market. From Figure \ref{fig: Sensitivity Test over b for Unconditional Systemic Probability Default} generated from the baseline scenario but with changing $b$, we observe that as $b$ increases, the probabilities of both the individual bank and the market state decrease. Furthermore, the effect of liquidity infusions on the systemic risk and individual default probability becomes more pronounced with a higher level of risk aversion (e.g. $\tfrac{1}{\gamma} = 80$) in the market.
    
    \begin{figure}[h]
        \centering
        \begin{adjustbox}{width=\textwidth}
            \includegraphics{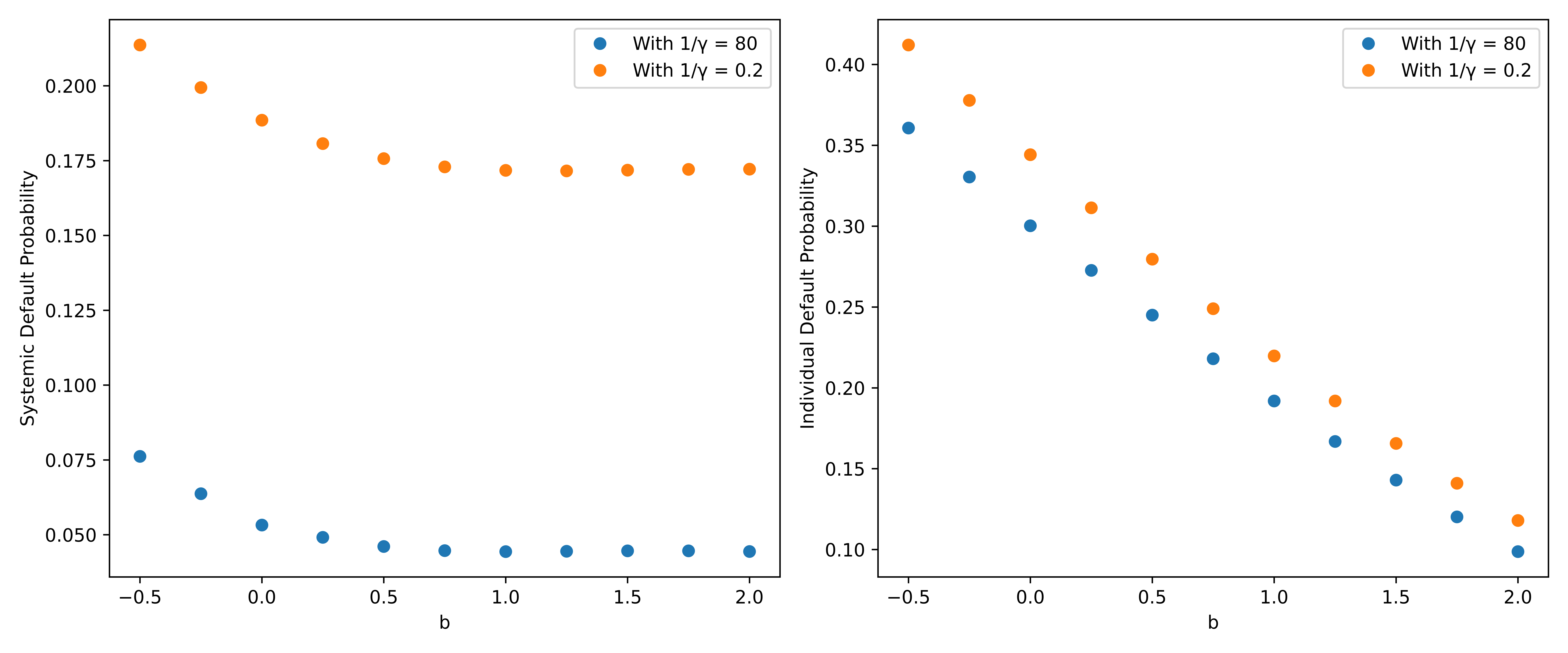}
        \end{adjustbox}
        \caption{The impact of liquidity parameter $b$ on the individual and systemic default probabilities for two degrees of risk sensitivity $\tfrac{1}{\gamma} = 0.2,\, 80$, with the parameter values $ \sigma = 0.3,\, \rho = 0.4,\, \xi = 1,\, q = 10,\, a = 2.5,\, b(t) = 1,$ for all t, $\hat{q} = 0,$ the default threshold $\theta = -0.7$, and the time horizon $T=0.25$.}
        \label{fig: Sensitivity Test over b for Unconditional Systemic Probability Default}
    \end{figure}
\end{itemize}

\subsubsection{Results and Interpretation: Conditional Default Probability under Specific Common Shocks}\label{CDF}

In this section, we analyze the conditional probability of default of a representative bank given specific trajectories of the common noise $(w^0_t)_{t\in \mfT}$ as defined in section \ref{def: default}. The baseline scenario is defined by the parameter values  $\bar{x}_0 = 0, \sigma = 1, \rho = 0.5, \xi = 1, q = 1, \gamma = 1, a = 1, b(t) = 1,$ for all t$, \hat{q} = 1$, the default threshold $\theta = -0.7$ and the time horizon $T=0.25$. We consider two trajectories for the common shock, respectively, denoted by $(\mathcal{P}^1)_{t\in \mfT}$ and $(\mathcal{P}^2)_{t\in \mfT}$. Under trajectory $\mathcal{P}^2$ the market state experiences a larger number of negative shocks compared to $\mathcal{P}^1$.

The equilibrium market state under trajectories $\mathcal{P}^1$ and $\mathcal{P}^2$ is depicted in Figure \ref{fig: Market State over Time}. From \eqref{last_step_CXFP}, the time evolution of the conditional density function of the bank $p(x^i,t | \mathcal{F}_t^0)$\, within the survival set $(\theta,\infty)$ is illustrated in  Figure \ref{fig: Bank PDF over Time under First Trajectory}. In other words, the figure depicts the conditional probability density of the bank that has not defaulted up to time $t$.
We observe that as time goes by, the respective cumulative distribution function decreases, indicating an increase in the conditional probability of default.
This observation is further demonstrated in Table \ref{table: Probability default over time} where we present the associated conditional probability of individual default under trajectory $\mathcal{P}^1$ over time.
We observe that for the baseline setting, the conditional probability of individual default escalates over the course of time. Furthermore, in Figure \ref{fig: Market State over Time}, we observe that a critical event happens around $t \in [0.05, 0.1]$, leading to the market state being closer to the default threshold. As the bank aims to track the market state, this negative impact is also translated into the bank's conditional probability of default. This event is demonstrated in Table \ref{table: Probability default over time}, where the conditional probability of the bank's default increases sharply around the same time.
\begin{figure}[h]
    \centering
    \begin{adjustbox}{width=0.8\textwidth}
        \includegraphics{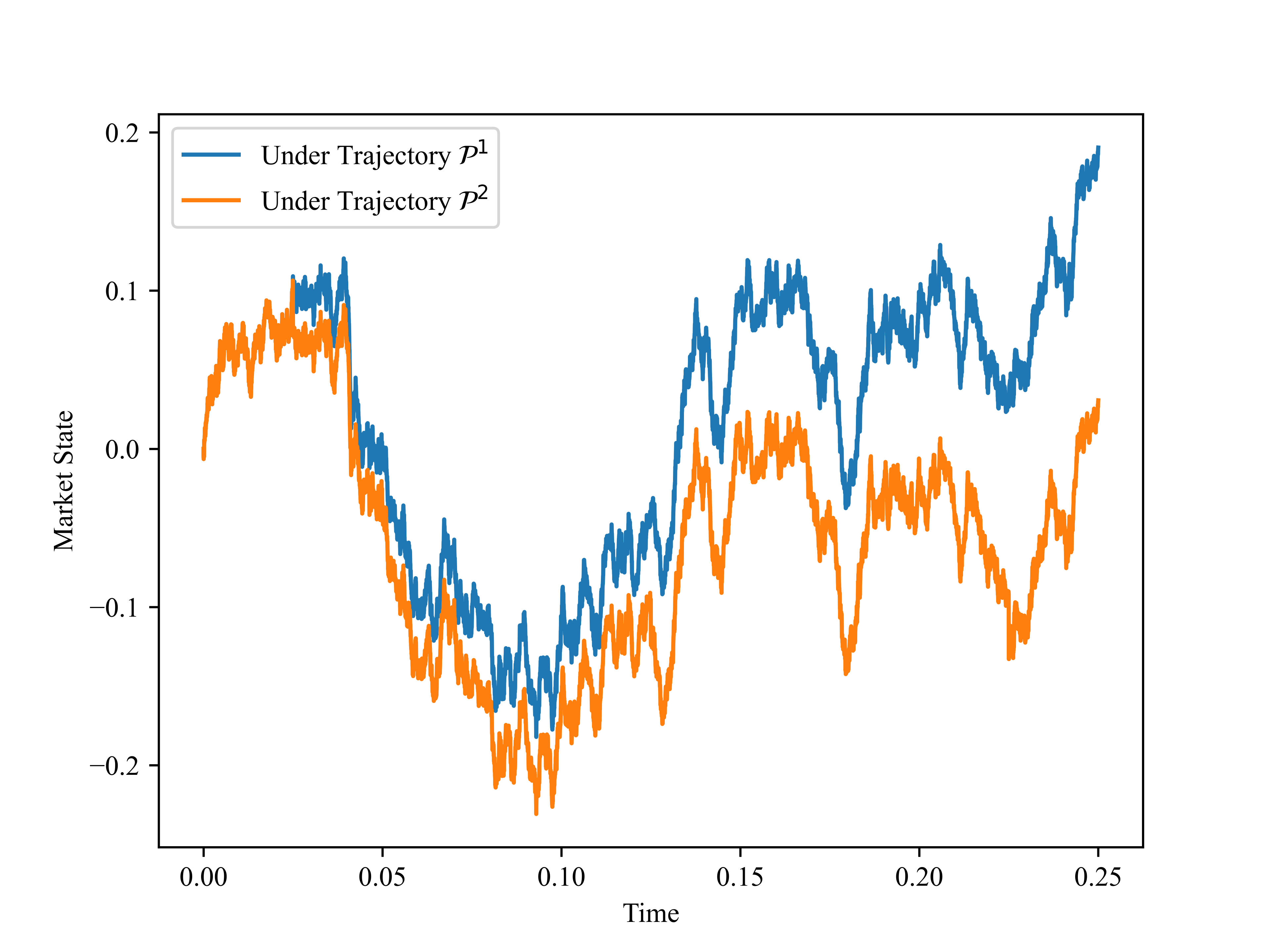}
    \end{adjustbox}
    \caption{Market state over time under the trajectories $\mathcal{P}^1$ and $\mathcal{P}^2$ described in Section \ref{CDF} with parameter values $\bar{x}_0 = 0, \sigma = 1, \rho = 0.5, \xi = 1, q = 1, \gamma = 1, a = 1, b(t) = 1,$ for all t$, \hat{q} = 1$, the default threshold $\theta = -0.7$.}
    \label{fig: Market State over Time}
\end{figure}
\begin{figure}[h]
    \centering
    \begin{adjustbox}{width=0.8\textwidth}
        \includegraphics{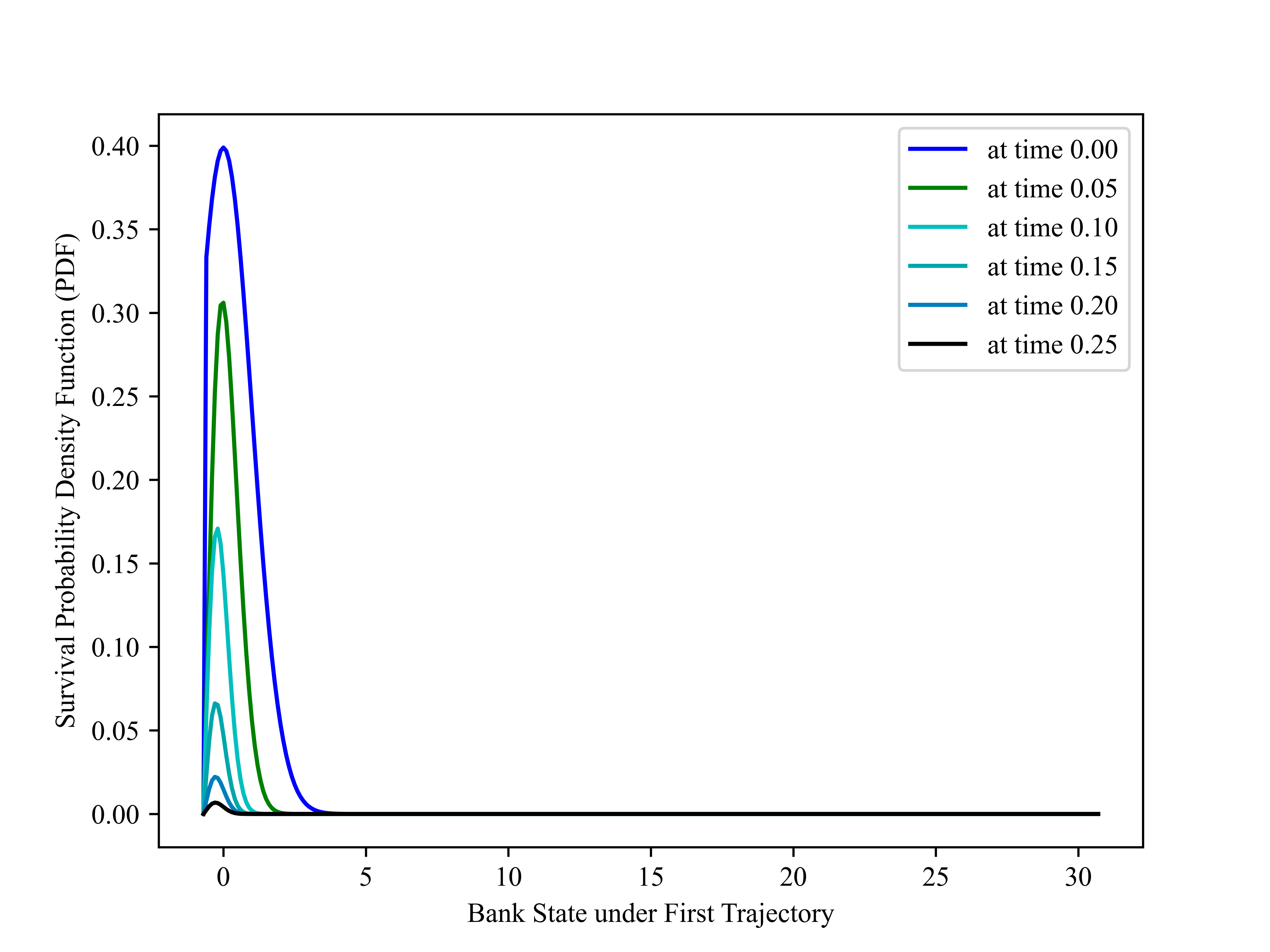}
    \end{adjustbox}
    \caption{Bank survival probability distribution over time under the trajectory $\mathcal{P}^1$ described in Section \ref{CDF} with parameter values $\bar{x}_0 = 0, \sigma = 1, \rho = 0.5, \xi = 1, q = 1, \gamma = 1, a = 1, b(t) = 1,$ for all t$, \hat{q} = 1$, the default threshold $\theta = -0.7$.}
    \label{fig: Bank PDF over Time under First Trajectory}
\end{figure}


Consider the economic environment under the common shock $\mathcal{P}^2$ characterized by a greater magnitude of negative shocks at certain times, for which the market state is depicted in \ref{fig: Market State over Time}. We observe that the market state under $\mathcal{P}^2$ moves more closely to the default threshold compared to $\mathcal{P}^1$, capturing the amplified negative shocks in the market. From \eqref{last_step_CXFP}, the respective time evolution of the conditional density function of the bank $p(x^i,t | \mathcal{F}_t^0)$ within the survival set $(\theta, \infty)$ is presented in Figure \ref{fig: Bank PDF over Time under Second Trajectory}. According to Table \ref{table: Probability default over time}, we observe that as the bank is experiencing more adversity under $\mathcal{P}^2$, the probability of default increases compared to $\mathcal{P}^1$. We also remark that in both cases, from Figure \ref{fig: Market State over Time}, the market has not defaulted.

\begin{figure}[h]
    \centering
    \begin{adjustbox}{width=0.8\textwidth}
        \includegraphics{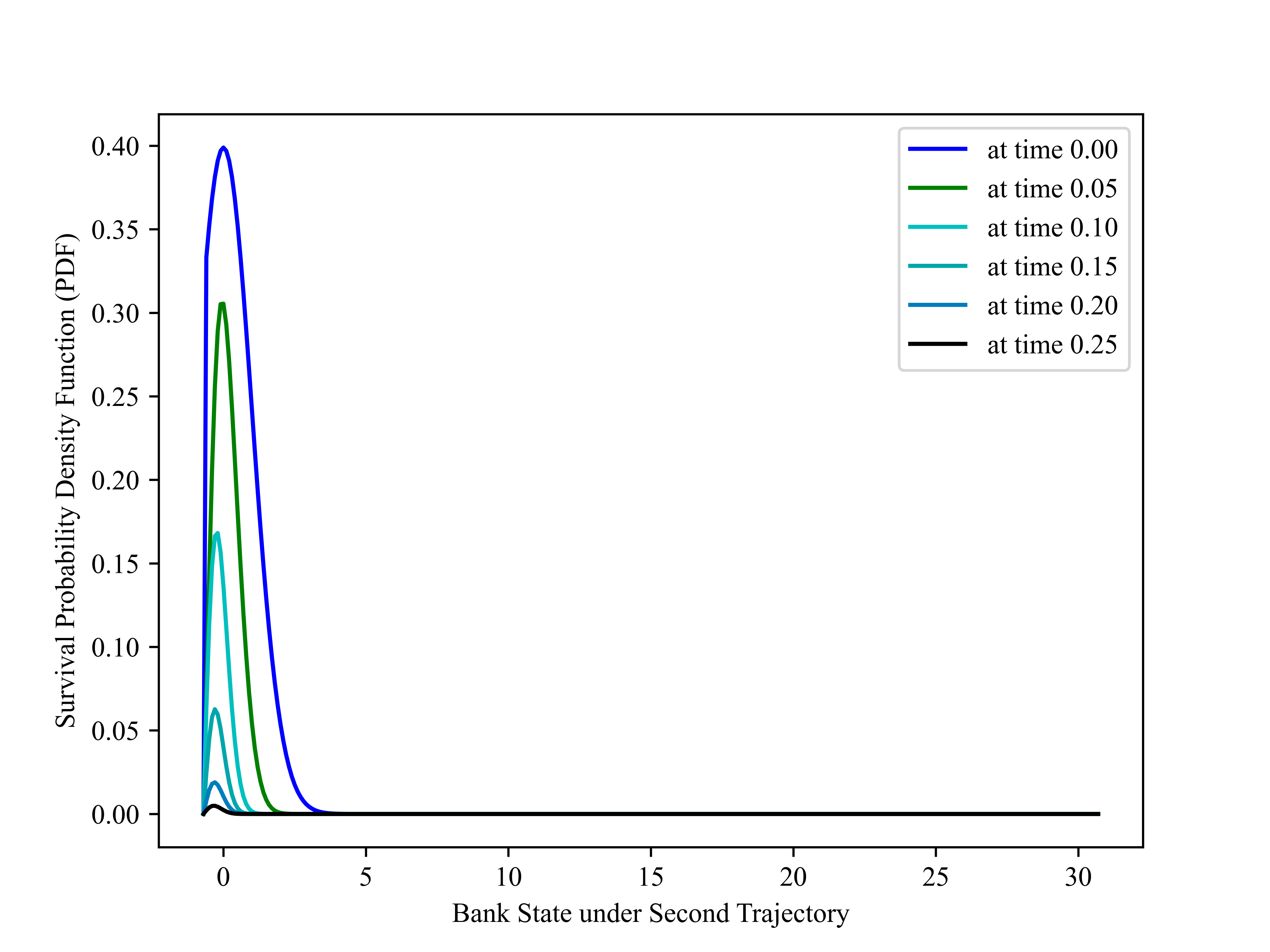}
    \end{adjustbox}
    \caption{Bank survival probability distribution over time under the trajectory $\mathcal{P}^2$ described in Section \ref{CDF} with parameter values $\bar{x}_0 = 0, \sigma = 1, \rho = 0.5, \xi = 1, q = 1, \gamma = 1, a = 1, b(t) = 1,$ for all t$, \hat{q} = 1$, the default threshold $\theta = -0.7$.}
    \label{fig: Bank PDF over Time under Second Trajectory}
\end{figure}

\begin{table}[h]
     \centering
     \begin{adjustbox}{width=0.7\textwidth}
         \begin{tabular}{|>{\centering\arraybackslash}p{1cm}|c|c|}
            \hline
            Time     & \begin{tabular}{@{}c@{}} Conditional Probability of \\ Individual Default  under $\mathcal{P}^1$\end{tabular}  & \begin{tabular}{@{}c@{}}Conditional Probability of \\ Individual Default  under $\mathcal{P}^2$\end{tabular} \\
            \hline
            0& 0.2578 & 0.2578 \\
            0.05&  0.6599&  0.6610\\
            0.1&  0.8634& 0.8675\\
            0.15& 0.9545& 0.9588\\
            0.2& 0.9854 & 0.9883\\
            0.25& 0.9957& 0.9971\\
            \hline
         \end{tabular}
     \end{adjustbox}
     \caption{Probability of individual default over time with parameter values $\bar{x}_0 = 0, \sigma = 1, \rho = 0.5, \xi = 1, q = 1, \gamma = 1, a = 1, b(t) = 1,$ for all t$, \hat{q} = 1$, the default threshold $\theta = -0.7$ subject to the trajectories $\mathcal{P}^1$ and $\mathcal{P}^2$ described in Section \ref{CDF}.} 
     \label{table: Probability default over time}
\end{table}

\section{Conclusion}
This paper delves into the exploration of LQG risk-sensitive MFGs where agents are influenced by a common noise in their dynamics and wish to minimize an exponential cost functional. We focus on a scenario where the number of agents approaches infinity. The optimal strategies of agents, leading to a Nash equilibrium for the system, admit a linear feedback representation in terms of the state and the mean field. Moreover, risk sensitivity degree, the covariance of the common shock and the covariance of the idiosyncratic shock explicitly affect the coefficients of the optimal strategy. 

Applying this framework, we extend our investigation to an interbank transaction context. Our study encompasses the analysis of individual and market default scenarios across various parameter settings. Furthermore, an examination of individual default is conducted under specific trajectories of the common market noise. Our investigation reveals insightful outcomes in the context of interbank transactions, where agents, in this case banks, exhibit homogeneity and correlation, as specified in Section \ref{sec: Application}. We observe that high correlation among these banks contributes to diminished probability of individual default due to the benefits of risk-sharing yet heightened market default probability as the default of one bank leads to a higher chance of the market default. Additionally, banks with lower risk aversion are prone to experience an elevated individual default risk. As a consequence in this homogeneous setting, the systemic risk increases as well. However, higher degrees of risk-aversion shared by all banks, improve the systemic risk. Moreover, introducing liquidity infusions within the institutions helps to mitigate systemic and individual default risks, a factor that becomes more influential in the presence of higher levels of risk aversion. Finally, upon investigating the conditional probability of an individual bank default under the influence of specific economic shocks, greater negative shocks exerted upon banks correspond to elevated probabilities of default. 

The significance of this research lies in its contribution to comprehending risk-sensitive decision-making amid the presence of common noise. Through our analysis, we provide insights that enhance the understanding of how agents' optimal strategies adapt to a dynamic environment characterized by risk aversion and interconnectedness. 

Future studies can build upon the presented LQG risk-sensitive MFG model with common noise, considering its limitations. Due to the variational approach taken for the analysis of the optimal control, the considered cost functional needs to be convex with respect to its variables. This characteristic is ensured by the imposed assumptions (i.e. Assumption \ref{ass:ConvCond} and the nonnegativity of $1/\gamma_k$ implying the risk aversion of the agents). Further research could be valuable in exploring conditions when agents exhibit risk-seeking behavior, that is, when $1/\gamma_k$ is negative. Additionally, it could be intriguing to investigate the existence of an approximate Nash ($\epsilon$-Nash) equilibrium in the finite-population system in subsequent research. 

\bibliographystyle{apalike}
\bibliography{main.bib}

\end{document}